\documentclass{article}
\usepackage[top=2cm,bottom=2cm,left=4cm, right=4cm]{geometry}
\usepackage{amsmath, amssymb, amsthm}
\usepackage{xcolor}
\usepackage{mathrsfs}
\usepackage{enumerate}
\usepackage{xcolor}
\usepackage{booktabs}
\usepackage{apacite}
\let\cite\shortcite
\let\citeA\shortciteA
\usepackage{graphicx}
\usepackage{algorithm}
\usepackage{algpseudocode}
\usepackage{authblk}

\newcommand{\com}[2]{\left(\begin{array}{c}#1 \\ #2\end{array}\right)}
\newcommand{\lb}{\left(}
\newcommand{\rb}{\right)}
\newcommand{\eps}{\epsilon}

\newcommand{\td}{\tilde}
\newcommand{\E}{\mathbb{E}}
\newcommand{\R}{\mathbb{R}}
\newcommand{\ep}{\epsilon}

\newcommand{\norm}[2]{\left\|#2\right\|_{#1}}
\newcommand{\lnorm}{|\!|}

\newcommand{\pd}[2]{\frac{\partial #1}{\partial #2}}

\newcommand{\kbigo}[2]{O_{L^{#1}}\lb #2\rb}
\DeclareMathOperator*{\argmin}{arg\,min}

\DeclareMathOperator*{\tr}{\ensuremath{tr}}
\DeclareMathOperator*{\Var}{\ensuremath{Var}}
\DeclareMathOperator*{\Cov}{\ensuremath{Cov}}

\DeclareMathOperator*{\diag}{\ensuremath{diag}}
\DeclareMathOperator*{\sign}{\ensuremath{sign}}
\DeclareMathOperator*{\tdist}{\ensuremath{t}}
\DeclareMathOperator*{\vect}{\ensuremath{vec}}
\DeclareMathOperator*{\spanvec}{\ensuremath{span}}
\DeclareMathOperator*{\rank}{\ensuremath{rank}}
\newcommand{\exps}[1]{\exp\left\{#1\right\}}
\newcommand{\polyLog}{\mathrm{polyLog(n)}}
\newcommand{\betanull}{\beta^{*}}
\newcommand{\lse}{\hat{\beta}^{LS}}
\newcommand{\mest}{\hat{\beta}}
\newcommand{\rmax}{\max_{i}}
\newcommand{\cmax}{\max_{j\in J_{n}}}
\newcommand{\cmin}{\min_{j\in J_{n}}}

\newcommand{\m}[1]{\mathbb{#1}}
\newcommand{\fd}{\frac{\partial\hat{\beta}}{\partial\epsilon^{T}}}
\newcommand{\Fd}[1]{\frac{\partial\hat{\beta}_{#1}}{\partial\epsilon^{T}}}
\newcommand{\lra}{\Longrightarrow}
\newcommand{\wm}{X^{T}DX}
\newcommand{\sd}[1]{\frac{\partial\hat{\beta}}{\partial\epsilon_{#1}\partial\epsilon^{T}}}
\newcommand{\Sd}[2]{\frac{\partial\hat{\beta}_{#2}}{\partial\epsilon_{#1}\partial\epsilon^{T}}}
\newcommand{\SD}[1]{\frac{\partial\hat{\beta}_{#1}}{\partial\epsilon\partial\epsilon^{T}}}

\newcommand{\wmj}{X_{[j]}^{T}D_{[j]}X_{[j]}}

\newcommand{\lammax}{\lambda_{+}}
\newcommand{\lammin}{\lambda_{-}}

\newcommand{\tendsto}{\rightarrow}

\definecolor{magenta}{RGB}{255,75,150}

\newcommand{\betaHat}{\hat{\beta}(\rho)}
\newcommand{\id}{\text{Id}}
\newcommand{\SOPI}{second-order Poincar\'e inequality }
\newcommand{\trsp}{^{T}}
\newcommand{\Rp}{\R^{p}}
\newcommand{\Rpminus}{\R^{p-1}}
\newcommand{\betaHatj}{\hat{\beta}_{j}(\rho)}
\usepackage{xr}
\usepackage[normalem]{ulem}

\newtheorem*{theorem*}{Theorem}
\newtheorem{theorem}{Theorem}[section]
\newtheorem{proposition}[theorem]{Proposition}
\newtheorem{lemma}[theorem]{Lemma}
\newtheorem{corollary}[theorem]{Corollary}
\newtheorem{remark}[theorem]{Remark}
\newtheorem{definition}[theorem]{Definition}

\allowdisplaybreaks
\begin{document}
\title{Asymptotics For High Dimensional Regression $M$-Estimates: Fixed Design Results}
\author{Lihua Lei\thanks{\textbf{contact}: lihua.lei@berkeley.edu. Support from Grant FRG DMS-1160319 is gratefully acknowledged.}, Peter J. Bickel\thanks{Support from Grant FRG DMS-1160319 is gratefully acknowledged.}, and Noureddine El Karoui\thanks{
    Support from Grant NSF DMS-1510172 is gratefully acknowledged. \\
\textit{AMS 2010 MSC:} Primary: 62J99, Secondary: 62E20;\\
\textbf{Keywords:} M-estimation, robust regression, high-dimensional statistics, second order Poincar\'e inequality, leave-one-out analysis.}}
\affil{Department of Statistics, University of California, Berkeley}
\date{\today}
\maketitle
\begin{abstract}
We investigate the asymptotic distributions of coordinates of regression M-estimates in the moderate $p/n$ regime, where the number of covariates $p$ grows proportionally with the sample size $n$. Under appropriate regularity conditions, we establish the coordinate-wise asymptotic normality of regression M-estimates assuming a fixed-design matrix. Our proof is based on the \SOPI \cite{sopi} and leave-one-out analysis \cite{elkaroui11}. Some relevant examples are indicated to show that our regularity conditions are satisfied by a broad class of design matrices. We also show a counterexample, namely the ANOVA-type design, to emphasize that the technical assumptions are not just artifacts of the proof. Finally, the numerical experiments confirm and complement our theoretical results.
\end{abstract}
\section{Introduction}\label{sec:intro}
High-dimensional statistics has a long history \cite{huber73, wachter76,wachter78} with considerable  renewed interest over the last two decades. In many applications, the researcher collects data which can be represented as a matrix, called a design matrix and denoted by $X\in\R^{n\times p}$, as well as a response vector $y\in \R^{n}$ and aims to study the connection between $X$ and $y$. The linear model is among the most popular models as a starting point of data analysis in various fields. A linear model assumes that 
\begin{equation}\label{eq:linearmodel}
y = X\betanull + \eps,
\end{equation}
where % $\textbf{1}$ is the $n\times 1$ vector with all entries 1, $\alpha^{*}\in \R^{1}$ is the intercept,
$\betanull\in\R^{p}$ is the coefficient vector which measures the marginal contribution of each predictor and $\eps$ is a random vector which captures the unobserved errors. % Usually, an intercept term is incorporated in which case $X$ contains a column with all entries equal to 1. 

The aim of this article is to provide valid inferential results for features of $\betanull$. For example, a researcher might be interested in testing whether a given predictor has a negligible effect on the response, or equivalently whether $\betanull_{j} = 0$ for some $j$.  Similarly, linear contrasts of $\betanull$ such as $\betanull_{1} - \betanull_{2}$ might be of interest in the case of the group comparison problem in which the first two predictors represent the same feature but are collected from two different groups.

An M-estimator, defined as 
\begin{equation}\label{eq:betaHat}
\betaHat = \argmin_{\beta\in\R^{p}}\frac{1}{n}\sum_{i=1}^{n}\rho(y_{i} - x_{i}^{T}\beta)
\end{equation}
where $\rho$ denotes a loss function, is among the most popular estimators used in practice \cite{relles68,huber73}. In particular, if $\rho(x) = \frac{1}{2}x^{2}$, $\betaHat$ is the famous Least Square Estimator (LSE). We intend to explore the distribution of $\betaHat$, based on which we can achieve the inferential goals mentioned above.

 The most well-studied approach is the asymptotic analysis, which assumes that the scale of the problem grows to infinity and use the limiting result as an approximation. In regression problems, the scale parameter of a problem is the sample size $n$ and the number of predictors $p$. The classical approach is to fix $p$ and let $n$ grow to infinity. It has been shown \cite{relles68, yohai72, huber72, huber73} that $\betaHat$ is consistent in terms of $L_{2}$ norm and asymptotically normal in this regime. The asymptotic variance can be then approximated by the bootstrap \cite{bickel81}. Later on, the studies are extended to the regime in which both $n$ and $p$ grow to infinity but $p / n$ converges to $0$ \cite{yohai79, portnoy84, portnoy85, portnoy86, portnoy87, mammen89}. The consistency, in terms of the $L_{2}$ norm, the asymptotic normality and the validity of the bootstrap still hold in this regime. Based on these results, we can construct a 95\% confidence interval for $\beta_{0j}$ simply as $\betaHatj\pm 1.96 \sqrt{\widehat{\Var}(\betaHatj)}$ where $\widehat{\Var}(\betaHatj)$ is calculated by bootstrap. Similarly we can calculate p-values for the hypothesis testing procedure.

We ask whether the inferential results developed under the low-dimensional assumptions and the software built on top of them can be relied on for moderate and high-dimensional analysis? Concretely, if in a study $n=50$ and $p=40$, can the software built upon the assumption that $p/n\simeq 0$ be relied on when $p/n=.8$? Results in random matrix theory \cite{mp67} already offer an answer in the negative side for many PCA-related questions in multivariate statistics. The case of regression is more subtle: For instance for least-squares, standard degrees of freedom adjustments effectively take care of many dimensionality-related problems. But this nice property does not extend to more general regression M-estimates. 

Once these questions are raised, it becomes very natural to analyze the behavior and performance of statistical methods in the regime where $p/n$ is fixed. Indeed, it will help us to keep track of the inherent statistical difficulty of the problem when assessing the variability of our estimates. In other words, we assume in the current paper that $p / n\rightarrow \kappa > 0$ while let $n$ grows to infinity. Due to identifiability issues, it is impossible to make inference on $\betanull$ if $p > n$ without further structural or distributional assumptions. We discuss this point in details in Section \ref{subsec:Identifiability}. Thus we consider the regime where $p / n\rightarrow \kappa \in (0, 1)$. We call it the moderate $p/n$ regime. This regime is  also the natural regime in random matrix theory \cite{mp67,wachter78,johnstone01,bai10}. It has been shown that the asymptotic results derived in this regime sometimes provide an extremely accurate approximation\sout{s} to finite sample distributions of estimators at least in certain cases \cite{johnstone01} where $n$ and $p$ are both small.

\subsection{Qualitatively Different Behavior of Moderate $p/n$ Regime}\label{subsec:qualitative}
First, $\betaHat$ is no longer consistent in terms of $L_{2}$ norm and the risk $\E \|\betaHat - \betanull\|^{2}$ tends to a non-vanishing quantity determined by $\kappa$, the loss function $\rho$ and the error distribution through a complicated system of non-linear equations \cite{elkaroui11, elkaroui13, elkaroui15, bean11}. This $L_2$-inconsistency prohibits the use of standard perturbation-analytic techniques to assess the behavior of the estimator.   It also leads to qualitatively different behaviors for the residuals in moderate dimensions; in contrast to the low-dimensional case, they cannot be relied on to give accurate information about the distribution of the errors. However, this seemingly negative result does not exclude the possibility of inference since $\betaHat$ is still consistent in terms of $L_{2+\nu}$ norms for any $\nu > 0$ and in particular in $L_{\infty}$ norm. Thus, we can  at least hope to perform inference on each coordinate.

Second, classical optimality results do not hold in this regime. In the regime $p / n\rightarrow 0$, the maximum likelihood estimator is shown to be optimal \cite{huber64,huber72,bickel15}. In other words, if the error distribution is known then the M-estimator associated with the loss $\rho(\cdot) = -\log f_{\eps}(\cdot)$ is asymptotically efficient, provided the design is of appropriate type,  where $f_{\eps}(\cdot)$ is the density of entries of $\eps$. However, in the moderate $p/n$ regime, it has been shown that the optimal loss is no longer the log-likehood but an other function with a complicated but explicit form \cite{bean13}, at least for certain designs. The suboptimality of maximum likelihood estimators suggests that classical techniques fail to provide valid intuition in the moderate $p/n$ regime.

Third, the joint asymptotic normality of $\betaHat$, as a $p$-dimensional random vector, may be violated for a fixed design matrix $X$. This has been proved for least-squares by \citeA{huber73} in his pioneering work. For general M-estimators, this negative result is a simple consequence of the results of \citeA {elkaroui11}: They exhibit an ANOVA design (see below) where even marginal fluctuations are not Gaussian. By contrast, for random design, they show that $\betaHat$ is jointly asymptotically normal when the design matrix is elliptical with general covariance by using the non-asymptotic stochastic representation for $\betaHat$ as well as elementary properties of vectors uniformly distributed on the uniform sphere in $\mathbb{R}^p$; See section 2.2.3 of \citeA{elkaroui11} or the supplementary material of \citeA{bean13} for details. This does not contradict \citeA{huber73}'s negative result in that it takes the randomness from both $X$ and $\eps$ into account while \citeA{huber73}'s result only takes the randomness from $\eps$ into account. Later, \citeA{elkaroui15} shows that each coordinate of $\betaHat$ is asymptotically normal for a broader class of random designs. This is also an elementary consequence of the analysis in \citeA{elkaroui13}. However, to the best of our knowledge, beyond the ANOVA situation mentioned above, there are no distributional results for fixed design matrices. This is the topic of this article.

Last but not least, bootstrap inference fails in this moderate-dimensional regime. This has been shown by \citeA{bickel83} for least-squares and residual bootstrap in their influential work. %The failure is due to the same reason as mentioned above. 
Recently, \citeA{elkaroui15boot} studied the results to general M-estimators and showed that all commonly used bootstrapping schemes, including pairs-bootstrap, residual bootstrap and jackknife, fail to provide a consistent variance estimator and hence valid inferential statements. These latter results even apply to the marginal distributions of the coordinates of $\betaHat$. Moreover, there is no simple, design independent, modification to achieve consistency \cite{elkaroui15boot}. 

\subsection{Our Contributions}
In summary, the behavior of the estimators we consider in this paper is completely different in the moderate $p/n$ regime from its counterpart in the low-dimensional regime. As discussed in the next section, moving one step further in the moderate $p/n$ regime is interesting from both the practical and theoretical perspectives. The main contribution of this article is to establish coordinate-wise asymptotic normality of $\betaHat$ for certain \emph{fixed design matrices} $X$ in this regime under technical assumptions. The following theorem informally states our main result.

\begin{theorem*}[Informal Version of Theorem \ref{thm:main} in Section \ref{sec:main}]
  Under appropriate conditions on the design matrix $X$, the distribution of $\eps$ and the loss function $\rho$, as $p/n\tendsto \kappa \in (0,1)$, while $n\tendsto \infty$,
\[\max_{1\le j\le p}d_{\mathrm{TV}}\lb\mathcal{L}\lb\frac{\betaHatj - \E \betaHatj}{\sqrt{\Var(\betaHatj)}}\rb, N(0, 1)\rb = o(1)\]
where $d_{\mathrm{TV}}(\cdot, \cdot)$ is the total variation distance and $\mathcal{L}(\cdot)$ denotes the law.
\end{theorem*}

It is worth mentioning that the above result can be extended to finite dimensional linear contrasts of $\hat{\beta}$. % In other words, given a sequence of vectors $\{a_{n}: a_{n}\in \R^{p}\}$ such that the number of non-zero entries of each $a_{n}$ is uniformly bounded, $a_{n}^{T}\hat{\beta}$ is asymptotically normal. 
For instance, one might be interested in making inference on $\betanull_{1} - \betanull_{2}$ in the problems involving the group comparison. The above result can be extended to give the asymptotic normality of $\hat{\beta}_{1} - \hat{\beta}_{2}$.

Besides the main result, we have several other contributions. First, we use a new approach to establish asymptotic normality. Our main technique is based on the \SOPI (SOPI), developed by \citeA{sopi} to derive, among many other results, the fluctuation behavior of linear spectral statistics of random matrices. In contrast to classical approaches such as the Lindeberg-Feller central limit theorem, the \SOPI is capable of dealing with nonlinear and potentially implicit functions of independent random variables. Moreover, we use different expansions for $\betaHat$ and residuals based on double leave-one-out ideas introduced in \citeA{elkaroui11}, in contrast to the classical perturbation-analytic expansions. See aforementioned paper and follow-ups. An informal interpretation of the results of \citeA{sopi} is that if the Hessian of the nonlinear function of random variables under consideration is sufficiently small, this function acts almost linearly and hence a standard central limit theorem holds. 

Second, to the best of our knowledge this is the first inferential result for fixed  (non ANOVA-like) design in the moderate $p/n$ regime. Fixed designs arise naturally from an experimental design or a conditional inference perspective. That is, inference is ideally carried out without assuming randomness in predictors; see Section \ref{subsec:random_fix} for more details. We clarify the regularity conditions for coordinate-wise asymptotic normality of $\betaHat$ explicitly, which are checkable for LSE and also checkable for general M-estimators if the error distribution is known. We also prove that these conditions are satisfied with by a broad class of designs. 

The ANOVA-like design described in Section \ref{subsubsec:counterexample} exhibits a situation where the distribution of $\betaHatj$ is not going to be asymptotically normal. As such the results of Theorem \ref{thm:main} below are somewhat surprising.

For complete inference, we need both the asymptotic normality and the asymptotic bias and variance. Under suitable symmetry conditions on the loss function and the error distribution, it can be shown that $\betaHat$ is unbiased (see Section \ref{subsubsec:DisAssump} for details) and thus it is left to derive the asymptotic variance. As discussed at the end of Section \ref{subsec:qualitative}, classical approaches, e.g. bootstrap, fail in this regime. For least-squares, classical results continue to hold and we discuss it in section \ref{sec:lse} for the sake of completeness. However, for M-estimators, there is no closed-form result. We briefly touch upon the variance estimation in Section \ref{subsubsec:FurtherDis}. The derivation for general situations is beyond the scope of this paper and left to the future research. 

\subsection{Outline of Paper}
The rest of the paper is organized as follows: In Section \ref{sec:long_intro}, we clarify details which are mentioned in the current section. In Section \ref{sec:main}, we state the main result (Theorem \ref{thm:main}) formally and explain the technical assumptions. Then we show several examples of random designs which satisfy the assumptions with high probability. In Section 4, we introduce our main technical tool, \SOPI \cite{sopi}, and apply it on M-estimators as the first step to prove Theorem \ref{thm:main}. Since the rest of the proof of Theorem \ref{thm:main} is complicated and lengthy, we illustrate the main ideas in Appendix \ref{app:heuristic}. The rigorous proof is left to Appendix \ref{app:main}. In Section \ref{sec:lse}, we provide reminders about the theory of least-squares estimation for the sake of completeness, by taking advantage of its explicit form. In Section \ref{sec:numerical}, we display the numerical results. The proof of other results are stated in Appendix \ref{app:others} and more numerical experiments are presented in Appendix \ref{app:numerical}. 

\section{More Details on Background}\label{sec:long_intro}
\subsection{Moderate $p/n$ Regime: a more informative type of asymptotics?}
In Section \ref{sec:intro}, we mentioned that the ratio $p/n$ measures the difficulty of statistical inference. The moderate $p/n$ regime provides an approximation of finite sample properties with the difficulties fixed at the same level as the original problem. Intuitively, this regime should capture more variation in finite sample problems and provide a more accurate approximation. We will illustrate this via simulation. 

Consider a study involving 50 participants and $40$ variables; we can either use the asymptotics in which $p$ is fixed to be $40$, $n$ grows to infinity or $p / n$ is fixed to be $0.8$, and $n$ grows to infinity to perform approximate inference. Current software rely on low-dimensional asymptotics for inferential tasks, but there is no evidence that they yield more accurate inferential statements than the ones we would have obtained using moderate dimensional asymptotics. In fact, numerical evidence \cite{johnstone01,elkaroui13pnas,bean13} show that the reverse is true.

We exhibit a further numerical simulation showing that. Consider a case that $n = 50$, $\eps$ has i.i.d. entries  and $X$ is one realization of a matrix generated  with i.i.d. gaussian (mean 0, variance 1) entries. For $\kappa \in \{0.1, 0.2, \ldots, 0.9\}$ and different error distributions, we use the Kolmogorov-Smirnov (KS) statistics to quantify the distance between the finite sample distribution and two types of asymptotic approximation of the distribution of $\hat{\beta}_{1}(\rho)$.  

Specifically, we use the Huber loss function $\rho_{\mathrm{Huber},k}$ with default parameter $k = 1.345$ \cite{huber11}, i.e.
\[\rho_{\mathrm{Huber},k}(x) = \left\{
  \begin{array}{ll}
    \frac{1}{2}x^{2} & |x|\le k\\
    k(|x| - \frac{1}{2}k) & |x| > k
  \end{array}
\right.\]
Specifically, we generate three design matrices $X^{(0)}$, $X^{(1)}$ and $X^{(2)}$: $X^{(0)}$ for small sample case with a sample size $n = 50$ and a dimension $p = n\kappa$; $X^{(1)}$ for low-dimensional asymptotics ($p$ fixed) with a sample size $n = 1000$ and a dimension $p = 50\kappa$; and $X^{(2)}$ for moderate-dimensional asymptotics ($p / n$ fixed) with a sample size $n = 1000$ and a dimension $p = n\kappa$. Each of them is generated as one realization of an i.i.d. standard gaussian design and then treated as fixed across $K=100$ repetitions. For each design matrix, vectors $\eps$ of appropriate length are generated with i.i.d. entries. The entry has either a standard normal distribution, or a $t_3$-distribution, or a standard Cauchy distribution, i.e. $t_1$. Then we use $\eps$ as the response, or equivalently assume $\betanull = 0$, and obtain the M-estimators $\hat{\beta}^{(0)}, \hat{\beta}^{(1)}, \hat{\beta}^{(2)}$. Repeating this procedure for $K = 100$ times results in $K$ replications in three cases. Then we extract the first coordinate of each estimator, denoted by $\{\hat{\beta}^{(0)}_{k, 1}\}_{k=1}^{K}, \{\hat{\beta}^{(1)}_{k, 1}\}_{k=1}^{K}, \{\hat{\beta}^{(2)}_{k, 1}\}_{k=1}^{K}$. Then the two-sample Kolmogorov-Smirnov statistics can be obtained by 
\[\mathrm{KS}_{1} = \sqrt{\frac{n}{2}}\max_{x}|\hat{F}_{n}^{(0)}(x) - \hat{F}_{n}^{(1)}(x)|, \quad \mathrm{KS}_{2} = \sqrt{\frac{n}{2}}\max_{x}|\hat{F}_{n}^{(0)}(x) - \hat{F}_{n}^{(2)}(x)|,\]
where $\hat{F}_{n}^{(r)}$ is the empirical distribution of $\{\hat{\beta}^{(r)}_{k, 1}\}_{k=1}^{K}$. We can then compare the accuracy of two asymptotic regimes by comparing $\mathrm{KS}_{1}$ and $\mathrm{KS}_{2}$. The smaller the value of $\mathrm{KS}_i$, the better the approximation. 

Figure \ref{fig:asym_toy} displays the results for these error distributions. We see that for gaussian errors and even $\tdist_3$ errors, the $p/n$-fixed/moderate-dimensional approximation is uniformly more accurate than the widely used $p$-fixed/low-dimensional approximation. For Cauchy errors, the low-dimensional approximation performs better than the moderate-dimensional one when $p / n$ is small but worsens when the ratio is large especially when $p / n$ is close to 1. Moreover, when $p / n$ grows, the two approximations have qualitatively different behaviors: the $p$-fixed approximation becomes less and less accurate while the $p / n$-fixed approximation does not suffer much deterioration when $p / n$ grows. The qualitative and quantitative differences of these two approximations reveal the practical importance of exploring the $p / n$-fixed asymptotic regime. (See also \citeA{johnstone01}.)

\begin{figure}
  \centering
  \includegraphics[width = 5.5in, height = 2.2in]{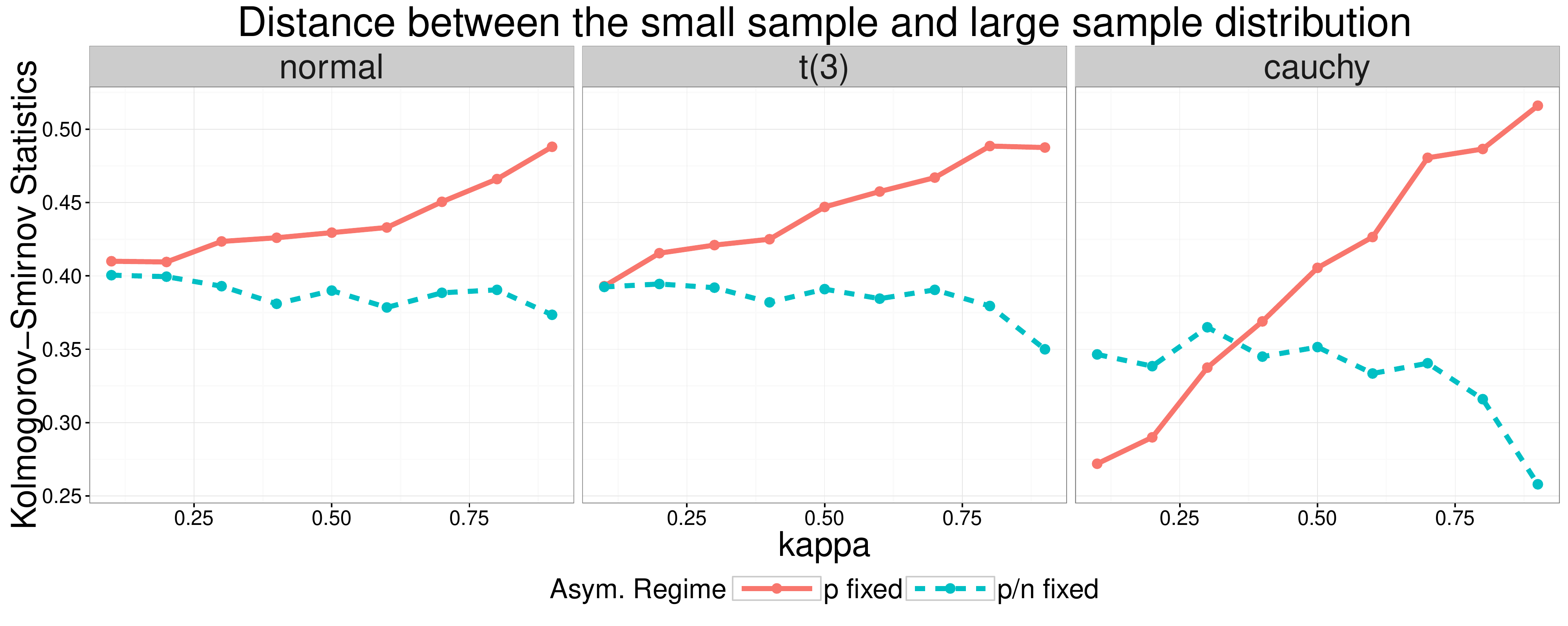}
  \caption{Axpproximation accuracy of $p$-fixed asymptotics and $p / n$-fixed asymptotics: each column represents an error distribution; the x-axis represents the ratio $\kappa$ of the dimension and the sample size and the y-axis represents the Kolmogorov-Smirnov statistic; the red solid line corresponds to $p$-fixed approximation and the blue dashed line corresponds to $p / n$-fixed approximation.}\label{fig:asym_toy}
\end{figure}

\subsection{Random vs fixed design?}\label{subsec:random_fix}
As discussed in Section \ref{sec:intro}.1, assuming a fixed design or a random design could lead to qualitatively different inferential results. % We should emphasize that the type of design generally only reflects the researchers' perspectives of data generating procedure but may not correspond to any observable phenomenon since only one design matrix $X$ is observed.
% \bluetext{Peter asked what this sentence means.} \LL{In practice, one is given a design matrix and it is up to the researcher to assume an underlying data generating process but the assumption is not verifiable because only one design matrix is observed.}

In the random design setting, $X$ is considered as being generated from a super population. For example, the rows of $X$ can be regarded as an i.i.d. sample from a distribution known, or partially known, to the researcher. In situations where one uses techniques such as cross-validation \cite{StoneMCV1974}, pairs bootstrap in regression \cite{EfronBook82} or sample splitting \cite{wasserman09}, the researcher effectively assumes exchangeability of the data $(x_i^{T},y_i)_{i=1}^n$. Naturally, this is only compatible with an assumption of random design. Given the extremely widespread use of these techniques in contemporary machine learning and statistics, one could argue that the random design setting is the one under which most of modern statistics is carried out, especially for prediction problems. Furthermore, working under a random design assumption forces the researcher to take into account two sources of randomness as opposed to only one in the fixed design case. Hence working under a random design assumption should yield conservative confidence intervals for $\betanull_j$.

In other words, in settings where the researcher collects data without control over the values of the predictors, the random design assumption is arguably the more natural one of the two. 

However, it has now been understood for almost a decade that common random design assumptions in high-dimension (e.g. $x_i=\Sigma^{1/2} z_i$ where $z_{i,j}$'s are  i.i.d with mean 0 and variance 1 and a few moments and $\Sigma$ ``well behaved") suffer from considerable geometric limitations, which have substantial impacts on the performance of the estimators considered in this paper \cite{elkaroui11}. As such, confidence statements derived from that kind of analysis can be relied on only after performing a few graphical tests on the data (see \citeA{nekMarkoRiskPub2010}). These geometric limitations are simple consequences of the concentration of measure phenomenon \cite{ledoux01}. 

On the other hand, in the fixed design setting, $X$ is considered a fixed matrix. In this case, the inference only takes the randomness of $\eps$ into consideration. This perspective is popular in several situations. The first one is the experimental design. The goal is to study the effect of a set of factors, which can be controlled by the experimenter, on the response. In contrast to the observational study, the experimenter can design the experimental condition ahead of time based on the inference target. For instance, a one-way ANOVA design encodes the covariates into binary variables (see Section \ref{subsubsec:counterexample} for details) and it is fixed prior to the experiment. Other examples include two-way ANOVA designs, factorial designs, Latin-square designs, etc. \cite{scheffe99}.

Another situation which is concerned with fixed design is the survey sampling where the inference is carried out conditioning on the data \cite{cochran77}. Generally, in order to avoid unrealistic assumptions, making inference conditioning on the design matrix $X$ is necessary. Suppose the linear model (\ref{eq:linearmodel}) is true and identifiable (see Section \ref{subsec:Identifiability} for details), then all information of $\betanull$ is contained in the conditional distribution $\mathcal{L}(y | X)$ and hence the information in the marginal distribution $\mathcal{L}(X)$ is redundant. The conditional inference framework is more robust to the data generating procedure due to the irrelevance of $\mathcal{L}(X)$. 

Also, results based on fixed design assumptions may be preferable from a theoretical point of view in the sense that they could potentially be used to establish corresponding results for certain classes of random designs. Specifically, given a marginal distribution $\mathcal{L}(X)$, one only has to prove that $\mathcal{X}$ satisfies the assumptions for fixed design with high probability. 

% In the low-dimensional setting, there is essentially no conceptual and technical difference between working under random and fixed design assumptions. \bluetext{Peter suggested forwarding some references here.}

In conclusion, fixed and random design assumptions play complementary roles in moderate-dimensional settings. We focus on the least understood of the two, the fixed design case, in this paper.

\subsection{Modeling and Identification of Parameters}\label{subsec:Identifiability}

\noindent  The problem of identifiability is especially important in the fixed design case. Define $\beta^{*}(\rho)$ in the population as 
\begin{equation}\label{eq:id}
\beta^{*}(\rho) = \argmin_{\beta\in \R^{p}}\frac{1}{n}\sum_{i=1}^{n}\E\rho(y_{i} - x_{i}^{T}\beta).
\end{equation}
One may ask whether $\beta^{*}(\rho) = \beta^{*}$ regardless of $\rho$ in the fixed design case. We provide an affirmative answer in the following proposition by assuming that $\eps_{i}$ has a symmetric distribution around $0$ and $\rho$ is even.

\begin{proposition}\label{prop:mest_id_fix_sym}
  Suppose $X$ has a full column rank and $\eps_{i}\stackrel{d}{=}-\eps_{i}$ for all $i$. Further assume $\rho$ is an even convex function such that for any $i = 1, 2, \ldots$ and $\alpha \not = 0$,
\begin{equation}\label{eq:symmetric_rho}
\frac{1}{2}\lb\E \rho(\eps_{i} - \alpha) + \E\rho(\eps_{i} + \alpha)\rb > \E \rho(\eps_{i}).
\end{equation}
Then $\beta^{*}(\rho) = \betanull$ regardless of the choice of $\rho$.
\end{proposition}

The proof is left to Appendix \ref{app:others}. It is worth mentioning that Proposition \ref{prop:mest_id_fix_sym} only requires the marginals of $\eps$ to be symmetric but does not impose any constraint on the dependence structure of $\eps$. Further, if $\rho$ is strongly convex, then for all $\alpha \not = 0$,
\[\frac{1}{2}\lb\rho(x - \alpha) + \rho(x + \alpha)\rb > \rho(x).\]
As a consequence, the condition \eqref{eq:symmetric_rho} is satisfied provided that $\eps_{i}$ is non-zero with positive probability. 

If $\eps$ is asymmetric, we may still be able to identify $\betanull$ if $\eps_{i}$ are i.i.d. random variables. In contrast to the last case, we should incorporate an intercept term as a shift towards the centroid of $\rho$. More precisely, we define $\alpha^{*}(\rho)$ and $\beta^{*}(\rho)$ as 
\[(\alpha^{*}(\rho), \beta^{*}(\rho)) = \argmin_{\alpha\in \R, \beta\in \R^{p}}\frac{1}{n}\sum_{i=1}^{n}\E \rho(y_{i} - \alpha - x_{i}^{T}\beta).\]

\begin{proposition}\label{prop:mest_id_fix}
Suppose $(\textbf{1} ,  X)$ is of full column rank and $\eps_{i}$ are i.i.d. such that $\E \rho(\eps_{1} - \alpha)$ as a function of $\alpha$ has a unique minimizer $\alpha(\rho)$. Then $\beta^{*}(\rho)$ is uniquely defined with $\beta^{*}(\rho) = \beta^{*}$ and $\alpha^{*}(\rho) = \alpha(\rho)$.
\end{proposition}
The proof is left to Appendix \ref{app:others}. For example, let $\rho(z) = |z|$. Then the minimizer of $\E \rho(\eps_{1} - a)$ is a median of $\eps_{1}$, and is unique if $\eps_{1}$ has a positive density. It is worth pointing out that incorporating an intercept term is essential for identifying $\betanull$. For instance, in the least-square case, $\beta^{*}(\rho)$ no longer equals to $\betanull$ if $\E\eps_{i}\not= 0$. Proposition \ref{prop:mest_id_fix} entails that the intercept term guarantees $\beta^{*}(\rho) = \betanull$, although the intercept term itself depends on the choice of $\rho$ unless more conditions are imposed. 

If $\eps_{i}$'s are neither symmetric nor i.i.d., then $\betanull$ cannot be identified by the previous criteria because $\betanull(\rho)$ depends on $\rho$. Nonetheless, from a modeling perspective, it is popular and reasonable to assume that $\eps_{i}$'s are symmetric or i.i.d. in many situations. Therefore, Proposition \ref{prop:mest_id_fix_sym} and Proposition \ref{prop:mest_id_fix} justify the use of M-estimators in those cases and M-estimators derived from different loss functions can be compared because they are estimating the same parameter.

\section{Main Results}\label{sec:main}

\subsection{Notation and Assumptions}\label{subsec:notation}
% \orangetext{To simplify notation, we write $X$ for $(X ,  \textbf{1})$ and $\betanull$ for $((\betanull)^{T} ,  0)^{T}$ when no confusion can arise. It is worth pointing out that our Theorem \ref{thm:main} holds for general design matrices without an intercept term. For convenience, we write the dimension of $\betanull$ as $p$ (instead of $p + 1$) to handle the general case}. \LL{A bit strange? My hope is to incorporate intercept term as an default. However, this will limit our examples unless we can show the bound for smallest eigenvalue for centered design matrix. } 
Let $x_{i}^{T}\in \R^{1\times p}$ denote the $i$-th row of $X$ and $X_{j}\in \R^{n\times 1}$ denote the $j$-th column of X. Throughout the paper we will denote by $X_{ij}\in \R$ the $(i, j)$-th entry of $X$, % by $X_{(i)}\in \R^{(n-1)\times p}$ the design matrix $X$ after removing the $i$-th row, 
by $X_{[j]}\in \R^{n\times (p-1)}$ the design matrix $X$ after removing the $j$-th column, % by $X_{(i), [j]}\in \R^{(n-1)\times (p-1)}$ the design matrix after removing both $i$-th row and $j$-th column, 
and by $x_{i, [j]}^{T}\in \R^{1\times (p-1)}$ the vector $x_{i}^{T}$ after removing $j$-th entry. The M-estimator $\betaHat$ associated with the loss function $\rho$ is defined as
\begin{equation}\label{eq:betaHat_general}
\betaHat = \argmin_{\beta\in\Rp}\frac{1}{n}\sum_{k=1}^{n}\rho(y_{k} - x_{k}^{T}\beta) = \argmin_{\beta\in\Rp}\frac{1}{n}\sum_{k=1}^{n}\rho(\eps_{k} - x_{k}^{T}(\beta - \betanull)) 
\end{equation}
We define $\psi = \rho'$ to be the first derivative of $\rho$. We will write $\betaHat$ simply $\hat{\beta}$ when no confusion can arise.

When the original design matrix $X$ does not contain an intercept term, we can simply replace $X$ by $(\textbf{1} ,  X)$ and augment $\beta$ into a $(p + 1)$-dimensional vector $(\alpha, \beta^{T})^{T}$. Although being a special case, we will discuss the question of intercept in Section \ref{subsubsec:RemarkTheorem1} due to its important role in practice.

% it is worth pointing out that \eqref{eq:betaHat_general} considers a slightly more general M-estimator than \eqref{eq:betaHat}. In fact, if we replace $X$ by $(\textbf{1} ,  X)$ and augment $\beta$ into a $(p + 1)$-dimensional vector $(\alpha, \beta^{T})^{T}$, then the estimate \eqref{eq:betaHat_general} includes the estimate \eqref{eq:betaHat} as the last $p$ coordinates. At the same time, the notation is more concise compared to \eqref{eq:betaHat}.

~\\
\noindent \textbf{Equivariance and reduction to the null case}\\
Notice that our target quantity $\frac{\hat{\beta}_{j} - \E \hat{\beta}_{j}}{\sqrt{\Var(\hat{\beta}_{j})}}$ is invariant to the choice of $\betanull$, provided that $\betanull$ is identifiable as discussed in Section \ref{subsec:Identifiability}, we can assume $\betanull = 0$ without loss of generality. In this case, we assume in particular that the design matrix $X$ has full column rank. Then $y_{k} = \eps_{k}$ and 
\[\hat{\beta} = \argmin_{\beta\in\Rp}\frac{1}{n}\sum_{k=1}^{n}\rho(\eps_{k} - x_{k}^{T}\beta).\]
 Similarly we define the leave-$j$-th-predictor-out version as
\[\hat{\beta}_{[j]} = \argmin_{\beta\in\Rpminus}\frac{1}{n}\sum_{k=1}^{n}\rho(\eps_{k} - x_{k, [j]}^{T}\beta).\]
Based on these notations we define the full residuals $R_{k}$ as
\[R_{k} = \eps_{k} - x_{k}^{T}\hat{\beta}, \quad k = 1, 2, \ldots, n\]
and the leave-$j$-th-predictor-out residual as
\[r_{k, [j]} = \eps_{k} - x_{k, [j]}^{T}\hat{\beta}_{[j]}, \quad k = 1,2,\ldots, n, \quad j = 1, \ldots, p.\]
Three $n\times n$ diagonal matrices are defined as
\begin{equation}\label{eq:defOfDiagMatrices}
D = \diag(\psi'(R_{k}))_{k=1}^{n}, \quad \td{D} = \diag(\psi''(R_{k}))_{k=1}^{n}, \quad D_{[j]} = \diag(\psi'(r_{k, [j]}))_{k=1}^{n}.
\end{equation}
We say a random variable $Z$ is $\sigma^{2}$-sub-gaussian if for any $\lambda\in \R$,
\[\E e^{\lambda Z}\le e^{\frac{\lambda^{2}\sigma^{2}}{2}}.\]

In addition, we use $J_{n}\subset \{1, \ldots, p\}$ to represent the indices of parameters which are of interest. Intuitively, more entries in $J_{n}$ would require more stringent conditions for the asymptotic normality. % We will use $\min_{i} (\max_{i})$ as a shortcut for $\min_{1\le i\le n} (\max_{1\le i\le n})$ and $\min_{j} (\max_{j})$ as a shortcut for $\min_{j\in J_{n}} (\max_{j\in J_{n}})$. 

Finally, we adopt Landau's notation ($O(\cdot), o(\cdot), O_{p}(\cdot), o_{p}(\cdot)$). In addition, we say $a_{n} = \Omega(b_{n})$ if $b_{n} = O(a_{n})$ and similarly, we say $a_{n} = \Omega_{p}(b_{n})$ if $b_{n} = O_{p}(a_{n})$. To simplify the logarithm factors, we use the symbol $\polyLog$ to denote any factor that can be upper bounded by $(\log n)^{\gamma}$ for some $\gamma > 0$. Similarly, we use $\frac{1}{\polyLog}$ to denote any factor that can be lower bounded by $\frac{1}{(\log n)^{\gamma'}}$ for some $\gamma' > 0$.

\subsection{Technical Assumptions and main result}
Before stating the assumptions, we need to define several quantities of interest. Let
\[\lammax = \lambda_{\max}\lb\frac{X^{T}X}{n}\rb, \quad \lammin = \lambda_{\min}\lb\frac{X^{T}X}{n}\rb
\]
 be the largest (resp. smallest) eigenvalue of the  matrix $\frac{X^{T}X}{n}$.
Let $e_i \in \mathbb{R}^n$ be the $i$-th canonical basis vector and
\[
h_{j, 0} \triangleq (\psi(r_{1, [j]}), \ldots, \psi(r_{n, [j]}))^{T}, \quad h_{j, 1, i} \triangleq (I - D_{[j]}X_{[j]}(X_{[j]}^{T}D_{[j]}X_{[j]})^{-1}X_{[j]}^{T})e_{i}.\]
Finally, let 
\begin{align*}
\Delta_{C} &= \max\left\{\max_{j\in J_{n}}\frac{|h_{j, 0}^{T}X_{j}|}{\lnorm h_{j, 0}\lnorm_{2} }, \max_{i\le n, j\in J_{n}}\frac{|h_{j, 1, i}^{T}X_{j}|}{\lnorm h_{j, 1, i}\lnorm_{2} }\right\}, \\
Q_{j} &= \Cov(h_{j, 0})
\end{align*}  
Based on the quantities defined above, we state our technical assumptions on the design matrix $X$ followed by the main result. A detailed explanation of the assumptions follows. 
\begin{enumerate}[\textbf{A}1]
\item $\rho(0) = \psi(0) = 0$ and there exists positive numbers $K_{0} = \Omega\lb\frac{1}{\polyLog}\rb$, $K_{1}, K_{2} = O\lb \polyLog\rb$, such that for any $x\in\m{R}$,
\[K_{0} \le \psi'(x)\le K_{1}, \quad \bigg|\frac{d}{dx}(\sqrt{\psi'}(x))\bigg| = \frac{|\psi''(x)|}{\sqrt{\psi'(x)}}\le K_{2};\]
\item $\ep_{i} = u_{i}(W_{i})$ where $(W_{1}, \ldots, W_{n})\sim N(0, I_{n\times n})$ and $u_{i}$ are smooth functions with $\|u'_{i}\|_{\infty}\le c_{1}$ and $\|u''_{i}\|_{\infty}\le c_{2}$ for some $c_{1}, c_{2} = O(\polyLog)$. Moreover, assume $\min_{i}\Var(\eps_{i}) = \Omega\lb\frac{1}{\polyLog}\rb$.
\item $\lammax = O(\polyLog)$ and $\lammin = \Omega\lb\frac{1}{\polyLog}\rb$;
\item $\cmin\frac{X_{j}^{T}Q_{j}X_{j}}{\tr(Q_{j})} = \Omega\lb\frac{1}{\polyLog}\rb$;
\item $\E\Delta_{C}^{8} = O\lb\polyLog\rb$.
\end{enumerate}

\begin{theorem}\label{thm:main}
  Under assumptions $\textbf{A}1-\textbf{A}5$, as $p/n\rightarrow \kappa$ for some $\kappa\in (0, 1)$, while $n\rightarrow \infty$,
\[\max_{j \in J_n}d_{\mathrm{TV}}\lb\mathcal{L}\lb\frac{\hat{\beta}_{j} - \E \hat{\beta}_{j}}{\sqrt{\Var(\hat{\beta}_{j})}}\rb, N(0, 1)\rb = o(1),\]
where $d_{\mathrm{TV}}(P, Q) = \sup_{A}|P(A) - Q(A)|$ is the total variation distance.
\end{theorem}
We provide several examples where our assumptions hold in Section \ref{subsec:examples}. We also provide an example where the asymptotic normality does not hold in Section \ref{subsubsec:counterexample}. This shows that our assumptions are not just artifacts of the proof technique we developed, but that there are (probably many) situations where asymptotic normality will not hold, even coordinate-wise. 
\subsubsection{Discussion of Assumptions}\label{subsubsec:DisAssump}
Now we discuss assumptions \textbf{A}1 - \textbf{A}5. Assumption \textbf{A}1 implies the boundedness of the first-order and the second-order derivatives of $\psi$. The upper bounds are satisfied by most loss functions including the $L_{2}$ loss, the smoothed $L_{1}$ loss, the smoothed Huber loss, etc. The non-zero lower bound $K_{0}$ implies the strong convexity of $\rho$ and is required for technical reasons. It can be removed by considering first a ridge-penalized M-estimator and taking appropriate limits as in \citeA{elkaroui13, elkaroui15}. In addition, in this paper we consider the smooth loss functions and the results can be extended to non-smooth case via approximation.

Assumption \textbf{A}2 was proposed in \citeA{sopi} when deriving the second-order Poincar\'{e} inequality discussed in Section \ref{sec:heuristic}.1. It means that the results apply to non-Gaussian distributions, such as the uniform distribution on $[0, 1]$ by taking $u_{i} = \Phi$, the cumulative distribution function of standard normal distribution. Through the gaussian concentration \cite{ledoux01}, we see that \textbf{A}2 implies that $\eps_{i}$ are $c_{1}^{2}$-sub-gaussian. Thus \textbf{A}2 controls the tail behavior of $\eps_{i}$. The boundedness of $u'_{i}$ and $u''_{i}$ are required only for the direct application of Chatterjee's results. In fact, a look at his proof suggests that one can obtain a similar result to his Second-Order Poincar\'e inequality involving moment bounds on $u'_{i}(W_i)$ and $u''_{i}(W_i)$. This would be a way to weaken our assumptions to permit to have the heavy-tailed distributions expected in robustness studies. Since we are considering strongly convex loss-functions, it is not completely unnatural to restrict our attention to light-tailed errors. Furthermore, efficiency - and not only robustness - questions are one of the main reasons to consider these estimators in the moderate-dimensional context. The potential gains in efficiency obtained by considering regression M-estimates \cite{bean13} apply in the light-tailed context, which further justify our interest in this theoretical setup. 

Assumption \textbf{A}3 is completely checkable since it only depends on $X$. It controls the singularity of the design matrix. Under \textbf{A}1 and \textbf{A}3, it can be shown that the objective function is strongly convex with curvature (the smallest eigenvalue of the Hessian matrix) lower bounded by $\Omega\lb\frac{1}{\polyLog}\rb$ everywhere.

Assumption \textbf{A}4 is controlling the left tail of quadratic forms. It is fundamentally connected to aspects of the concentration of measure phenomenon \cite{ledoux01}. This condition is proposed and emphasized under the random design setting by \citeA{elkaroui13pnas}. Essentially, it means that for a matrix $Q_{j}$ ,which does not depend on $X_{j}$, the quadratic form $X_{j}^{T}Q_{j}X_{j}$ should have the same order as $\tr(Q_{j})$. 

Assumption \textbf{A}5 is proposed by \citeA{elkaroui13} under the random design settings. It is motivated by leave-one-predictor-out analysis. Note that $\Delta_{C}$ is the maximum of linear contrasts of $X_{j}$, whose coefficients do not depend on $X_{j}$. It is easily checked for design matrix $X$ which is a realization of a random matrix with i.i.d sub-gaussian entries for instance.

\begin{remark}\label{remark:unbiase}
In certain applications, it is reasonable to make the following additional assumption:
\begin{enumerate}[\textbf{A}1]
\setcounter{enumi}{5}
\item $\rho$ is an even function and $\eps_{i}$'s have symmetric distributions.
\end{enumerate}
Although assumption \textbf{A}6 is not necessary to Theorem \ref{thm:main}, it can simplify the result. Under assumption \textbf{A}6, when $X$ is full rank, we have, if $\stackrel{d}{=}$ denotes equality in distribution,
\begin{align*}
\hat{\beta} - \beta^{*} &= \argmin_{\eta\in \Rp} \frac{1}{n}\sum_{i=1}^{n}\rho(\eps_{i} - x_{i}^{T}\eta) = \argmin_{\eta\in \Rp} \frac{1}{n}\sum_{i=1}^{n}\rho(-\eps_{i} + x_{i}^{T}\eta)\\ 
&\stackrel{d}{=}\argmin_{\eta\in \Rp} \frac{1}{n}\sum_{i=1}^{n}\rho(\eps_{i} + x_{i}^{T}\eta) = \beta^{*} - \hat{\beta}.
\end{align*}
This implies that $\hat{\beta}$ is an unbiased estimator, provided it has a mean, which is the case here. Unbiasedness is useful in practice, since then Theorem \ref{thm:main} reads
\[\max_{j\in J_n}d_{\mathrm{TV}}\lb\mathcal{L}\lb\frac{\hat{\beta}_{j} - \beta_{j}^{*}}{\sqrt{\Var(\hat{\beta}_{j})}}\rb, N(0, 1)\rb = o(1)\;.\]
For inference, we only need to estimate the asymptotic variance. 
\end{remark}

\subsubsection{An important remark concerning Theorem \ref{thm:main}}\label{subsubsec:RemarkTheorem1}

When $J_{n}$ is a subset of $\{1, \ldots, p\}$, the coefficients in $J_{n}^{c}$ become nuisance parameters. Heuristically, in order for identifying $\betanull_{J_{n}}$, one only needs the subspaces $\spanvec(X_{J_{n}})$ and $\spanvec(X_{J_{n}^{c}})$ to be distinguished and $X_{J_{n}}$ has a full column rank. Here $X_{J_{n}}$ denotes the sub-matrix of $X$ with columns in $J_{n}$. Formally, let 
\[\hat{\Sigma}_{J_{n}} = \frac{1}{n}X_{J_{n}}^{T}(I - X_{J_{n}^{c}}(X_{J_{n}^{c}}^{T}X_{J_{n}^{c}})^{-}X_{J_{n}^{c}}^{T})X_{J_{n}}\]
where $A^{-}$ denotes the generalized inverse of $A$, and 
\[\td{\lambda}_{+} = \lambda_{\max}\lb \hat{\Sigma}_{J_{n}}\rb, \quad \td{\lambda}_{-} = \lambda_{\min}\lb \hat{\Sigma}_{J_{n}}\rb.\]
Then $\hat{\Sigma}_{J_{n}}$ characterizes the behavior of $X_{J_{n}}$ after removing the effect of $X_{J_{n}^{c}}$. In particular, we can modify the assumption \textbf{A}3 by 
\begin{enumerate}[\textbf{A}3*]
\setcounter{enumi}{2}  
  \item $\td{\lambda}_{+} = O(\polyLog)$ and $\td{\lambda}_{-} = \Omega\lb\frac{1}{\polyLog}\rb$.
\end{enumerate}
Then we are able to derive a stronger result in the case where $|J_{n}| < p$ than Theorem \ref{thm:main} as follows.
\begin{corollary}\label{cor:transform}
  Under assumptions \textbf{A}1-2, \textbf{A}4-5 and \textbf{A}3*, as $p/n\rightarrow \kappa$ for some $\kappa\in (0, 1)$,
\[\max_{j \in J_n}d_{\mathrm{TV}}\lb\mathcal{L}\lb\frac{\hat{\beta}_{j} - \E \hat{\beta}_{j}}{\sqrt{\Var(\hat{\beta}_{j})}}\rb, N(0, 1)\rb = o(1).\]
\end{corollary}
It can be shown that $\td{\lambda}_{+}\le \lambda_{+}$ and $\td{\lambda}_{-}\ge \lambda_{-}$ and hence the assumption \textbf{A}3* is weaker than \textbf{A}3. It is worth pointing out that the assumption \textbf{A}3* even holds when $X_{J_{n}}^{c}$ does not have full column rank, in which case $\betanull_{J_{n}}$ is still identifiable and $\hat{\beta}_{J_{n}}$ is still well-defined, although $\betanull_{J_{n}^{c}}$ and $\hat{\beta}_{J_{n}^{c}}$ are not; see Appendix \ref{subapp:partial_id} for details. % The advantage of \textbf{A}3* over \textbf{A}3 is more essential when $|J_{n}|$ is small.

\subsection{Examples}\label{subsec:examples}
Throughout this subsection (except subsubsection \ref{subsubsec:counterexample}), we consider the case where $X$ is a realization of a random matrix, denoted by $Z$ (to be distinguished from $X$). We will verify that the assumptions \textbf{A}3-\textbf{A}5 are satisfied with high probability under different regularity conditions on the distribution of $Z$. This is a standard way to justify the conditions for fixed design \cite{portnoy84, portnoy85} in the literature on regression M-estimates.

% \bluetext{This is all quite poorly written b/c there is a confusion between random variables and their realizations. After so much insistence on fixed design, I think throughout the random variables should be called Z and we just say that $X$ is a realization of $Z$. I have corrected only 1 spot: }

\subsubsection{Random Design with Independent Entries}
First we consider a random matrix $Z$ with i.i.d. sub-gaussian entries. 

\begin{proposition}\label{prop:subgauss}
  Suppose $Z$ has i.i.d. mean-zero $\sigma^{2}$-sub-gaussian entries with $\Var(Z_{ij})= \tau^{2} > 0$ for some $\sigma = O(\polyLog)$ and $\tau = \Omega\lb\frac{1}{\polyLog}\rb$, then, when $X$ is a realization of $Z$, assumptions \textbf{A}3-\textbf{A}5 for $X$ are satisfied with high probability over $Z$ for $J_{n} = \{1, \ldots, p\}$.
\end{proposition}
In practice, the assumption of identical distribution might be invalid. In fact the assumptions \textbf{A}4, \textbf{A}5 and the first part of \textbf{A}3 ($\lammax = O\lb\polyLog\rb$) are still satisfied with high probability if we only assume the independence between entries and boundedness of certain moments. To control $\lammin$, we rely on \citeA{litvak05} which assumes symmetry of each entry. We obtain the following result based on it.
\begin{proposition}\label{prop:subgauss_noniid}
  Suppose $Z$ has independent $\sigma^{2}$-sub-gaussian entries with 
\[Z_{ij}\stackrel{d}{=}-Z_{ij}, \quad \Var(Z_{ij}) > \tau^{2}\]
for some $\sigma = O\lb\polyLog\rb$ and $\tau = \Omega\lb\frac{1}{\polyLog}\rb$, then, when $X$ is a realization of $Z$, assumptions \textbf{A}3-\textbf{A}5 for $X$ are satisfied with high probability over $Z$ for $J_{n} = \{1, \ldots, p\}$.
\end{proposition}
Under the conditions of Proposition \ref{prop:subgauss_noniid}, we can add an intercept term into the design matrix. Adding an intercept allows us to remove the mean-zero assumption for $Z_{ij}$'s. In fact, suppose $Z_{ij}$ is symmetric with respect to $\mu_{j}$, which is potentially non-zero, for all $i$, then according to section \ref{subsubsec:RemarkTheorem1}, we can replace $Z_{ij}$ by $Z_{ij} - \mu_{j}$ and Proposition \ref{prop:subgauss_intercept} can be then applied.

\begin{proposition}\label{prop:subgauss_intercept}
  Suppose $Z = (\textbf{1}  ,  \td{Z})$ and $\td{Z}\in \R^{n\times (p - 1)}$ has independent $\sigma^{2}$-sub-gaussian entries with 
\[\td{Z}_{ij} - \mu_{j}\stackrel{d}{=}\mu_{j} -\td{Z}_{ij}, \quad \Var(\td{Z}_{ij}) > \tau^{2}\]
for some $\sigma = O\lb\polyLog\rb$, $\tau = \Omega\lb\frac{1}{\polyLog}\rb$ and arbitrary $\mu_{j}$. Then, when $X$ is a realization of $Z$, assumptions \textbf{A}3*, \textbf{A}4 and \textbf{A}5 for $X$ are satisfied with high probability over $Z$ for $J_{n} = \{2, \ldots, p\}$.
\end{proposition}

\subsubsection{Dependent Gaussian Design}
To show that our assumptions handle a variety of situations, we now assume that the observations, namely the rows of $Z$, are i.i.d. random vectors with a covariance matrix $\Sigma$. In particular we show that the Gaussian design, i.e. $z_{i}\stackrel{i.i.d.}{\sim}N(0, \Sigma)$, satisfies the assumptions with high probability.
\begin{proposition}\label{prop:depend_gauss}
Suppose $z_{i}\stackrel{i.i.d.}{\sim}N(0, \Sigma)$ with $\lambda_{\max}(\Sigma) = O\lb\polyLog\rb$ and $\lambda_{\min}(\Sigma) = \Omega\lb\frac{1}{\polyLog}\rb$, then, when $X$ is a realization of $Z$, assumptions \textbf{A}3-\textbf{A}5 for $X$ are satisfied with high probability over $Z$ for $J_{n} = \{1, \ldots, p\}$.
\end{proposition}
This result extends to the matrix-normal design \cite{muirhead82}[Chapter 3], i.e. $(Z_{ij})_{i\le n, j\le p}$ is one realization of  a $np$-dimensional random variable $Z$ with multivariate gaussian distribution 
\[\vect(Z)\triangleq (z_{1}^{T}, z_{2}^{T}, \ldots, z_{n}^{T})\sim N(0, \Lambda\otimes \Sigma),\]
and $\otimes$ is the Kronecker product. It turns out that assumptions $\textbf{A}3-\textbf{A}5$ are satisfied if both $\Lambda$ and $\Sigma$ are well-behaved. 
\begin{proposition}\label{prop:depend_gauss_row}
Suppose $Z$ is matrix-normal with $\vect(Z)\sim N(0, \Lambda\otimes \Sigma)$ and 
\[\lambda_{\max}(\Lambda), \lambda_{\max}(\Sigma) = O\lb\polyLog\rb, \quad \lambda_{\min}(\Lambda), \lambda_{\min}(\Sigma) = \Omega\lb\frac{1}{\polyLog}\rb\;.\]
Then, when $X$ is a realization of $Z$,assumptions \textbf{A}3-\textbf{A}5 for $X$ are satisfied with high probability over $Z$ for $J_{n} = \{1, \ldots, p\}$.
\end{proposition}
In order to incorporate an intercept term, we need slightly more stringent condition on $\Lambda$. Instead of assumption \textbf{A}3, we prove that assumption \textbf{A}3* - see subsubsection \ref{subsubsec:RemarkTheorem1} - holds with high probability. 
\begin{proposition}\label{prop:depend_gauss_intercept}
Suppose $Z$ contains an intercept term, i.e. $Z = (\textbf{1} ,  \td{Z})$ and $\td{Z}$ satisfies the conditions of Proposition \ref{prop:depend_gauss_row}. Further assume that 
\begin{equation}\label{eq:Lambdaone}
\frac{\max_{i}|(\Lambda^{-\frac{1}{2}}\textbf{1})_{i}|}{\min_{i}|(\Lambda^{-\frac{1}{2}}\textbf{1})_{i}|} = O\lb\polyLog\rb.
\end{equation}
Then, when $X$ is a realization of $Z$, assumptions \textbf{A}3*, \textbf{A}4 and \textbf{A}5 for $X$ are satisfied with high probability over $Z$ for $J_{n} = \{2, \ldots, p\}$.
\end{proposition}
When $\Lambda = I$, the condition \eqref{eq:Lambdaone} is satisfied. Another non-trivial example is the exchangeable case where $\Lambda_{ij}$ are all equal for $i\not =j$. In this case, $\textbf{1}$ is an eigenvector of $\Lambda$ and hence it is also an eigenvector of $\Lambda^{-\frac{1}{2}}$. Thus $\Lambda^{-\frac{1}{2}}\textbf{1}$ is a multiple of $\textbf{1}$ and the condition \eqref{eq:Lambdaone} is satisfied.

\subsubsection{Elliptical Design}
Furthermore, we can move from Gaussian-like structure to generalized elliptical models where  $z_i= \zeta_{i}\Sigma^{1/2}\mathcal{Z}_i$ where $\{\zeta_{i}, \mathcal{Z}_{ij}: i = 1,\ldots, n; j = 1,\ldots, p\}$ are independent random variables, $\mathcal{Z}_{ij}$ having for instance mean 0 and variance 1.  % If $\mathcal{Z}_{ij}$ were ${\cal N}(0,1)$, the previous example would simply be the matrix-normal distribution with $\Lambda=\diag(|\lambda_i|)$ and $\Sigma=\id_p$ given $\{\zeta_{i}: i = 1, \ldots, n\}$. 
The elliptical family is quite flexible in modeling data. It represents a type of data formed by a common driven factor and independent individual effects. It is widely used in multivariate statistics (\citeA{anderson62,tyler87}) and various fields, including finance \cite{cizek05} and biology \cite{posekany11}. In the context of high-dimensional statistics, this class of model was used to refute universality claims in random matrix theory \cite{nekCorrEllipD}. In robust regression, \citeA{elkaroui11} used elliptical models to show that the limit of $\|\hat{\beta}\|_{2}^{2}$ depends on the distribution of $\zeta_{i}$ and hence the geometry of the predictors. As such, studies limited to Gaussian-like design were shown to be of very limited statistical interest. See also the deep classical inadmissibility results \cite{baranchik73, jurevckova97}. However, as we will show in the next proposition, the common factors $\zeta_{i}$ do not distort the shape of the asymptotic distribution. A similar phenomenon happens in the random design case - see \citeA{elkaroui13pnas,bean13}.
\begin{proposition}\label{prop:ellip}
  Suppose $Z$ is generated from an elliptical model, i.e. 
\[Z_{ij} = \zeta_{i} \mathcal{Z}_{ij},\]
where $\zeta_{i}$ are independent random variables taking values in $[a, b]$ for some $0 < a < b < \infty$ and $\mathcal{Z}_{ij}$ are independent random variables satisfying the conditions of Proposition \ref{prop:subgauss} or Proposition \ref{prop:subgauss_noniid}. Further assume that $\{\zeta_{i}: i = 1, \ldots, n\}$ and $\{\mathcal{Z}_{ij}: i = 1, \ldots, n; j = 1, \ldots, p\}$ are independent. Then, when $X$ is a realization of $Z$, assumptions \textbf{A}3-\textbf{A}5 for $X$ are satisfied with high probability over $Z$ for $J_{n} = \{1, \ldots, p\}$.
\end{proposition}
Thanks to the fact that $\zeta_{i}$ is bounded away from 0 and $\infty$,  the proof of Proposition \ref{prop:ellip} is straightforward, as shown in Appendix \ref{app:others}. However, by a more refined argument and assuming identical distributions $\zeta_{i}$, we can relax this condition. 
\begin{proposition}\label{prop:ellip_relax}
Under the conditions of Proposition \ref{prop:ellip} (except the boundedness of $\zeta_{i}$) and assume $\zeta_{i}$ are i.i.d. samples generated from some distribution $F$, independent of $n$, with
\[P\lb\zeta_{1} \ge t\rb\le c_{1}e^{-c_{2}t^{\alpha}},\]
for some fixed $c_{1}, c_{2}, \alpha > 0$ and $F^{-1}(q) > 0$ for any $q \in (0, 1)$ where $F^{-1}$ is the quantile function of $F$ and is continuous. Then, when $X$ is a realization of $Z$, assumptions \textbf{A}3-\textbf{A}5 for $X$ are satisfied with high probability over $Z$ for $J_{n} = \{1, \ldots, p\}$.
\end{proposition}

\subsubsection{A counterexample}\label{subsubsec:counterexample}
Consider a one-way ANOVA situation. In other words, let the design matrix have exactly 1 non-zero entry per row, whose value is 1. Let $\{k_i\}_{i=1}^n$ be integers in $\{1,\ldots,p\}$. And let $X_{i,j}=1(j=k_i)$. Furthermore, let us constrain $n_{j}=|\{i: k_i=j\}|$ to be such that $1\leq n_{j} \leq 2 \lfloor p/n \rfloor$. Taking for instance $k_i=(i \mod p)$ is an easy way to produce such a matrix. The associated statistical model is just $y_i=\eps_i+\betanull_{k_i}$\;.

It is easy to see that 
$$
\hat{\beta}_j=\argmin_{\beta\in\mathbb{R}} \sum_{i: k_i=j} \rho(y_i-\beta_j)=\argmin_{\beta\in\mathbb{R}} \sum_{i: k_i=j} \rho(\eps_i-(\beta_j-\betanull_j))\;.
$$
This is of course a standard location problem. 
In the moderate-dimensional setting we consider, $n_{j}$ remains finite as $n\tendsto \infty$. So $\hat{\beta}_j$ is a non-linear function of finitely many random variables and will in general not be normally distributed. 

For concreteness, one can take $\rho(x)=|x|$, in which case $\hat{\beta}_j$ is a median of $\{y_i\}_{\{i: k_i=j\}}$. The cdf of $\hat{\beta}_j$ is known exactly by elementary order statistics computations (see \citeA{david81}) and is not that of a Gaussian random variable in general. In fact, the ANOVA design considered here violates the assumption \textbf{A}3 since $\lammin = \min_{j}n_{j} / n = O\lb 1 / n\rb$. Further, we can show that the assumption \textbf{A}5 is also violated, at least in the least-square case; see Section \ref{subsec:an_lse} for details.

\subsection{Comments and discussions}

\subsubsection{Asymptotic Normality in High Dimensions}
In the $p$-fixed regime, the asymptotic distribution is easily defined as the limit of $\mathcal{L}(\mest)$ in terms of weak topology \cite{vandervaart}. However, in regimes where the dimension $p$ grows, the notion of asymptotic distribution is more delicate. a conceptual question arises from the fact that the dimension of the estimator $\hat{\beta}$ changes with $n$ and thus there is no well-defined distribution which can serve as the limit of $\mathcal{L}(\hat{\beta})$, where $\mathcal{L}(\cdot)$ denotes the law.  One remedy is proposed by \citeA{mallows72}. Under this framework, a triangular array $\{W_{n, j}, j = 1, 2,\ldots, p_{n}\}$, with $\E W_{n, j} = 0, \E W_{n, j}^{2} = 1$, is called jointly asymptotically normal if for any deterministic sequence $a_{n}\in \R^{p_{n}}$ with $\|a_{n}\|_{2} = 1$,
\[\mathcal{L}\lb \sum_{j=1}^{p_{n}}a_{n, j}W_{n, j}\rb\rightarrow N(0, 1).\]
When the zero mean and unit variance are not satisfied, it is easy to modify the definition by normalizing random variables.
\begin{definition}[joint asymptotic normality]\label{def:jan_mallows}
~\\

$\{W_{n}: W_{n}\in \R^{p_{n}}\}$ is jointly asymptotically normal if and only if for any sequence $\{a_{n}: a_{n}\in \R^{p_{n}}\}$,
\[\mathcal{L}\lb\frac{a_{n}^{T}(W_{n} - \E W_{n})}{\sqrt{a_{n}^{T}\Cov(W_{n})a_{n}}}\rb\rightarrow N(0, 1).\]
\end{definition}
The above definition of asymptotic normality is strong and appealing but was shown not to hold for least-squares in the moderate $p/n$ regime \cite{huber73}. In fact, \citeA{huber73} shows that $\lse$ is jointly asymtotically normal only if
\[\max_{i}(X(X^{T}X)^{-1}X^{T})_{i,i}\rightarrow 0.\]
When $p / n\rightarrow \kappa \in(0, 1)$, provided $X$ is full rank,
\[\max_{i}(X(X^{T}X)^{-1}X^{T})_{i,i}\ge \frac{1}{n}\tr(X(X^{T}X)^{-1}X^{T}) = \frac{p}{n}\rightarrow \kappa  > 0.\]
In other words, in moderate $p / n$ regime, the asymptotic normality cannot hold for all linear contrasts, even in the case of least-squares.

In applications, however, it is usually not necessary to consider all linear contrasts but instead a small subset of them, e.g. all coordinates or low dimensional linear contrasts such as $\betanull_{1} - \betanull_{2}$. We can naturally modify Definition \ref{def:jan_mallows} and adapt to our needs by imposing constraints on $a_{n}$. A popular concept, which we use in Section \ref{sec:intro} informally, is called coordinate-wise asymptotic normality and defined by restricting $a_{n}$ to be the canonical basis vectors, which have only one non-zero element. An equivalent definition is stated as follows.
\begin{definition}[coordinate-wise asymptotic normal]\label{def: casn_mallows}
~\\

  $\{W_{n}: W_{n}\in \R^{p_{n}}\}$ is coordinate-wise asymptotically normal if and only if for any sequence $\{j_{n}: j_{n}\in \{1, \ldots,  p_{n}\}\}$,
\[\mathcal{L}\lb\frac{W_{n, j_{n}} - \E W_{n, j_{n}}}{\sqrt{\Var(W_{n, j_{n}})}}\rb\rightarrow N(0, 1).\]
\end{definition}
A more convenient way to define the coordinate-wise asymptotic normality is to introduce a metric $d(\cdot, \cdot)$, e.g. Kolmogorov distance and total variation distance, which induces the weak convergence topology. Then $W_{n}$ is coordinate-wise asymptotically normal if and only if
\[\max_{j}d\lb\mathcal{L}\lb\frac{W_{n, j} - \E W_{n, j}}{\sqrt{\Var(W_{n, j})}}\rb, N(0, 1)\rb = o(1).\]

\subsubsection{Discussion about inference and technical assumptions}\label{subsubsec:FurtherDis}

\noindent \textbf{Variance and bias estimation}\\

To complete the inference, we need to compute the bias and variance. As discussed in Remark \ref{remark:unbiase}, the M-estimator is unbiased if the loss function and the error distribution are symmetric. For the variance, it is easy to get a conservative estimate via resampling methods such as Jackknife as a consequence of Efron-Stein's inequality; see \citeA{elkaroui13} and \citeA{elkaroui15boot} for details. Moreover, by the variance decomposition formula,
\[\Var(\hat{\beta}_{j}) = \E \left[\Var(\hat{\beta}_{j} | X)\right] + \Var\left[\E (\hat{\beta}_{j} | X)\right]\ge \E \left[\Var(\hat{\beta}_{j}|X)\right],
\]
the unconditional variance, when $X$ is a random design matrix, is a conservative estimate. The unconditional variance can be calculated by solving a non-linear system; see \citeA{elkaroui13} and \citeA{donoho16}.

However, estimating the exact variance is known to be hard. \citeA{elkaroui15boot} show that the existing resampling schemes, including jacknife, pairs-bootstrap, residual bootstrap, etc., are either too conservative or too anti-conservative when $p/n$ is large. The challenge, as mentioned in \citeA{elkaroui13, elkaroui15boot}, is due to the fact that the residuals $\{R_{i}\}$ do not mimic the behavior of $\{\eps_{i}\}$ and that the resampling methods effectively modifies the geometry of the dataset from the point of view of the statistics of interest. We believe that variance estimation in moderate $p/n$ regime should rely on different methodologies from the ones used in low-dimensional estimation.

~\\
\noindent \textbf{Technical assumptions}\\

On the other hand, we assume that $\rho$ is strongly convex. One remedy would be adding a ridge regularized term as in \citeA{elkaroui13} and the new problem is amenable to analysis with the method we used in this article. However, the regularization term introduces a non-vanishing bias, which is as hard to be derived as the variance. For unregularized M-estimators, the strong convexity is also assumed by other works \cite{elkaroui13, donoho16}. However, we believe that this assumption is unnecessary and can be removed at least for well-behaved design matrices. Another possibility, for errors that have more than 2 moments is to just add a small quadratic term to the loss function, e.g. $\lambda x^2/2$ with a small $\lambda$. Finally, we recall that in many situations, least-squares is actually more efficient than $\ell_1$-regression (see numerical work in \citeA{bean13}) in moderate dimensions. This is for instance the case for double-exponential errors if $p/n$ is greater than .3 or so. As such working with strongly convex loss functions is as problematic for moderate-dimensional regression M-estimates as it would be in the low-dimensional setting.

To explore traditional robustness questions, we will need to weaken the requirements of Assumption \textbf{A}2. This requires substantial work and an extension of the main results of \citeA{sopi}. Because the technical part of the paper is already long, we leave this interesting statistical question to future works.

% Last but not least, to explore the robustness of M-estimator, it is conducive to relax the assumption \textbf{A}2, which essentially assumes that $\eps_{i}$ are sub-gaussian. \sout{In fact, this assumption is only needed in \SOPI. However, by modifying SOPI slightly, we obtain a more general inequality which involves $u_{i}'$ and $u_{i}''$ explicitly. Instead of simply using the infinity norm, we can carry out a more refined analysis. To avoid distraction we will discuss this extension in future works.}

\section{Proof Sketch}\label{sec:heuristic}

Since the proof of Theorem \ref{thm:main} is somewhat technical, we illustrate the main idea in this section. 

First notice that the M-estimator $\hat{\beta}$ is an implicit function of independent random variables $\eps_{1}, \ldots, \eps_{n}$, which is determined by 
\begin{equation}\label{eq:first_order}
\frac{1}{n}\sum_{i=1}^{n}x_{i}\psi(\eps_{i} - x_{i}\hat{\beta}) = 0.
\end{equation}
The Hessian matrix of the loss function in \eqref{eq:betaHat_general} is $\frac{1}{n}X^{T}DX\succeq D_{0}\lammin I_{p}$ under the notation introduced in section \ref{subsec:notation}. The assumption \textbf{A}3 then implies that the loss function is strongly convex, in which case $\hat{\beta}$ is unique. Then $\hat{\beta}$ can be seen as a non-linear function of $\eps_i$'s. A powerful central limit theorem for this type of statistics is the second-order Poincar\'e inequality (SOPI), developed in \citeA{sopi} and used there to re-prove central limit theorems for linear spectral statistics of large random matrices. We recall one of the main results for the convenience of the reader. 

\begin{proposition}[SOPI; \citeNP{sopi}]\label{prop:sopi}
  Let $\mathscr{W} = (\mathscr{W}_{1}, \ldots, \mathscr{W}_{n}) = (u_{1}(W_{1}), \ldots, u_{n}(W_{n}))$ where $W_{i}\stackrel{i.i.d.}{\sim} N(0, 1)$ and $\|u'_{i}\|_{\infty}\le c_{1}, \|u''_{i}\|_{\infty}\le c_{2}$. Take any $g\in C^{2}(\R^{n})$ and let $\nabla_{i}g$, $\nabla g$ and $\nabla^{2}g$ denote the $i$-th partial derivative, gradient and Hessian of $g$. Let
\[\kappa_{0} = \lb\E \sum_{i=1}^{n}\big|\nabla_{i}g(\mathscr{W})\big|^{4}\rb^{\frac{1}{2}}, \quad \kappa_{1}  = (\E \|\nabla g(\mathscr{W})\|_{2}^{4})^{\frac{1}{4}}, \quad \kappa_{2} = (\E \|\nabla^{2}g(\mathscr{W})\|_{\mathrm{op}}^{4})^{\frac{1}{4}},\]
and $U = g(\mathscr{W})$. If $U$ has finite fourth moment, then
\[d_{\mathrm{TV}}\lb \mathcal{L}\lb\frac{U - \E U}{\sqrt{\Var(U)}}\rb, N(0, 1)\rb\le \frac{2\sqrt{5}(c_{1}c_{2}\kappa_{0} + c_{1}^{3}\kappa_{1}\kappa_{2})}{\Var(U)}.\]
\end{proposition}
From (\ref{eq:first_order}), it is not hard to compute the gradient and Hessian of $\hat{\beta}_{j}$ with respect to $\eps$. Recalling the definitions in Equation \eqref{eq:defOfDiagMatrices} on p. \pageref{eq:defOfDiagMatrices}, we have
\begin{lemma}\label{lem:beta_deriv}
Suppose $\psi\in C^{2}(\R^{n})$, then
\begin{equation}\label{eq:fd}
\Fd{j} = e_{j}^{T}(\wm)^{-1}X^{T}D
\end{equation}
\begin{equation}\label{eq:sd}
\SD{j} = G^{T}\diag(e_{j}^{T}(\wm)^{-1}X^{T}\tilde{D})G
\end{equation}
where $e_{j}$ is the $j$-th cononical basis vectors in $\R^{p}$ and 
\[G = I - X(\wm)^{-1}X^{T}D.\]
\end{lemma}
Recalling the definitions of $K_i$'s in Assumption \textbf{A}1 on p. \pageref{thm:main}, we can bound $\kappa_{0}$, $\kappa_{1}$ and $\kappa_{2}$ as follows.
\begin{lemma}\label{lem:beta_deriv_bound}
Let $\kappa_{0j}, \kappa_{1j}, \kappa_{2j}$ defined as in Proposition \ref{prop:sopi} by setting $\mathscr{W} = \eps$ and $g(\mathscr{W})= \hat{\beta}_{j}$. Let 
\begin{equation}\label{eq:def_Mj}
M_{j} = \E\|e_{j}^{T}(\wm)^{-1}X^{T}D^{\frac{1}{2}}\|_{\infty},
\end{equation}
then 
\[\kappa_{0j}^{2}\le \frac{K_{1}^{2}}{(nK_{0}\lammin)^{\frac{3}{2}}}\cdot M_{j}, \quad \kappa_{1j}^{4}\le \frac{K_{1}^{2}}{(nK_{0}\lammin)^{2}}, \quad \kappa_{2j}^{4}\le \frac{K_{2}^{4}}{(nK_{0}\lammin)^{\frac{3}{2}}}\cdot \lb\frac{K_{1}}{K_{0}}\rb^{4}\cdot M_{j}.\]
\end{lemma}
As a consequence of the \SOPI, we can bound the total variation distance between $\hat{\beta}_{j}$ and a normal distribution by $M_{j}$ and $\Var(\hat{\beta}_{j})$. More precisely, we prove the following Lemma.
\begin{lemma}\label{lem:tv_bound}
Under assumptions \textbf{A}1-\textbf{A}3, 
  \[\max_{j} d_{\mathrm{TV}}\lb\mathcal{L}\lb\frac{\hat{\beta}_{j} - \E \hat{\beta}_{j}}{\sqrt{\Var(\hat{\beta}_{j})}}\rb, N(0, 1)\rb = O_{p}\lb\frac{\max_{j}(nM_{j}^{2})^{\frac{1}{8}}}{n\cdot \min_{j}\Var(\hat{\beta}_{j})}\cdot \polyLog\rb.\] 
\end{lemma}
Lemma \ref{lem:tv_bound} is the key to prove Theorem \ref{thm:main}. To obtain the coordinate-wise asymptotic normality, it is left to establish an upper bound for $M_{j}$ and a lower bound for $\Var(\hat{\beta}_{j})$. In fact, we can prove that 
\begin{lemma}\label{lem:heuristic}
Under assumptions \textbf{A}1 - \textbf{A}5,
\[\max_{j}M_{j} = O\lb\frac{\polyLog}{n}\rb, \quad \min_{j}\Var(\hat{\beta}_{j}) = \Omega\lb\frac{1}{n\cdot \polyLog}\rb.\]
\end{lemma}
Then Lemma \ref{lem:tv_bound} and Lemma \ref{lem:heuristic} together imply that 
\[\max_{j} d_{\mathrm{TV}}\lb\mathcal{L}\lb\frac{\hat{\beta}_{j} - \E \hat{\beta}_{j}}{\sqrt{\Var(\hat{\beta}_{j})}}\rb, N(0, 1)\rb = O\lb \frac{\polyLog}{n^{\frac{1}{8}}}\rb = o(1).\]
Appendix \ref{app:heuristic}, provides a roadmap of the proof of  Lemma \ref{lem:heuristic} under a special case where the design matrix $X$ is one realization of a random matrix with i.i.d. sub-gaussian entries. It also serves as an outline of the rigorous proof in Appendix \ref{app:main}. 

\subsection{Comment on the Second-Order Poincar\'e inequality}
Notice that when $g$ is a linear function such that $g(z) = \sum_{i=1}^{n}a_{i}z_{i}$, then the Berry-Esseen inequality \cite{esseen45} implies that
\[d_{K}\lb\mathcal{L}\lb\frac{W - \E W}{\sqrt{\Var(W)}}\rb, N(0, 1)\rb\preceq \frac{\sum_{i=1}^{n}|a_{i}|^{3}}{\lb\sum_{i=1}^{n}a_{i}^{2}\rb^{\frac{3}{2}}},\]
where 
\[d_{K}(F, G) = \sup_{x}|F(x) - G(x)|.\]
On the other hand, the \SOPI implies that
\[d_{K}\lb\mathcal{L}\lb\frac{W - \E W}{\sqrt{\Var(W)}}\rb, N(0, 1)\rb\le d_{\mathrm{TV}}\lb\mathcal{L}\lb\frac{W - \E W}{\sqrt{\Var(W)}}\rb, N(0, 1)\rb\preceq \frac{\lb\sum_{i=1}^{n}a_{i}^{4}\rb^{\frac{1}{2}}}{\sum_{i=1}^{n}a_{i}^{2}}\;.\]
This is slightly worse than the Berry-Esseen bound and requires stronger conditions on the distributions of variates but provides bounds for TV metric instead of Kolmogorov metric. This comparison shows that \SOPI can be regarded as a generalization of the Berry-Esseen bound for non-linear transformations of independent random variables.

\section{Least-Squares Estimator}\label{sec:lse}
The Least-Squares Estimator is a special case of an M-estimator with $\rho(x) = \frac{1}{2}x^{2}$. Because the estimator can then be written explicitly, the analysis of its properties is extremely simple and it has been understood for several decades (see arguments in e.g. \citeA{huber73}[Lemma 2.1] and \citeA{HuberBook81}[Proposition 2.2]). In this case, the hat matrix $H=X(X\trsp X)^{-1}X^{T}$ captures all the problems associated with dimensionality in the problem. In particular, proving the asymptotic normality simply requires an application of the Lindeberg-Feller theorem.

It is however somewhat helpful to compare the conditions required for asymptotic normality in this simple case and the ones we required in the more general setup of Theorem \ref{thm:main}. We do so briefly in this section. 
\subsection{Coordinate-Wise Asymptotic Normality of LSE}\label{subsec:an_lse}
Under the linear model \eqref{eq:linearmodel}, when $X$ is full rank,
\[\lse = \betanull + (X^{T}X)^{-1}X^{T}\eps,\]
thus each coordinate of $\lse$ is a linear contrast of $\eps$ with zero mean. Instead of assumption \textbf{A}2, which requires $\eps_{i}$ to be sub-gaussian, we only need to assume $\max_{i}\E |\eps_{i}|^{3} < \infty$, under which the Berry-Essen bound for non-i.i.d. data \cite{esseen45} implies that 
\[d_{K}\lb \mathcal{L}\lb\frac{\hat{\beta}_{j} - \betanull_{j}}{\sqrt{\Var(\hat{\beta}_{j})}}\rb, N(0, 1)\rb \preceq \frac{\|e_{j}(X^{T}X)^{-1}X^{T}\|_{3}^{3}}{\|e_{j}^{T}(X^{T}X)^{-1}X^{T}\|_{2}^{3}}\le \frac{\|e_{j}(X^{T}X)^{-1}X^{T}\|_{\infty}}{\|e_{j}(X^{T}X)^{-1}X^{T}\|_{2}}.\]
This motivates us to define a matrix specific quantity $S_{j}(X)$ such that
\begin{equation}
  \label{eq:Hx}
  S_{j}(X) = \frac{\|e_{j}^{T}(X^{T}X)^{-1}X^{T}\|_{\infty}}{\|e_{j}^{T}(X^{T}X)^{-1}X^{T}\|_{2}}
\end{equation}
then the Berry-Esseen bound implies that $\max_{j\in J_{n}}S_{j}(X)$ determines the coordinate-wise asymptotic normality of $\lse$. 
\begin{theorem}\label{thm:OLS_casn_kd}
If $\E \max_{i}|\eps_{i}|^{3} < \infty$, then 
\[\max_{j\in J_{n}}d_{K}\lb \frac{\hat{\beta}_{LS, j} - \beta_{0, j}}{\sqrt{\Var(\hat{\beta}_{LS, j})}}, N(0, 1)\rb\le A\cdot\frac{\E |\eps_{i}|^{3}}{(\E \eps_{i}^{2})^{\frac{3}{2}}}\cdot \max_{j\in J_{n}}S_{j}(X),\]
where $A$ is an absolute constant and $d_{K}(\cdot, \cdot)$ is the Kolmogorov distance, defined as 
\[d_{K}(F, G) = \sup_{x}|F(x) - G(x)|.\]
\end{theorem}
It turns out that $\max_{j\in J_{n}}S_{j}(X)$ plays in the least-squares setting the role of $\Delta_{C}$ in assumption \textbf{A}5. Since it has been known that a condition like $S_j(X)\tendsto 0$ is necessary for asymptotic normality of least-square estimators (\citeA{huber73}[Proposition 2.2]), this shows in particular that our Assumption \textbf{A}5, or a variant, is also needed in the general case. See Appendix \ref{app:relationSjDelta} for details.

\subsection{Discussion}
Naturally, checking the conditions for asymptotic normality is much easier in the least-squares case than in the general case under consideration in this paper. In particular: 
\begin{enumerate}
\item Asymptotic normality conditions can be checked for a broader class of random design matrices. See Appendix \ref{app:lseExamples} for details.
\item For orthogonal design matrices, i.e $X\trsp X=c\id$ for some $c>0$, $S_{j}(X) = \frac{\|X_j\|_{\infty}}{\|X_j\|_{2}}$. Hence, the condition $S_{j}(X) = o(1)$ is true if and only if no entry dominates the $j-th$ row of $X$.
\item The ANOVA-type counterexample we gave in Section \ref{subsubsec:counterexample} still provides a counter-example. The reason now is different: namely the sum of finitely many independent random variables is evidently in general non-Gaussian. In fact, in this case, $S_{j}(X) = \frac{1}{\sqrt{n_{j}}}$ is bounded away from $0$. 
\end{enumerate}
Inferential questions are also extremely simple in this context and essentially again dimension-independent for the reasons highlighted above. 
Theorem \ref{thm:OLS_casn_kd} naturally reads, 
\begin{equation}\label{eq:betaj_lse}
\frac{\hat{\beta}_{j} - \betanull_{j}}{\sigma\sqrt{e_{j}^{T}(X^{T}X)^{-1}e_{j}}}\stackrel{d}{\rightarrow} N(0, 1).
\end{equation}
Estimating $\sigma$ is still simple under minimal conditions provided $n-p\tendsto \infty$: see \citeA{bickel83}[Theorem 1.3] or standard computations concerning the normalized residual sum-of-squares (using variance computations for the latter may require up to 4 moments for $\eps_i$'s). Then we can replace $\sigma$ in \eqref{eq:betaj_lse} by $\hat{\sigma}$ with
\[\hat{\sigma}^{2} = \frac{1}{n - p}\sum_{k=1}^{n}R_{k}^{2}\]
where $R_{k} = y_{k} - x_{k}^{T}\hat{\beta}$ and construct confidence intervals for $\betanull_{j}$ based on $\hat{\sigma}$. If $n-p$ does not tend to $\infty$, the normalized residual sum of squares is evidently not consistent even in the case of Gaussian errors, so this requirement may not be dispensed of.

\section{Numerical Results}\label{sec:numerical}
As seen in the previous sections and related papers, there are five important factors that affect the distribution of $\hat{\beta}$: the design matrix $X$, the error distribution $\mathcal{L}(\eps)$, the sample size $n$, the ratio $\kappa$, and the loss function $\rho$. The aim of this section is to assess the quality of the agreement between the asymptotic theoretical results of Theorem \ref{thm:main} and the empirical, finite-dimensional properties of $\betaHat$. We also perform a few simulations where some of the assumptions of Theorem \ref{thm:main} are violated to get an intuitive sense of whether those assumptions appear necessary or whether they are simply technical artifacts associated with the method of proof we developed. As such, the numerical experiments we report on in this section can be seen as a complement to Theorem \ref{thm:main} rather than only a simple check of its practical relevance.

The design matrices we consider are one realization of random design matrices of the following three types:
\begin{description}
\item[(i.i.d. design)]: $X_{ij}\stackrel{i.i.d.}{\sim} F$;
\item[(elliptical design)]: $X_{ij} = \zeta_{i}\td{X}_{ij}$, where $\td{X}_{ij}\stackrel{i.i.d.}{\sim} N(0, 1)$ and $\zeta_{i}\stackrel{i.i.d.}{\sim}F$. In addition, $\{\zeta_{i}\}$ is independent of $\{\td{X}_{ij}\}$;
\item[(partial Hadamard design)]: a matrix formed by a random set of $p$ columns of a $n\times n$ Hadamard matrix, i.e. a $n\times n$ matrix whose columns are orthogonal with entries restricted to $\pm 1$. 
\end{description}
Here we consider two candidates for $F$ in i.i.d. design and elliptical design: standard normal distribution $N(0, 1)$ and t-distribution with two degrees of freedom (denoted $\tdist_2$). For the error distribution, we assume that $\eps$ has i.i.d. entries with one of the above two distributions, namely $N(0, 1)$ and $\tdist_2$. The $\tdist$-distribution violates our assumption \textbf{A}2.

To evaluate the finite sample performance, we consider the sample sizes $n\in \{100, 200, 400, 800\}$ and $\kappa \in \{0.5, 0.8\}$. In this section we will consider a Huber loss with $k = 1.345$ \cite{HuberBook81}, i.e. 
\[\rho(x) = \left\{
  \begin{array}{ll}
    \frac{1}{2}x^{2} & |x| \le k\\
    kx - \frac{k^{2}}{2} & |x| > k
  \end{array}
\right.\]
$k=1.345$ is the default in \texttt{R} and yields 95\% relative efficiency for Gaussian errors in low-dimensional problems. We also carried out the numerical work for $L_1$-regression, i.e. $\rho(x)=|x|$. See Appendix \ref{app:numerical} for details.

\subsection{Asymptotic Normality of A Single Coordinate}\label{subsec:normalityOneCoord}
First we simulate the finite sample distribution of $\hat{\beta}_{1}$, the first coordinate of $\hat{\beta}$. For each combination of sample size $n$ ($100, 200, 400$ and $800$), type of design (i.i.d, elliptical and Hadamard), entry distribution $F$ (normal and $\tdist_2$) and error distribution $\mathcal{L}(\eps)$ (normal and $\tdist_2$), we run 50 simulations with each consisting of the following steps:
\begin{enumerate}[(Step 1)]
\item Generate one design matrix $X$;
\item Generate the 300 error vectors $\eps$;
\item Regress each $Y = \eps$ on the design matrix $X$ and end up with 300 random samples of $\hat{\beta}_{1}$, denoted by $\hat{\beta}_{1}^{(1)}, \ldots, \hat{\beta}_{1}^{(300)}$; 
\item Estimate the standard deviation of $\hat{\beta}_{1}$ by the sample standard error $\widehat{\mathrm{sd}}$;
\item Construct a confidence interval $\mathcal{I}^{(k)} = \left[\hat{\beta}_{1}^{(k)} - 1.96\cdot  \widehat{\mathrm{sd}}, \hat{\beta}_{1}^{(k)} + 1.96 \cdot \widehat{\mathrm{sd}}\right]$ for each $k = 1, \ldots, 300$;
\item Calculate the empirical 95\% coverage by the proportion of confidence intervals which cover the true $\beta_{1} = 0$.
\end{enumerate}
Finally, we display the boxplots of the empirical 95\% coverages of $\hat{\beta}_{1}$ for each case in Figure \ref{fig:singleCoord}. It is worth mentioning that our theories cover two cases: 1) i.i.d design with normal entries and normal errors (orange bars in the first row and the first column), see Proposition \ref{prop:subgauss}; 2) elliptical design with normal factors $\zeta_{i}$ and normal errors (orange bars in the second row and the first column), see Proposition \ref{prop:ellip}. 

We first discuss the case $\kappa = 0.5$. In this case, there are only two samples per parameter. Nonetheless, we observe that the coverage is quite close to 0.95, even with a sample size as small as $100$, in both cases that are covered by our theories. For other cases, it is interesting to see that the coverage is valid and most stable in the partial hadamard design case and is not sensitive to the distribution of multiplicative factor in elliptical design case even when the error has a $\tdist_2$ distribution. For i.i.d. designs, the coverage is still valid and stable when the entry is normal. By contrast, when the entry has a $\tdist_2$ distribution, the coverage has a large variation in small samples. The average coverage is still close to 0.95 in the i.i.d. normal design case but is slightly lower than 0.95 in the i.i.d. $\tdist_2$ design case. In summary, the finite sample distribution of $\hat{\beta}_{1}$ is more sensitive to the entry distribution than the error distribution. This indicates that the assumptions on the design matrix are not just artifacts of the proof but are quite essential. 

The same conclusion can be drawn from the case where $\kappa = 0.8$ except that the variation becomes larger in most cases when the sample size is small. However, it is worth pointing out that even in this case where there is 1.25 samples per parameter, the sample distribution of $\hat{\beta}_{1}$ is well approximated by a normal distribution with a moderate sample size ($n \ge 400$). This is in contrast to the classical rule of thumb which suggests that 5-10 samples are needed per parameter.

\begin{figure}
  \centering
  \includegraphics[width = 0.49\textwidth]{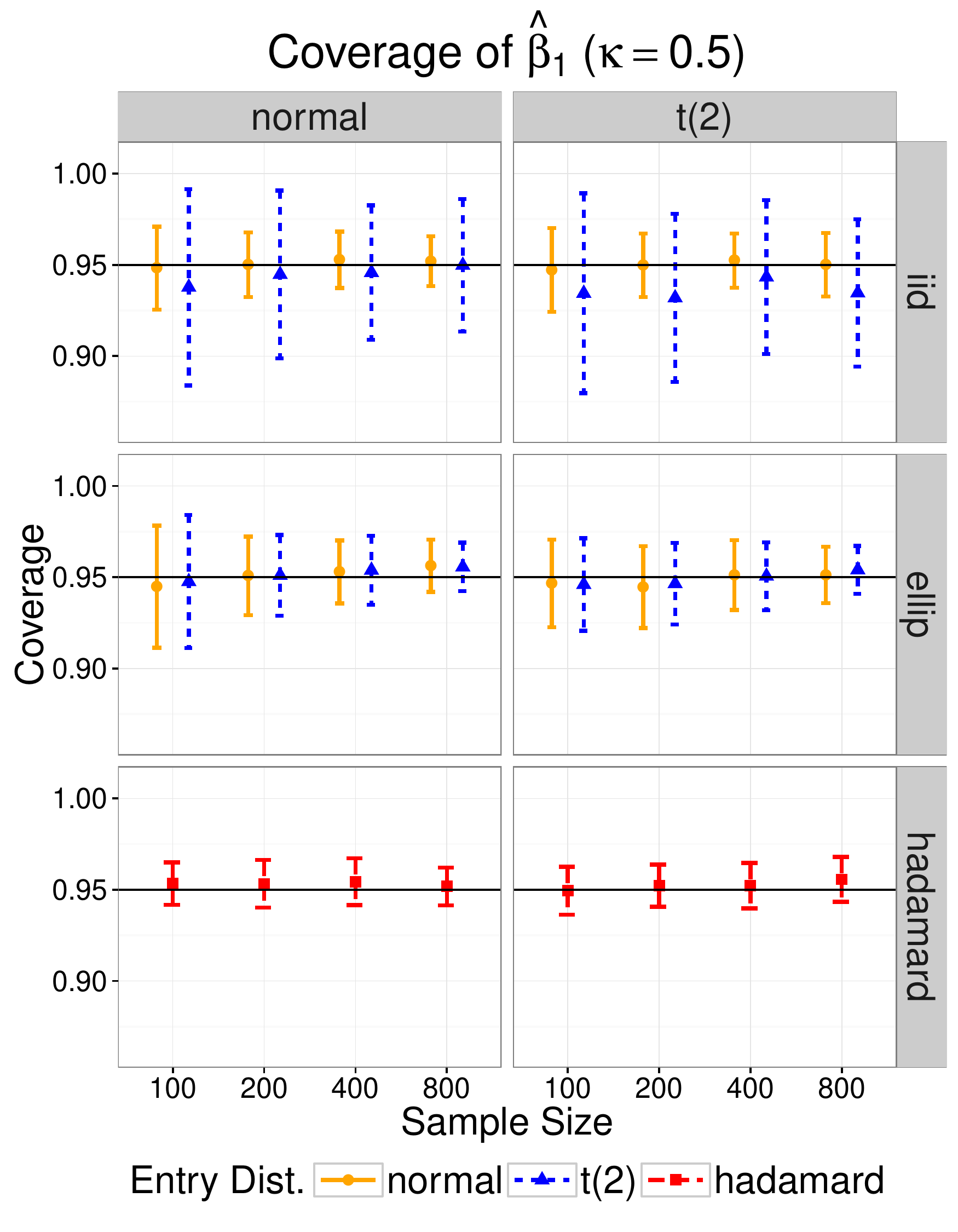}
  \includegraphics[width = 0.49\textwidth]{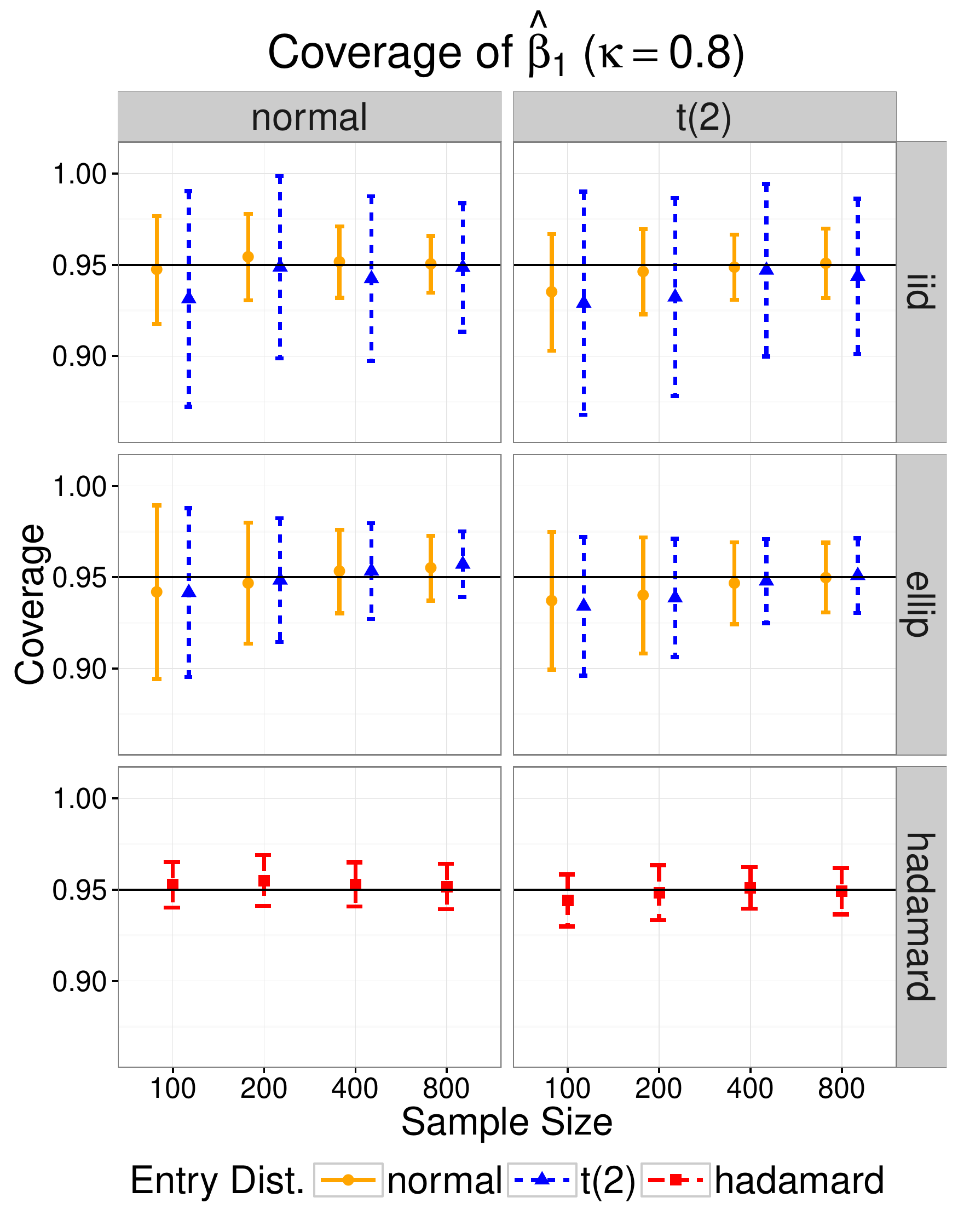}
  \caption{Empirical 95\% coverage of $\hat{\beta}_{1}$ with $\kappa = 0.5$ (left) and $\kappa = 0.8$ (right) using $\text{Huber}_{1.345}$ loss. The x-axis corresponds to the sample size, ranging from $100$ to $800$; the y-axis corresponds to the empirical 95\% coverage. Each column represents an error distribution and each row represents a type of design. The orange solid bar corresponds to the case $F = \text{Normal}$; the blue dotted bar corresponds to the case $F = \tdist_2$; the red dashed bar represents the Hadamard design. 
% The sample size is $n=100$. \redtext{$F$ is the $t_2$ distribution. In other words, in the ``i.i.d design'' the entries of $X$ are one draw from i.i.d $t_2$ random variables. In the  ``elliptical design'' cases, the elliptical factor is drawn i.i.d $t_2$, while  $\tilde{X_{ij}}$ correspond to one draw of ${\cal N}(0,1)$ random variables.} \sout{entry distribution $F$ is $\tdist(2)$ in i.i.d. design and elliptical design.} Each column represents an error distribution and each row represents a type of design. The x-axis corresponds to the quantile of standard normal distribution and the y-axis corresponds to the quantile of \redtext{normalized $\hat{\beta}_{1}$'s.}\sout{data. The red solid line is the $45^{\circ}$ line and the black dots represent $\hat{\beta}_{1}$ after normalizing and sorting.}
}\label{fig:singleCoord}
\end{figure}

\subsection{Asymptotic Normality for Multiple Marginals}
Since our theory holds for general $J_{n}$, it is worth checking the approximation for multiple coordinates in finite samples. For illustration, we consider 10 coordinates, namely $\hat{\beta}_{1} \sim \hat{\beta}_{10}$, simultaneously and calculate the minimum empirical 95\% coverage. To avoid the finite sample dependence between coordinates involved in the simulation, we estimate the empirical coverage independently for each coordinate. Specifically, we run 50 simulations with each consisting of the following steps:
\begin{enumerate}[(Step 1)]
\item Generate one design matrix $X$;
\item Generate the 3000 error vectors $\eps$;
\item Regress each $Y = \eps$ on the design matrix $X$ and end up with 300 random samples of $\hat{\beta}_{j}$ for each $j = 1, \ldots, 10$ by using the $(300(j - 1) + 1)$-th to $300j$-th response vector $Y$;
\item Estimate the standard deviation of $\hat{\beta}_{j}$ by the sample standard error $\widehat{\mathrm{sd}}_{j}$ for $j = 1, \ldots, 10$;
\item Construct a confidence interval $\mathcal{I}_{j}^{(k)} = \left[\hat{\beta}_{j}^{(k)} - 1.96\cdot  \widehat{\mathrm{sd}}_{j}, \hat{\beta}_{j}^{(k)} + 1.96 \cdot \widehat{\mathrm{sd}}_{j}\right]$ for each $j = 1, \ldots, 10$ and $k = 1, \ldots, 300$;
\item Calculate the empirical 95\% coverage by the proportion of confidence intervals which cover the true $\beta_{j} = 0$, denoted by $C_{j}$, for each $j = 1, \ldots, 10$, 
\item Report the minimum coverage $\min_{1\le j\le 10}C_{j}$.
\end{enumerate}

If the assumptions \textbf{A}1 - \textbf{A}5 are satisfied, $\min_{1\le j\le 10}C_{j}$ should also be close to 0.95 as a result of Theorem \ref{thm:main}. Thus, $\min_{1\le j\le 10}C_{j}$ is a measure for the approximation accuracy for multiple marginals. Figure \ref{fig:multipleCoord} displays the boxplots of this quantity under the same scenarios as the last subsection. In two cases that our theories cover, the minimum coverage is increasingly closer to the true level $0.95$. Similar to the last subsection, the approximation is accurate in the partial hadamard design case and is insensitive to the distribution of multiplicative factors in the elliptical design case. However, the approximation is very inaccurate in the i.i.d. $\tdist_2$ design case. Again, this shows the evidence that our technical assumptions are not artifacts of the proof.

On the other hand, the figure \ref{fig:multipleCoord} suggests using a conservative variance estimator, e.g. the Jackknife estimator, or corrections on the confidence level in order to make simultaneous inference on multiple coordinates. Here we investigate the validity of Bonferroni correction by modifying the step 5 and step 6. The confidence interval after Bonferroni correction is obtained by 
\begin{equation}\label{eq:Bonferroni}
\mathcal{I}_{j}^{(k)} = \left[\hat{\beta}_{j}^{(k)} - z_{1 - \alpha / 20}\cdot  \widehat{\mathrm{sd}}_{j}, \hat{\beta}_{j}^{(k)} + z_{1 - \alpha / 20} \cdot \widehat{\mathrm{sd}}_{j}\right]
\end{equation}
where $\alpha = 0.05$ and $z_{\gamma}$ is the $\gamma$-th quantile of a standard normal distribution. The proportion of $k$ such that $0\in\mathcal{I}_{j}^{(k)}$ for all $j\le 10$ should be at least $0.95$ if the marginals are all close to a normal distribution. We modify the confidence intervals in step 5 by \eqref{eq:Bonferroni} and calculate the proportion of $k$ such that $0\in\mathcal{I}_{j}^{(k)}$ for all $j$ in step 6. Figure \ref{fig:multipleCoordBon} displays the boxplots of this coverage. It is clear that the Bonferroni correction gives the valid coverage except when $n = 100, \kappa = 0.8$ and the error has a $\tdist_2$ distribution.

\section{Conclusion}\label{sec:discuss}
We have proved coordinate-wise asymptotic normality for regression M-estimates in the moderate-dimensional asymptotic regime $p / n\rightarrow \kappa \in (0, 1)$, for fixed design matrices under appropriate technical assumptions. Our design assumptions are satisfied with high probability for a broad class of random designs. The main novel ingredient of the proof is the use of the second-order Poincar\'e inequality. Numerical experiments confirm and complement our theoretical results. 

\begin{figure}
  \centering
  \includegraphics[width = 0.49\textwidth]{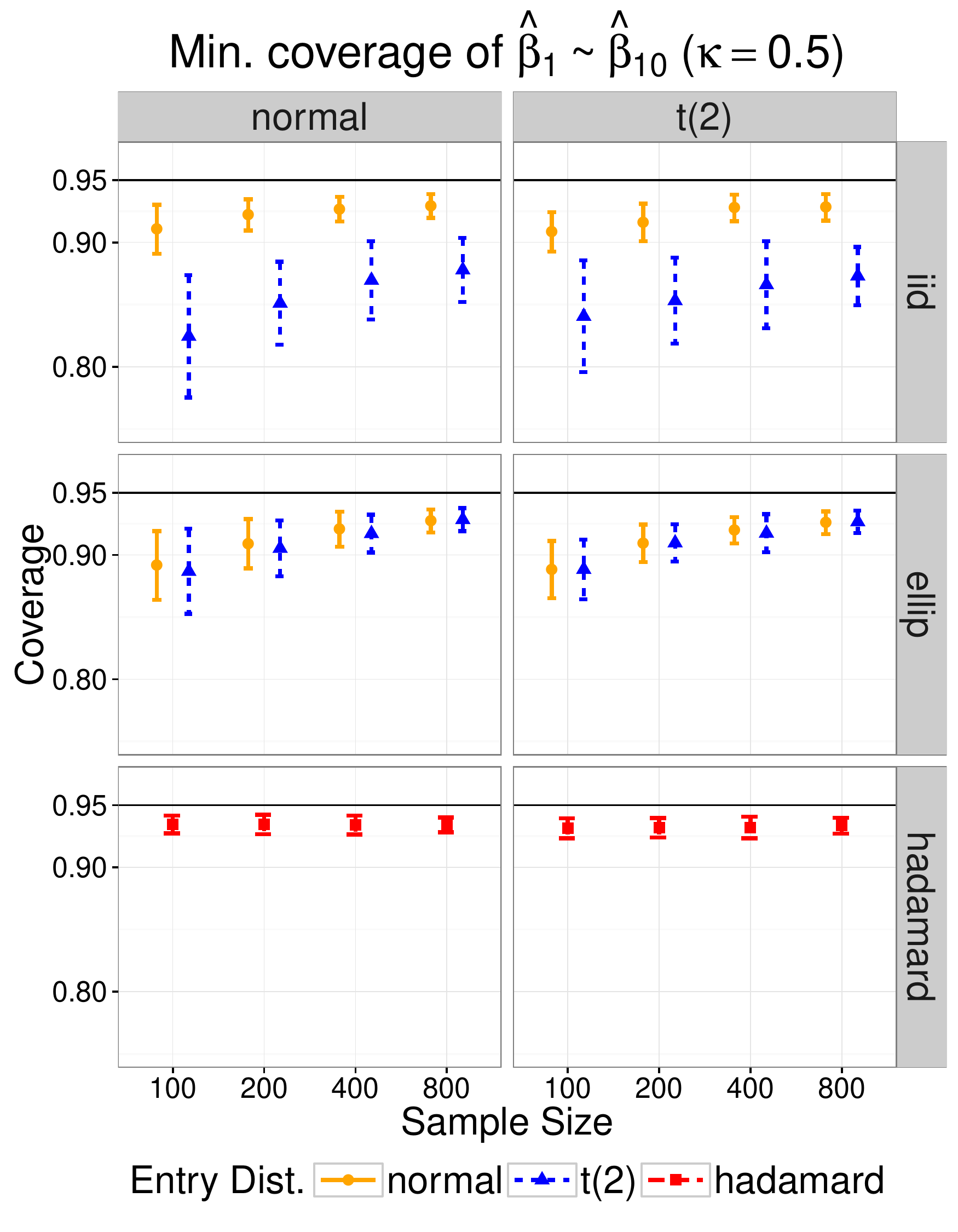}
  \includegraphics[width = 0.49\textwidth]{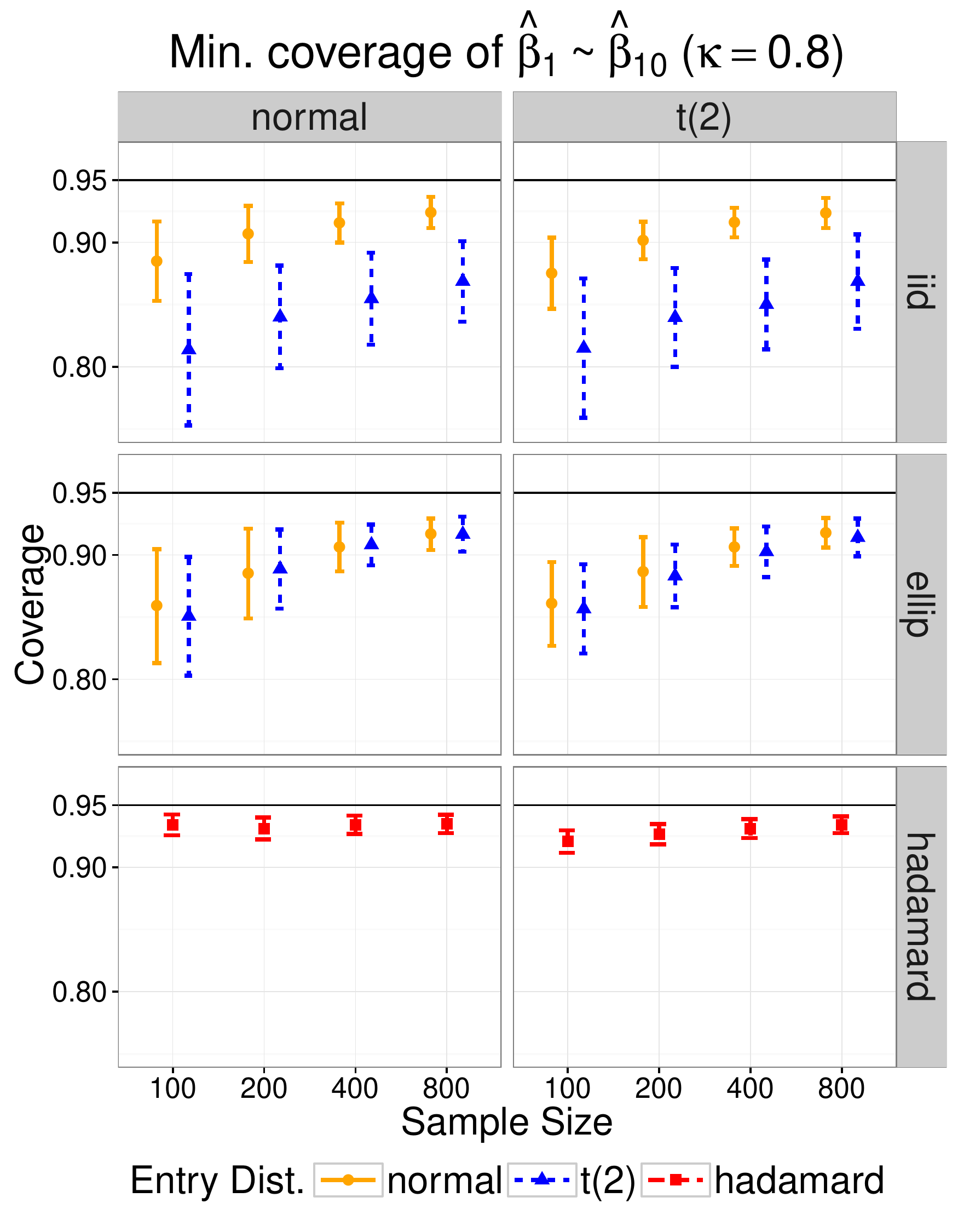}
  \caption{Mininum empirical 95\% coverage of $\hat{\beta}_{1} \sim\hat{\beta}_{10}$ with $\kappa = 0.5$ (left) and $\kappa = 0.8$ (right) using $\text{Huber}_{1.345}$ loss. The x-axis corresponds to the sample size, ranging from $100$ to $800$; the y-axis corresponds to the minimum empirical 95\% coverage. Each column represents an error distribution and each row represents a type of design. The orange solid bar corresponds to the case $F = \text{Normal}$; the blue dotted bar corresponds to the case $F = \tdist_2$; the red dashed bar represents the Hadamard design. 
% The sample size is $n=100$. \redtext{$F$ is the $t_2$ distribution. In other words, in the ``i.i.d design'' the entries of $X$ are one draw from i.i.d $t_2$ random variables. In the  ``elliptical design'' cases, the elliptical factor is drawn i.i.d $t_2$, while  $\tilde{X_{ij}}$ correspond to one draw of ${\cal N}(0,1)$ random variables.} \sout{entry distribution $F$ is $\tdist(2)$ in i.i.d. design and elliptical design.} Each column represents an error distribution and each row represents a type of design. The x-axis corresponds to the quantile of standard normal distribution and the y-axis corresponds to the quantile of \redtext{normalized $\hat{\beta}_{1}$'s.}\sout{data. The red solid line is the $45^{\circ}$ line and the black dots represent $\hat{\beta}_{1}$ after normalizing and sorting.}
}\label{fig:multipleCoord}
\end{figure}

\begin{figure}[H]
  \centering
  \includegraphics[width = 0.49\textwidth]{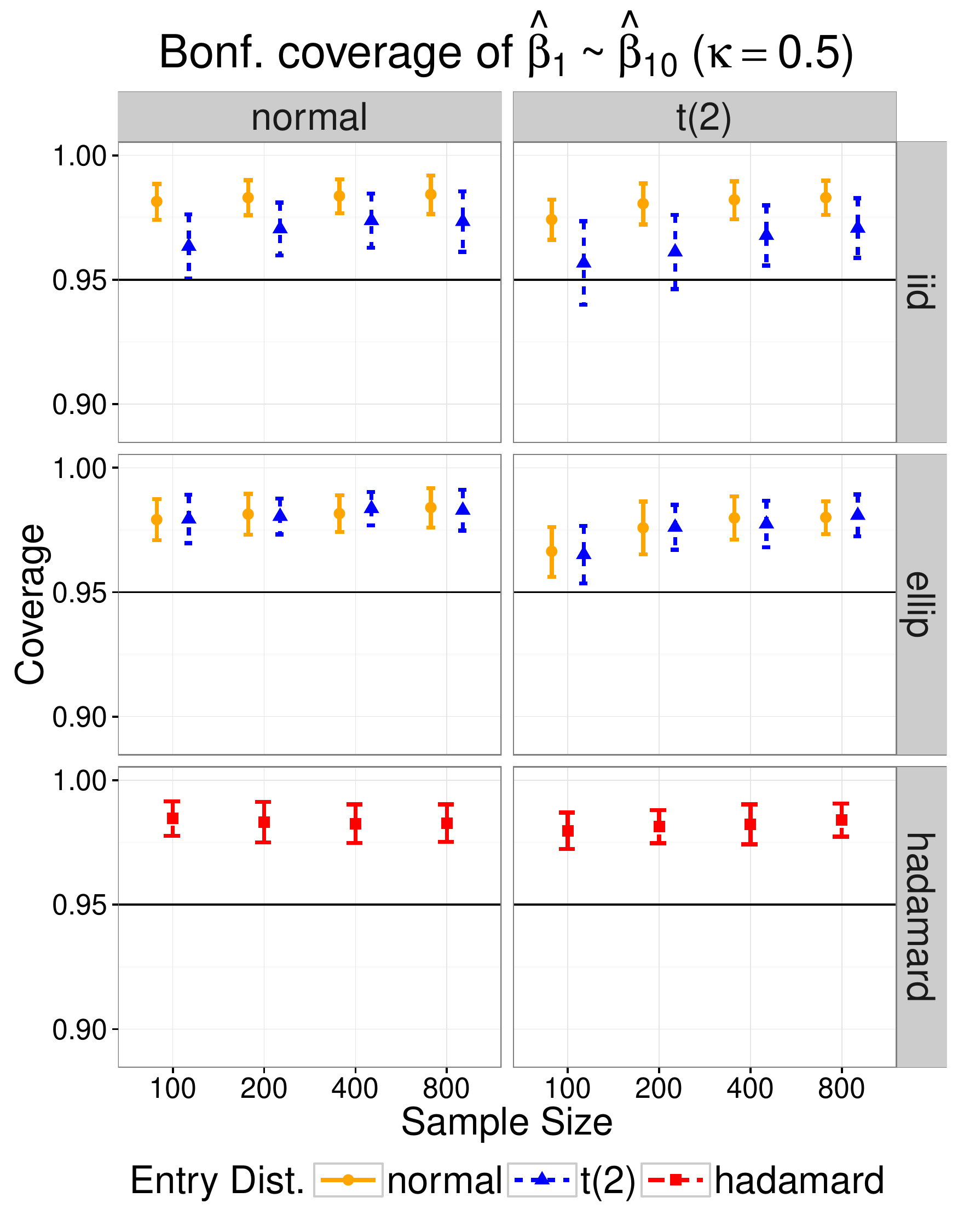}
  \includegraphics[width = 0.49\textwidth]{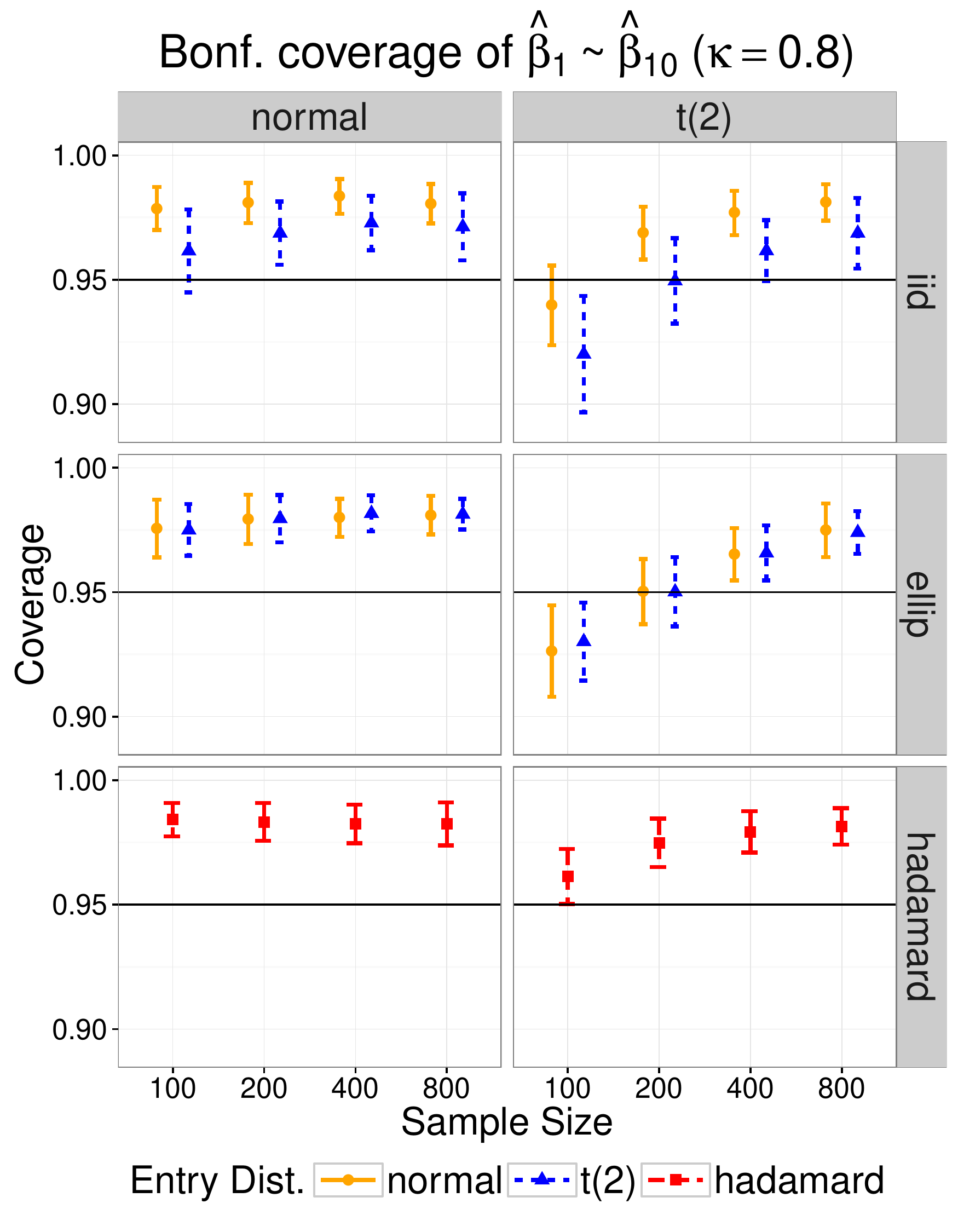}
  \caption{Empirical 95\% coverage of $\hat{\beta}_{1} \sim \hat{\beta}_{10}$ after Bonferroni correction with $\kappa = 0.5$ (left) and $\kappa = 0.8$ (right) using $\text{Huber}_{1.345}$ loss. The x-axis corresponds to the sample size, ranging from $100$ to $800$; the y-axis corresponds to the empirical uniform 95\% coverage after Bonferroni correction. Each column represents an error distribution and each row represents a type of design. The orange solid bar corresponds to the case $F = \text{Normal}$; the blue dotted bar corresponds to the case $F = \tdist_2$; the red dashed bar represents the Hadamard design. 
% The sample size is $n=100$. \redtext{$F$ is the $t_2$ distribution. In other words, in the ``i.i.d design'' the entries of $X$ are one draw from i.i.d $t_2$ random variables. In the  ``elliptical design'' cases, the elliptical factor is drawn i.i.d $t_2$, while  $\tilde{X_{ij}}$ correspond to one draw of ${\cal N}(0,1)$ random variables.} \sout{entry distribution $F$ is $\tdist(2)$ in i.i.d. design and elliptical design.} Each column represents an error distribution and each row represents a type of design. The x-axis corresponds to the quantile of standard normal distribution and the y-axis corresponds to the quantile of \redtext{normalized $\hat{\beta}_{1}$'s.}\sout{data. The red solid line is the $45^{\circ}$ line and the black dots represent $\hat{\beta}_{1}$ after normalizing and sorting.}
}\label{fig:multipleCoordBon}
\end{figure}

\bibliographystyle{apacite}

\bibliography{high_dim_robust_reg}

\newpage
	\appendix
	\begin{center}
          \begin{LARGE}
	\textbf{\textsc{APPENDIX}}
        \end{LARGE}
	\end{center}
%	\vspace{1cm}
	\renewcommand{\thesection}{\Alph{section}}
	\renewcommand{\theequation}{\Alph{section}-\arabic{equation}}
	\renewcommand{\thelemma}{\Alph{section}-\arabic{lemma}}
	\renewcommand{\thecorollary}{\Alph{section}-\arabic{corollary}}
	\renewcommand{\thesubsection}{\Alph{section}-\arabic{subsection}}
	\renewcommand{\thesubsubsection}{\Alph{section}-\arabic{subsection}.\arabic{subsubsection}}
	  % redefine the command that creates the equation no.
	\setcounter{equation}{0}  % reset counter
	\setcounter{section}{0}

\appendix
\section{Proof Sketch of Lemma \ref{lem:heuristic}}\label{app:heuristic}
In this Appendix, we provide a roadmap for proving Lemma \ref{lem:heuristic} by considering a special case where $X$ is one realization of a random matrix $Z$ with i.i.d. mean-zero $\sigma^{2}$-sub-gaussian entries. Random matrix theory \cite{geman80,silverstein85,bai93} implies that $\lammax = (1 + \sqrt{\kappa})^{2} + o_{p}(1) = O_{p}(1)$ and $\lammin = (1 - \sqrt{\kappa})^{2} + o_{p}(1) = \Omega_{p}(1)$. Thus, the assumption \textbf{A}3 is satisfied with high probability. Thus, the Lemma \ref{lem:tv_bound} in p. \pageref{lem:tv_bound} holds with high probability. It remains to prove the following lemma to obtain Theorem \ref{thm:main}.
\begin{lemma}
  Let $Z$ be a random matrix with i.i.d. mean-zero $\sigma^{2}$-sub-gaussian entries and $X$ be one realization of $Z$. Then under assumptions \textbf{A}1 and \textbf{A}2,
\[\max_{1\le j\le p}M_{j} = O_{p}\lb\frac{\polyLog}{n}\rb, \quad \min_{1\le j\le p}\Var(\hat{\beta}_{j}) = \Omega_{p}\lb\frac{1}{n\cdot \polyLog}\rb,\]
where $M_{j}$ is defined in \eqref{eq:def_Mj} in p.\pageref{eq:def_Mj} and the randomness in $o_{p}(\cdot)$ and $O_{p}(\cdot)$ comes from $Z$. 
\end{lemma}

\subsection{Upper Bound of $M_{j}$}\label{subsec:Mj}
First by Proposition \ref{prop:rmt_bai93}, 
\[\lammax = O_{p}(1), \quad \lammin = \Omega_{p}(1).\]
In the rest of the proof, the symbol $\E$ and $\Var$ denotes the expectation and the variance conditional on $Z$. Let $\td{Z} = D^{\frac{1}{2}}Z$, then $M_{j} = \E\|e_{j}^{T}(\td{Z}^{T}\td{Z})^{-1}\td{Z}^{T}\|_{\infty}$. Let $\td{H}_{j} = I - \td{Z}_{[j]}(\td{Z}_{[j]}^{T}\td{Z}_{[j]})^{-1}\td{Z}_{[j]}^{T}$, then by block matrix inversion formula (see Proposition \ref{prop:block_inv}), which we state as Proposition \ref{prop:block_inv} in Appendix \ref{app:mc}.
\begin{align*}
(\td{Z}^{T}\td{Z})^{-1}\td{Z}^{T} & = \lb
\begin{array}{cc}
  \td{Z}_{1}^{T}\td{Z}_{1}& \td{Z}_{1}^{T}\td{Z}_{[1]}\\
  \td{Z}_{[1]}^{T}\td{Z}_{1} & \td{Z}_{[1]}^{T}\td{Z}_{[1]}
\end{array}
\rb^{-1}\com{\td{Z}_{1}}{\td{Z}_{[1]}}\\
& =  \frac{1}{\td{Z}_{1}^{T}(I - \td{H}_{1})\td{Z}_{1}}\lb
\begin{array}{cc}
  1 & -\td{Z}_{1}^{T}\td{Z}_{[1]}(\td{Z}_{[1]}^{T}\td{Z}_{[1]})^{-1}\\
  * & *
\end{array}
\rb\com{\td{Z}_{1}}{\td{Z}_{[1]}}\\
& = \frac{1}{\td{Z}_{1}^{T}(I - \td{H}_{1})\td{Z}_{1}}\com{\td{Z}_{1}^{T}(I - \td{H}_{1})}{*}.
\end{align*}
This implies that
\begin{equation}\label{eq:M1}
M_{1} = \E\frac{\|\td{Z}_{1}^{T}(I - \td{H}_{1})\|_{\infty}}{\td{Z}_{1}^{T}(I - \td{H}_{1})\td{Z}_{1}}.
\end{equation}
Since $Z^{T}DZ / n\succeq K_{0}\lammin I$, we have
\[\frac{1}{\td{Z}_{1}^{T}(I - \td{H}_{1})\td{Z}_{1}} = e_{1}^{T}(\td{Z}^{T}\td{Z})^{-1}e_{1} = e_{1}^{T}(Z^{T}DZ)^{-1}e_{1} = \frac{1}{n}e_{1}^{T}\lb\frac{Z^{T}DZ}{n}\rb^{-1}e_{1} \le \frac{1}{nK_{0}\lammin}\]
and we obtain a bound for $M_{1}$ as
\[M_{1} \le \frac{\E \|\td{Z}_{1}^{T}(I - \td{H}_{1})\|_{\infty}}{nK_{0}\lammin} = \frac{\E \|Z_{1}^{T}D^{\frac{1}{2}}(I - \td{H}_{1})\|_{\infty}}{nK_{0}\lammin}.\]
Similarly,
\begin{equation}\label{eq:Mj_bound}
M_{j} \le  \frac{\E \|Z_{j}^{T}D^{\frac{1}{2}}(I - \td{H}_{j})\|_{\infty}}{nK_{0}\lammin}= \frac{\E\|Z_{j}^{T}D^{\frac{1}{2}}(I - D^{\frac{1}{2}}Z_{[j]}^{T}(Z_{[j]}^{T}DZ_{[j]})^{-1}Z_{[j]}D^{\frac{1}{2}})\|_{\infty}}{nK_{0}\lammin}.
\end{equation}
The vector in the numerator is a linear contrast of $Z_{j}$ and $Z_{j}$ has mean-zero i.i.d. sub-gaussian entries. For any fixed matrix $A\in \R^{n\times n}$, denote $A_{k}$ by its $k$-th column, then $A_{k}^{T}Z_{j}$ is $\sigma^{2}\|A_{k}\|_{2}^{2}$-sub-gaussian (see Section 5.2.3 of \citeA{vershynin10} for a detailed discussion) and hence by definition of sub-Gaussianity, 
\[P(|A_{k}^{T}Z_{j}|\ge \sigma\|A_{k}\|_{2}t)\le 2e^{-\frac{t^{2}}{2}}.\]
Therefore, by a simple union bound, we conclude that 
\[P(\|A^{T}Z_{j}\|_{\infty}\ge \sigma\max_{k}\|A_{k}\|_{2}t)\le 2ne^{-\frac{t^{2}}{2}}.\]
Let $t = 2\sqrt{\log n}$,
 \[P(\|A^{T}Z_{j}\|_{\infty}\ge 2\sigma\max_{k}\|A_{k}\|_{2}\sqrt{\log n})\le \frac{2}{n} = o(1).\]
This entails that 
\begin{equation}\label{eq:subgauss_max_bound}
\|A^{T}Z_{j}\|_{\infty} = O_{p}\lb \max_{k}\|A_{k}\|_{2}\cdot \polyLog\rb = O_{p}\lb \|A\|_{\mathrm{op}}\cdot \polyLog\rb.
\end{equation}
with high probability. In $M_{j}$, the coefficient matrix $(I - H_{j})D^{\frac{1}{2}}$ depends on $Z_{j}$ through $D$ and hence we cannot use (\ref{eq:subgauss_max_bound}) directly. However, the dependence can be removed by replacing $D$ by $D_{[j]}$ since  $r_{i, [j]}$ does not depend on $Z_{j}$. 

Since $Z$ has i.i.d. sub-gaussian entries, no column is highly influential. In other words, the estimator will not change drastically after removing $j$-th column. This would suggest $R_{i}\approx r_{i, [j]}$. It is proved by \citeA{elkaroui13} that 
\[\sup_{i, j}|R_{i} - r_{i, [j]}| = O_{p}\lb{\frac{\polyLog}{\sqrt{n}}}\rb.\]
It can be rigorously proved that
\[\big|\|Z_{j}^{T}D(I - \td{H}_{j})\|_{\infty} -  \|Z_{j}^{T}D_{[j]}(I - H_{j})\|_{\infty}\big| = O_{p}\lb\frac{\polyLog}{n}\rb,\]
where $H_{j} = I - D_{[j]}^{\frac{1}{2}}Z_{[j]}(Z_{[j]}^{T}D_{[j]}Z_{[j]})^{-1}Z_{[j]}^{T}D_{[j]}^{\frac{1}{2}}$; see Appendix \ref{subsec:Mj} for details. Since $D_{[j]}(I - H_{j})$ is independent of $Z_{j}$ and 
\[\|D_{[j]}(I - H_{j})\|_{\mathrm{op}} \le \|D_{[j]}\|_{\mathrm{op}}\le K_{1} = O\lb\polyLog\rb, \]
it follows from \eqref{eq:Mj_bound} and \eqref{eq:subgauss_max_bound} that
\[\|Z_{j}^{T}D_{[j]}(I - H_{j})\|_{\infty} = O_{p}\lb\frac{\polyLog}{n}\rb.\]
In summary, 
% Since $Z_{j}$ has i.i.d. mean-zero sub-gaussian entries, standard result implies that 
% \[\E_{Z_{j}}\max_{k}|e_{k}^{T}(I - H_{j})D_{[j]}Z_{j}|\le \max_{k}\|e_{k}^{T}D_{[j]}(I - H_{j})\|_{2}\sqrt{2\log 2n}.\]
% Since $I - H_{j}\preceq I$, 
% \[\|e_{k}^{T}D_{[j]}(I - H_{j})\|_{2} = \sqrt{e_{k}^{T}D_{[j]}(I - H_{j})^{2}D_{[j]}e_{k}}\le \sqrt{e_{k}^{T}D_{[j]}^{2}e_{k}}\le K_{1}.\]
% In summary, 
% \[M_{j} = \td{O}_{p}\lb\frac{1}{n}\rb.\]
\begin{equation}\label{eq:upper_Mj}
M_{j} = O_{p}\lb\frac{\polyLog}{n}\rb.
\end{equation}

\subsection{Lower Bound of $\Var(\hat{\beta}_{j})$}
% From Lemma \ref{lem:tv_bound} and (\ref{eq:upper_Mj}), we know that 
% \begin{equation}\label{eq:sopi+upper}
% \max_{j} d_{TV}\lb\mathcal{L}\lb\frac{\hat{\beta}_{j} - \E\hat{\beta}_{j}}{\sqrt{\Var(\hat{\beta}_{j})}}\Bigg|Z\rb, N(0, 1)\rb = \td{O}_{p}\lb\frac{1}{n^{\frac{9}{8}}\Var(\hat{\beta}_{j})}\rb.
% \end{equation}
% It is left to establish a lower bound for $\Var(\hat{\beta}_{j})$. In fact we can prove the following result
% \begin{equation}\label{eq:lower_var}
% \min_{j}\Var(\hat{\beta}_{j}) = \td{\Omega}_{p}\lb\frac{1}{n}\rb.
% \end{equation}
% Corollary \ref{cor:toy} is then implied by (\ref{eq:sopi+upper}) and \eqref{eq:lower_var}. Since this is the most tricky and technical part of the proof, we sketch the heuristic arguments in this subsection to illustrate main ideas. 
\subsubsection{Approximating $\Var(\hat{\beta}_{j})$ by $\Var(b_{j})$}\label{subsubsec:approx_varbetaj_varbj}
It is shown by \citeA{elkaroui13}\footnote{\citeA{elkaroui13} considers a ridge regularized M estimator, which is different from our setting. However, this argument still holds in our case and proved in Appendix \ref{app:main}.} that 
\begin{equation}\label{eq:rough_approx_bj}
\hat{\beta}_{j} \approx b_{j} \triangleq \frac{1}{\sqrt{n}}\frac{N_{j}}{\xi_{j}}
\end{equation}
where
\[N_{j} = \frac{1}{\sqrt{n}}\sum_{i=1}^{n}Z_{ij}\psi(r_{i, [j]}), \quad \xi_{j} = \frac{1}{n}Z_{j}^{T}(D_{[j]} - D_{[j]}Z_{[j]}(\wmj)^{-1}Z_{[j]}^{T}D_{[j]})Z_{j}.\]
% To gain intuition about the above approximation, we briefly go through the derivation for least-squares and $j = 1$. Let 
% \[H_{1} = I - Z_{[1]}(Z_{[1]}^{T}Z_{[1]})^{-1}Z_{[1]}^{T},\]
% then by block matrix inversion formula,
% \begin{align*}
% \hat{\beta}_{1} &= e_{1}^{T}(Z^{T}Z)^{-1}Z^{T}\eps\\
% & = e_{1}^{T}\lb
% \begin{array}{cc}
%   Z_{1}^{T}Z_{1}  & Z_{1}^{T}Z_{[1]}\\
%   Z_{[1]}^{T}Z_{1} & Z_{[1]}^{T}Z_{[1]}
% \end{array}
% \rb^{-1}\com{Z_{1}^{T}}{Z_{[1]}^{T}}\eps\\
% & = \frac{1}{Z_{1}^{T}(I - H_{1})Z_{1}}\cdot (1\,\,\,\, -Z_{1}^{T}Z_{[1]}(Z_{[1]}^{T}Z_{[1]})^{-1})\com{Z_{1}^{T}}{Z_{[1]}^{T}}\eps\\
% & = \frac{Z_{1}^{T}(I - H_{1})\eps}{Z_{1}^{T}(I - H_{1})Z_{1}} = \frac{\sum_{i=1}^{n}Z_{i1}\psi(r_{i, [1]})}{Z_{1}^{T}(I - H_{1})Z_{1}} = b_{1}.
% \end{align*}
% This means that the approximation \eqref{eq:rough_approx_bj} holds exactly in the least-squares case. For general M estimator, 
It has been shown by \citeA{elkaroui13} that 
\[\max_{j}|\hat{\beta}_{j} - b_{j}| = O_{p}\lb\frac{\polyLog}{n}\rb.\]
Thus, $\Var(\hat{\beta}_{j})\approx \Var(b_{j})$ and a more refined calculation in Appendix \ref{subsubsec:approx_varbetaj_varbj} shows that
\[|\Var(\hat{\beta}_{j}) - \Var(b_{j})| = O_{p}\lb\frac{\polyLog}{n^{\frac{3}{2}}}\rb.\]
It is left to show that 
\begin{equation}\label{eq:app_A_var_bj}
\Var(b_{j}) = \Omega_{p}\lb\frac{1}{n\cdot \polyLog}\rb.
\end{equation}

\subsubsection{Bounding $\Var(b_{j})$ via $\Var(N_{j})$}
By definition of $b_{j}$, 
\[\Var(b_{j}) = \Omega_{p}\lb\frac{\polyLog}{n}\rb\Longleftrightarrow \Var\lb\frac{N_{j}}{\xi_{j}}\rb = \Omega_{p}\lb \polyLog\rb.\]
% Note that
% \begin{align*}
%   \Var\lb \frac{N_{j}}{\xi_{j}}\rb = \E\lb \frac{N_{j}}{\xi_{j}} - \E \frac{N_{j}}{\xi_{j}}\rb^{2} = \E \lb \frac{N_{j} - \E N_{j}}{\xi_{j}} + (\E N_{j})\lb\frac{1}{\xi_{j}} - \E \frac{1}{\xi_{j}}\rb + \E N_{j} \E \frac{1}{\xi_{j}} - \E \frac{N_{j}}{\xi_{j}}\rb^{2}
% \end{align*}
As will be shown in Appendix \ref{subsubsec:I2}, 
\[\Var(\xi_{j}) = O_{p}\lb\frac{\polyLog}{n}\rb.\]
As a result, $\xi_{j}\approx \E \xi_{j}$ and 
\[\Var\lb \frac{N_{j}}{\xi_{j}}\rb\approx \Var\lb\frac{N_{j}}{\E \xi_{j}}\rb = \frac{\Var(N_{j})}{(\E \xi_{j})^{2}}.\]
% \[\bigg|\Var\lb \frac{N_{j}}{\xi_{j}}\rb - \E \lb \frac{N_{j} - \E N_{j}}{\xi_{j}}\rb^{2}\bigg| = O_{p}\lb\frac{\polyLog}{n^{2}}\rb.\]
% Thus, it is left to show that 
% \[\E \lb \frac{N_{j} - \E N_{j}}{\xi_{j}}\rb^{2} = \Omega_{p}\lb\frac{1}{\polyLog}\rb.\]
As in the previous paper \cite{elkaroui13}, we rewrite $\xi_{j}$ as
\[\xi_{j} =  \frac{1}{n}Z_{j}^{T}D_{[j]}^{\frac{1}{2}}(I - D_{[j]}^{\frac{1}{2}}Z_{[j]}(\wmj)^{-1}Z_{[j]}^{T}D_{[j]}^{\frac{1}{2}})D_{[j]}^{\frac{1}{2}}Z_{j}.\]
The middle matrix is idempotent and hence positive semi-definite. Thus, 
\begin{align*}
\xi_{j} & \le \frac{1}{n}Z_{j}^{T}D_{[j]}Z_{j} \le K_{1}\lammax = O_{p}\lb\polyLog\rb.
\end{align*}
Then we obtain that
\[\frac{\Var(N_{j})}{(\E \xi_{j})^{2}} = \Omega_{p}\lb \frac{\Var(N_{j})}{\polyLog}\rb,\]
and it is left to show that 
\begin{equation}\label{eq:app_A_var_Nj}
\Var(N_{j}) = \Omega_{p}\lb\frac{1}{\polyLog}\rb.
\end{equation}

\subsubsection{Bounding $\Var(N_{j})$ via $\tr(Q_{j})$}
Recall the definition of $N_{j}$ \eqref{eq:rough_approx_bj}, and that of $Q_{j}$ (see Section \ref{subsec:notation} in p.\pageref{subsec:notation}), we have
\[\Var(N_{j}) = \frac{1}{n}Z_{j}^{T}Q_{j}Z_{j}\]
Notice that $Z_{j}$ is independent of $r_{i, [j]}$ and hence the conditional distribution of $Z_{j}$ given $Q_{j}$ remains the same as the marginal distribution of $Z_{j}$. Since $Z_{j}$ has i.i.d. sub-gaussian entries, the Hanson-Wright inequality (\citeNP{hanson71, rudelson13}; see Proposition \ref{prop:hanson_wright}), shown in Proposition \ref{prop:hanson_wright}, implies that any quadratic form of $Z_{j}$, denoted by $Z_{j}^{T}Q_{j}Z_{j}$ is concentrated on its mean, i.e.
% \[P\lb \bigg|Z_{j}^{T}Q_{j}Z_{j} - \E Z_{j}^{T}Q_{j}Z_{j}\bigg|\ge \sigma^{2}t\rb\le 2\exp\{-c\min\left\{\frac{t^{2}}{\|Q_{j}\|_{\mathrm{F}}^{2}}, \frac{t}{\|Q_{j}\|_{\mathrm{op}}}\right\}\}\]
% for some universal constant $c$. 
\[Z_{j}^{T}Q_{j}Z_{j}\approx \E_{\tiny Z_{j}, \eps}Z_{j}^{T}Q_{j}Z_{j} = (\E Z_{1j}^{2}) \cdot \tr(Q_{j}).\]
% this entails that 
% \[Z_{j}^{T}AZ_{j} = \E_{Z_{j}} Z_{j}^{T}AZ_{j} + O_{p}(\|A\|_{\mathrm{F}}) = \tau^{2}\tr(A) + O_{p}(\|A\|_{\mathrm{F}}),\]
% where $\tau^{2} = \Var(Z_{ij})$. Apply the result to $\Var(N_{j})$, we have
% \[\Var(N_{j}) = \tau^{2}\cdot \frac{\tr(Q_{j})}{n} + O_{p}\lb\frac{\|Q_{j}\|_{\mathrm{F}}}{n}\rb.\]
% It can be proved that 
% \[\|Q_{j}\|_{\mathrm{F}} = O_{p}(\sqrt{n}\cdot \polyLog),\]
As a consequence, it is left to show that 
\begin{equation}\label{eq:app_A_tr_Qj}
\tr(Q_{j}) = \Omega_{p}\lb\frac{n}{\polyLog}\rb.
\end{equation}

\subsubsection{Lower Bound of $\tr(Q_{j})$}
By definition of $Q_{j}$, 
\[\tr(Q_{j}) = \sum_{i=1}^{n}\Var(\psi(r_{i, [j]})).\]
To lower bounded the variance of $\psi(r_{i, [j]})$, recall that for any random variable $W$,
\begin{equation}\label{eq:varformula}
\Var(W) = \frac{1}{2}\E (W - W')^{2}.
\end{equation}
where $W'$ is an independent copy of $W$. Suppose $g: \R\rightarrow \R$ is a function such that $|g'(x)|\ge c$ for all $x$, then (\ref{eq:varformula}) implies that
\begin{equation}\label{eq:varfunc}
\Var(g(W)) = \frac{1}{2}\E (g(W) - g(W'))^{2} \ge \frac{c^{2}}{2}\E (W - W')^{2} = c^{2}\Var(W).
\end{equation}
% The following lemma summarizes the results.
% \begin{lemma}\label{lem:varfunc}
% Suppose $g: \R\rightarrow \R$ is a function such that $|g'(x)|\ge c$ for all $x$, then for any random variable $W$, 
% \[\Var(g(W))\ge c^{2}\Var(W).\]
% \end{lemma}

In other words, (\ref{eq:varfunc}) entails that $\Var(W)$ is a lower bound for $\Var(g(W))$ provided that the derivative of $g$ is bounded away from 0. As an application, we see that
\[\Var(\psi(r_{i, [j]}))\ge K_{0}^{2}\Var(r_{i, [j]})\]
and hence
\[\tr(Q_{j}) \ge K_{0}^{2}\sum_{i=1}^{n}\Var(r_{i, [j]}).\]
By the variance decomposition formula, 
\begin{align*}
\Var(r_{i, [j]}) &= \E \lb\Var\lb r_{i, [j]} \big| \eps_{(i)}\rb\rb + \Var\lb\E \lb r_{i, [j]}\big| \eps_{(i)}\rb\rb \ge \E \lb\Var\lb r_{i, [j]} \big| \eps_{[i]}\rb\rb,
\end{align*}
where $\eps_{(i)}$ includes all but $i$-th entry of $\eps$. Given $\eps_{(i)}$, $r_{i, [j]}$ is a function of $\eps_{i}$. Using \eqref{eq:varfunc}, we have 
\[\Var(r_{i, [j]} | \eps_{(i)})\ge \inf_{\eps_{i}}\bigg|\pd{r_{i, [j]}}{\eps_{i}}\bigg|^{2}\cdot \Var (\eps_{i} | \eps_{(i)})\ge \inf_{\eps_{i}}\bigg|\pd{r_{i, [j]}}{\eps_{i}}\bigg|^{2}\cdot \Var (\eps_{i}).\]
This implies that 
\[\Var(r_{i, [j]})\ge \E \lb\Var\lb r_{i, [j]} \big| \eps_{[i]}\rb\rb \ge \E\inf_{\eps}\bigg|\pd{r_{i, [j]}}{\eps_{i}}\bigg|^{2}\cdot \min_{i}\Var (\eps_{i}).\]
Summing $\Var(r_{i, [j]})$ over $i = 1, \ldots, n$, we obtain that 
\[\tr(Q_{j}) = \sum_{i=1}^{n}\Var(r_{i, [j]})\ge \E\lb\sum_{i}\inf_{\eps}\bigg|\pd{r_{i, [j]}}{\eps_{i}}\bigg|^{2}\rb\cdot \min_{i}\Var (\eps_{i}).\]
It will be shown in Appendix \ref{subsubsec:I1} that under assumptions \textbf{A}1-\textbf{A}3, 
\begin{equation}
  \label{eq:Rideriv}
\E\sum_{i}\inf_{\eps}\bigg|\pd{r_{i, [j]}}{\eps_{i}}\bigg|^{2} = \Omega_{p}\lb\frac{n}{\polyLog}\rb.
\end{equation}
This proves \eqref{eq:app_A_tr_Qj} and as a result,
% \[\tr(Q_{j}) = \Omega_{p}\lb\frac{n}{\polyLog}\rb.\]
% putting pieces together, we conclude that
\[\min_{j}\Var(\hat{\beta}_{j}) = \Omega_{p}\lb\frac{1}{n\cdot \polyLog}\rb.\]

\section{Proof of Theorem \ref{thm:main}}\label{app:main}

\subsection{Notation}\label{app:notation}
To be self-contained, we summarize our notations in this subsection. The model we considered here is 
\[y = X\betanull + \eps\]
where $X\in \R^{n\times p}$ be the design matrix and $\eps$ is a random vector with independent entries. Notice that the target quantity $\frac{\hat{\beta}_{j} - \E \hat{\beta}_{j}}{\sqrt{\Var(\hat{\beta}_{j})}}$ is shift invariant, we can assume $\betanull = 0$ without loss of generality provided that $X$ has full column rank; see Section \ref{subsec:notation} for details.

Let $x_{i}^{T}\in \R^{1\times p}$ denote the $i$-th row of $X$ and $X_{j}\in \R^{n\times 1}$ denote the $j$-th column of X. Throughout the paper we will denote by $X_{ij}\in \R$ the $(i, j)$-th entry of $X$, by $X_{(i)}\in \R^{(n-1)\times p}$ the design matrix $X$ after removing the $i$-th row, by $X_{[j]}\in \R^{n\times (p-1)}$ the design matrix $X$ after removing the $j$-th column, by $X_{(i), [j]}\in \R^{(n-1)\times (p-1)}$ the design matrix after removing both $i$-th row and $j$-th column, and by $x_{i, [j]}\in \R^{1\times (p-1)}$ the vector $x_{i}$ after removing $j$-th entry. The M-estimator $\hat{\beta}$ associated with the loss function $\rho$ is defined as
\begin{equation}\label{eq:hatbeta}
\hat{\beta} = \argmin_{\beta\in\R^{p}}\frac{1}{n}\sum_{k=1}^{n}\rho(\eps_{k} - x_{k}^{T}\beta).
\end{equation}
Similarly we define the leave-$j$-th-predictor-out version as
\begin{equation}\label{eq:hatbetaj}
\hat{\beta}_{[j]} = \argmin_{\beta\in\R^{p}}\frac{1}{n}\sum_{k=1}^{n}\rho(\eps_{k} - x_{k, [j]}^{T}\beta).
\end{equation}
Based on these notation we define the full residual $R_{k}$ as
\begin{equation}\label{eq:Rk}
R_{k} = \eps_{k} - x_{k}^{T}\hat{\beta}, \quad k = 1, 2, \ldots, n
\end{equation}
the leave-$j$-th-predictor-out residual as
\begin{equation}\label{eq:rij}
r_{k, [j]} = \eps_{k} - x_{k, [j]}^{T}\hat{\beta}_{[j]}, \quad k = 1,2,\ldots, n, \,\,j \in J_{n}.
\end{equation}
Four diagonal matrices are defined as
\begin{equation}\label{eq:DtdD}
D = \diag(\psi'(R_{k})), \quad \td{D} = \diag(\psi''(R_{k})),
\end{equation}
\begin{equation} \label{eq:DjtdDj}
D_{[j]} = \diag(\psi'(r_{k, [j]})), \quad \td{D}_{[j]} = \diag(\psi''(r_{k, [j]})).
\end{equation}
Further we define $G$ and $G_{[j]}$ as 
\begin{equation}
  \label{eq:GGj}
  G = I - X(X^{T}DX)^{-1}X^{T}D, \quad G_{[j]} = I - X_{[j]}(X_{[j]}^{T}D_{[j]}X_{[j]})^{-1}X_{[j]}^{T}D_{[j]}.
\end{equation}
Let $J_{n}$ denote the indices of coefficients of interest. We say $a\in ]a_{1}, a_{2}[$ if and only if $a\in [\min\{a_{1}, a_{2}\}, \max\{a_{1}, a_{2}\}]$. Regarding the technical assumptions, we need the following quantities
\begin{equation}\label{eq:lambda}
\lammax = \lambda_{\mathrm{\max}}\lb\frac{X^{T}X}{n}\rb, \quad \lammin = \lambda_{\mathrm{\min}}\lb\frac{X^{T}X}{n}\rb
\end{equation}
be the largest (resp. smallest) eigenvalue of the matrix $\frac{X^{T}X}{n}$. Let  $e_{i}\in\R^{n}$ be the $i$-th canonical basis vector and
\begin{equation}\label{eq:hj}
h_{j, 0} = (\psi(r_{1, [j]}), \ldots, \psi(r_{n, [j]}))^{T}, \quad h_{j, 1, i} = G_{[j]}^{T}e_{i}.% (I - D_{[j]}X_{[j]}(X_{[j]}^{T}D_{[j]}X_{[j]})^{-1}X_{[j]}^{T})e_{i}
\end{equation}
Finally, let
\begin{align}
\Delta_{C} &= \max\left\{\max_{j\in J_{n}}\frac{|h_{j, 0}^{T}X_{j}|}{\lnorm h_{j, 0}\lnorm }, \max_{i\le n, j\in J_{n}}\frac{|h_{j, 1, i}^{T}X_{j}|}{\lnorm h_{j, 1, i}\lnorm }\right\},\label{eq:DeltaC}\\ 
Q_{j} & = \Cov(h_{j, 0}). \label{eq:Qj}
\end{align}
We adopt Landau's notation ($O(\cdot), o(\cdot), O_{p}(\cdot), o_{p}(\cdot)$). In addition, we say $a_{n} = \Omega(b_{n})$ if $b_{n} = O(a_{n})$ and similarly, we say $a_{n} = \Omega_{p}(b_{n})$ if $b_{n} = O_{p}(a_{n})$. To simplify the logarithm factors, we use the symbol $\polyLog$ to denote any factor that can be upper bounded by $(\log n)^{\gamma}$ for some $\gamma > 0$. Similarly, we use $\frac{1}{\polyLog}$ to denote any factor that can be lower bounded by $\frac{1}{(\log n)^{\gamma'}}$ for some $\gamma' > 0$. 

Finally we restate all the technical assumptions:
\begin{enumerate}[\textbf{A}1]
\item $\rho(0) = \psi(0) = 0$ and there exists $K_{0} = \Omega\lb\frac{1}{\polyLog}\rb$, $K_{1}, K_{2} = O\lb \polyLog\rb$, such that for any $x\in\m{R}$,
\[K_{0} \le \psi'(x)\le K_{1}, \quad \bigg|\frac{d}{dx}(\sqrt{\psi'}(x))\bigg| = \frac{|\psi''(x)|}{\sqrt{\psi'(x)}}\le K_{2};\]
\item $\ep_{i} = u_{i}(W_{i})$ where $(W_{1}, \ldots, W_{n})\sim N(0, I_{n\times n})$ and $u_{i}$ are smooth functions with $\|u'_{i}\|_{\infty}\le c_{1}$ and $\|u''_{i}\|_{\infty}\le c_{2}$ for some $c_{1}, c_{2} = O(\polyLog)$. Moreover, assume $\min_{i}\Var(\eps_{i}) = \Omega\lb\frac{1}{\polyLog}\rb$.
\item $\lammax = O(\polyLog)$ and $\lammin = \Omega\lb\frac{1}{\polyLog}\rb$;
\item $\cmin\frac{X_{j}^{T}Q_{j}X_{j}}{\tr(Q_{j})} = \Omega\lb\frac{1}{\polyLog}\rb$;
\item $\E\Delta_{C}^{8} = O\lb\polyLog\rb$.
\end{enumerate}

\subsection{Deterministic Approximation Results}\label{subsec:approx}
In Appendix \ref{app:heuristic}, we use several approximations under random designs, e.g. $R_{i}\approx r_{i, [j]}$. To prove them, we follow the strategy of \citeA{elkaroui13} which establishes the deterministic results and then apply the concentration inequalities to obtain high probability bounds. Note that $\hat{\beta}$ is the solution of 
\[0 = f(\beta)\triangleq \frac{1}{n}\sum_{i=1}^{n}x_{i}\psi(\eps_{i} - x_{i}^{T}\beta),\]
we need the following key lemma to bound $\|\beta_{1} - \beta_{2}\|_{2}$ by $\|f(\beta_{1}) - f(\beta_{2})\|_{2}$, which can be calculated explicily. 
\begin{lemma}\label{lem:keylemma}[\citeA{elkaroui13}, Proposition 2.1]
For any $\beta_{1}$ and $\beta_{2}$, 
\[\norm{2}{\beta_{1} - \beta_{2}}\le \frac{1}{K_{0}\lammin}\norm{2}{f(\beta_{1}) - f(\beta_{2})}.\]
\end{lemma}

\begin{proof}
% For any $x, a, b\in \R$, we say $x\tin [a, b]$ if $x$ lies between $a$ and $b$ ($a$ could be larger than $b$ in which case $x\tin [a, b]$ means $x\in[b, a]$). 
By the mean value theorem, there exists $\nu_{i}\in ]\eps_{i} - x_{i}^{T}\beta_{1}, \eps_{i} - x_{i}^{T}\beta_{2}[$ such that
\[\psi(\eps_{i} - x_{i}^{T}\beta_{1}) - \psi(\eps_{i} - x_{i}^{T}\beta_{2}) = \psi'(\nu_{i})\cdot x_{i}^{T}(\beta_{2} - \beta_{1}).\]
Then
\begin{align*}
\norm{2}{f(\beta_{1}) - f(\beta_{2})} &= \norm{2}{\frac{1}{n}\sum_{i=1}^{n}\psi'(\nu_{i})x_{i}x_{i}^{T}\lb\beta_{1} - \beta_{2}\rb}\\
& \ge \lambda_{\mathrm{min}}\lb \frac{1}{n}\sum_{i=1}^{n}\psi'(\nu_{i})x_{i}x_{i}^{T}\rb\cdot\norm{2}{\beta_{1} - \beta_{2}}\\
& \ge K_{0}\lammin\norm{2}{\beta_{1} - \beta_{2}}.
\end{align*}
\end{proof}
\newcommand{\cure}{\mathcal{E}}
\newcommand{\curr}{\Re}
\newcommand{\curb}{\mathcal{B}}
Based on Lemma \ref{lem:keylemma}, we can derive the deterministic results informally stated in Appendix \ref{app:heuristic}. Such results are shown by \citeA{elkaroui13} for ridge-penalized M-estimates and here we derive a refined version for unpenalized M-estimates. Throughout this subsection, we only assume assumption \textbf{A}1. This implies the following lemma, 
\begin{lemma}\label{lem:psi}
Under assumption \textbf{A}1, for any $x$ and $y$, 
\[|\psi(x)|\le K_{1}|x|, \quad|\sqrt{\psi'}(x) - \sqrt{\psi'}(y)|\le K_{2}|x - y|, \quad |\psi'(x) - \psi'(y)|\le 2\sqrt{K_{1}}K_{2}|x - y| \triangleq K_{3}|x - y|.\]
\end{lemma}
To state the result, we define the following quantities.
\begin{equation}\label{eq:TE}
T = \frac{1}{\sqrt{n}}\max\left\{\max_{i}\|x_{i}\|_{2}, \cmax \|X_{j}\|_{2}\right\}, \quad \cure = \frac{1}{n}\sum_{i=1}^{n}\rho(\eps_{i}),
\end{equation}
\begin{equation}\label{eq:UU0}
U = \norm{2}{\frac{1}{n}\sum_{i=1}^{n}x_{i}(\psi(\eps_{i}) - \E\psi(\eps_{i}))}, \quad U_{0} = \left\|\frac{1}{n}\sum_{i=1}^{n}x_{i}\E \psi(\eps_{i})\right\|_{2}.
\end{equation}
 The following proposition summarizes all deterministic results which we need in the proof.
\begin{proposition}\label{prop:deterministic}
Under Assumption $\textbf{A}1$, 
\begin{enumerate}[(i)]
\item The norm of M estimator is bounded by 
\[\|\hat{\beta}\|_{2} \le \frac{1}{K_{0}\lammin}(U + U_{0});\]
\item Define $b_{j}$ as
\[b_{j} = \frac{1}{\sqrt{n}}\frac{N_{j}}{\xi_{j}}\]
where
\[N_{j} = \frac{1}{\sqrt{n}}\sum_{i=1}^{n}X_{ij}\psi(r_{i, [j]}), \quad \xi_{j} = \frac{1}{n}X_{j}^{T}(D_{[j]} - D_{[j]}X_{[j]}(\wmj)^{-1}X_{[j]}^{T}D_{[j]})X_{j},\]
Then 
\[\cmax|b_{j}|\le \frac{1}{\sqrt{n}}\cdot\frac{\sqrt{2K_{1}}}{K_{0}\lammin}\cdot \Delta_{C}\cdot\sqrt{\cure},\]
\item The difference between $\hat{\beta}_{j}$ and $b_{j}$ is bounded by
\[\cmax|\hat{\beta}_{j} - b_{j}|\le \frac{1}{n}\cdot \frac{2K_{1}^{2}K_{3}\lammax T}{K_{0}^{4}\lammin^{\frac{7}{2}}}\cdot\Delta_{C}^{3}\cdot \cure.\]
\item The difference between the full and the leave-one-predictor-out residual is bounded by 
\[\cmax\rmax|R_{i} - r_{i, [j]}|\le \frac{1}{\sqrt{n}}\lb \frac{2K_{1}^{2}K_{3}\lammax T^{2}}{K_{0}^{4}\lammin^{\frac{7}{2}}}\cdot \Delta_{C}^{3}\cdot\cure + \frac{\sqrt{2}K_{1}}{K_{0}^{\frac{3}{2}}\lammin}\cdot\Delta_{C}^{2}\cdot\sqrt{\cure}\rb.\]
\end{enumerate}
\end{proposition}

\newcommand{\uj}{X_{[j]}^{T}D_{[j]}X_{j}}

\begin{proof}\begin{enumerate}[(i)]
  \item By Lemma \ref{lem:keylemma}, 
\[\|\hat{\beta}\|_{2} \le \frac{1}{K_{0}\lammin}\|f(\hat{\beta}) - f(0)\|_{2} = \frac{\|f(0)\|_{2}}{K_{0}\lammin},\]
since $\hat{\beta}$ is a zero of $f(\beta)$. By definition,
\[f(0) = \frac{1}{n}\sum_{i=1}^{n}x_{i}\psi(\eps_{i}) = \frac{1}{n}\sum_{i=1}^{n}x_{i}(\psi(\eps_{i}) - \E \psi(\eps_{i})) + \frac{1}{n}\sum_{i=1}^{n}x_{i}\E \psi(\eps_{i}).\]
This implies that
\[\norm{2}{f(0)} \le U + U_{0}.\]
\item First we prove that
\begin{equation}\label{eq:xij}
\xi_{j}\ge K_{0}\lammin.
\end{equation}
Since all diagonal entries of $D_{[j]}$ is lower bounded by $K_{0}$, we conclude that 
\[\lambda_{\mathrm{min}}\lb\frac{X^{T}D_{[j]}X}{n}\rb\ge K_{0}\lammin.\]
Note that $\xi_{j}$ is the Schur's complement (\citeNP{horn12}, chapter 0.8) of $\frac{X^{T}D_{[j]}X}{n}$, we have
\[\xi_{j}^{-1} = e_{j}^{T}\lb\frac{X^{T}D_{[j]}X}{n}\rb^{-1} e_{j}\le \frac{1}{K_{0}\lammin},\]
which implies \eqref{eq:xij}. As for $N_{j}$, we have
\begin{equation}\label{eq:Njhj}
N_{j} = \frac{X_{j}^{T}h_{j, 0}}{\sqrt{n}} = \frac{\norm{2}{h_{j, 0}}}{\sqrt{n}}\cdot \frac{X_{j}^{T}h_{j, 0}}{\norm{2}{h_{j, 0}}}.
\end{equation}
The the second term is bounded by $\Delta_{C}$ by definition, see \eqref{eq:DeltaC}. For the first term, the assumption \textbf{A}1 that $\psi'(x)\le K_{1}$ implies that
\[\rho(x) = \rho(x) - \rho(0) = \int_{0}^{x}\psi(y)dy\ge \int_{0}^{x}\frac{\psi'(y)}{K_{1}}\cdot \psi(y)dy = \frac{1}{2K_{1}}\psi^{2}(x).\]
Here we use the fact that $\sign(\psi(y)) = \sign(y)$. Recall the definition of $h_{j, 0}$, we obtain that 
\[\frac{\norm{2}{h_{j, 0}}}{\sqrt{n}} =  \sqrt{\frac{\sum_{i=1}^{n}\psi(r_{i, [j]})^{2}}{n}}\le \sqrt{2K_{1}}\cdot \sqrt{\frac{\sum_{i=1}^{n}\rho(r_{i, [j]})}{n}}.\] 
Since $\hat{\beta}_{[j]}$ is the minimizer of the loss function $\sum_{i=1}^{n}\rho(\eps_{i} - x_{i, [j]}^{T}\beta_{[j]})$, it holds that
\[\frac{1}{n}\sum_{i=1}^{n}\rho(r_{i, [j]})\le \frac{1}{n}\sum_{i=1}^{n}\rho(\eps_{i}) = \cure.\]
Putting together the pieces, we conclude that 
\begin{equation}\label{eq:Nj}
|N_{j}|\le \sqrt{2K_{1}}\cdot\Delta_{C}\sqrt{\cure}.
\end{equation}

By definition of $b_{j}$, 
\[|b_{j}|\le \frac{1}{\sqrt{n}}\cdot\frac{\sqrt{2K_{1}}}{K_{0}\lammin}\Delta_{C}\sqrt{\cure}.\]

\item The proof of this result is almost the same as \citeA{elkaroui13}. We state it here for the sake of completeness. Let $\mathbf{\td{b}_{j}}\in\R^{p}$ with
\begin{equation}\label{eq:boldbj}
(\mathbf{\td{b}_{j}})_{j} = b_{j}, \quad (\mathbf{\td{b}_{j}})_{[j]} = \hat{\beta}_{[j]} - b_{j}(\wmj)^{-1}\uj
\end{equation}
where the subscript $j$ denotes the $j$-th entry and the subscript $[j]$ denotes the sub-vector formed by all but $j$-th entry. Furthermore, define $\gamma_{j}$ with 
\begin{equation}\label{eq:gammaj}
(\gamma_{j})_{j} = -1, \quad (\gamma_{j})_{[j]} = (\wmj)^{-1}\uj.
\end{equation}
Then we can rewrite $\mathbf{\td{b}_{j}}$ as 
\[(\mathbf{\td{b}_{j}})_{j} = -b_{j}(\gamma_{j})_{j}, \quad (\mathbf{\td{b}_{j}})_{[j]} = \hat{\beta}_{[j]} - b_{j}(\gamma_{j})_{[j]}.\]
By definition of $\hat{\beta}_{[j]}$, we have $[f(\hat{\beta}_{[j]})]_{[j]} = 0$ and hence
\begin{align}\label{eq:[j]}
[f(\mathbf{\td{b}_{j}})]_{[j]} = [f(\mathbf{\td{b}_{j}})]_{[j]} - [f(\hat{\beta}_{[j]})]_{[j]} = \frac{1}{n}\sum_{i=1}^{n}x_{i, [j]}\left[\psi(\eps_{i} - x_{i}^{T}\mathbf{\td{b}_{j}}) - \psi(\eps_{i} - x_{i, [j]}^{T}\hat{\beta}_{[j]})\right].
\end{align}
By mean value theorem, there exists $\nu_{i, j}\in ]\eps_{i} - x_{i}^{T}\mathbf{\td{b}_{j}}, \eps_{i} - x_{i, [j]}^{T}\hat{\beta}_{[j]}[$ such that
\begin{align*}
&\psi(\eps_{i} - x_{i}^{T}\mathbf{\td{b}_{j}}) - \psi(\eps_{i} - x_{i, [j]}^{T}\hat{\beta}_{[j]}) = \psi'(\nu_{i, j})(x_{i, [j]}^{T}\hat{\beta}_{[j]} - x_{i}^{T}\mathbf{\td{b}_{j}})\\
= & \psi'(\nu_{i, j})(x_{i, [j]}^{T}\hat{\beta}_{[j]} - x_{i, [j]}^{T}(\mathbf{\td{b}_{j}})_{[j]} - X_{ij}b_{j})\\
= & \psi'(\nu_{i, j})\cdot b_{j}\cdot \left[x_{i, [j]}^{T}(\wmj)^{-1}\uj - X_{ij}\right]
%= & \lb\psi'(r_{i, [j]}) + d_{i, j}\rb \cdot b_{j}\cdot \left[x_{i, [j]}^{T}(\wmj)^{-1}\uj - X_{ij}\right].
\end{align*}
Let 
\begin{equation}\label{eq:dij}
d_{i, j} = \psi'(\nu_{i, j}) - \psi'(r_{i, [j]})
\end{equation}
and plug the above result into \eqref{eq:[j]}, we obtain that
\begin{align*}
[f(\mathbf{\td{b}_{j}})]_{[j]} & = \frac{1}{n}\sum_{i=1}^{n}x_{i, [j]}\cdot \lb\psi'(r_{i, [j]}) + d_{i, j}\rb \cdot b_{j}\cdot \left[x_{i, [j]}^{T}(\wmj)^{-1}\uj - X_{ij}\right]\\
& = b_{j}\cdot \frac{1}{n}\sum_{i=1}^{n}\psi'(r_{i, [j]})x_{i, [j]}\left[x_{i, [j]}^{T}(\wmj)^{-1}\uj - X_{ij}\right]\\
& \quad + b_{j}\cdot\frac{1}{n}\sum_{i=1}^{n}d_{i, j}x_{i, [j]}(x_{i, [j]}^{T}(\wmj)^{-1}\uj - X_{ij})\\
& = b_{j}\cdot \frac{1}{n}\left[\wmj(\wmj)^{-1}\uj - \uj\right]\\
& \quad + b_{j}\cdot\frac{1}{n}\sum_{i=1}^{n}d_{i, j}x_{i, [j]}\cdot x_{i}^{T}\gamma_{j}\\
& = b_{j}\cdot\frac{1}{n}\lb\sum_{i=1}^{n}d_{i, j}x_{i, [j]}x_{i}^{T}\rb\gamma_{j}.
\end{align*}
Now we calculate $[f(\mathbf{\td{b}_{j}})]_{j}$, the $j$-th entry of $f(\mathbf{\td{b}_{j}})$. Note that
\begin{align*}
&[f(\mathbf{\td{b}_{j}})]_{j}  = \frac{1}{n}\sum_{i=1}^{n}X_{ij}\psi(\eps_{i} - x_{i}^{T}\mathbf{\td{b}_{j}})\\
= &\frac{1}{n}\sum_{i=1}^{n}X_{ij}\psi(r_{i, [j]}) + b_{j}\cdot\frac{1}{n}\sum_{i=1}^{n}X_{ij}(\psi'(r_{i, [j]}) + d_{i, j})\cdot \left[x_{i, [j]}^{T}(\wmj)^{-1}\uj - X_{ij}\right]\\
= &\frac{1}{n}\sum_{i=1}^{n}X_{ij}\psi(r_{i, [j]}) + b_{j}\cdot \frac{1}{n}\sum_{i=1}^{n}\psi'(r_{i, [j]})X_{ij}\left[x_{i, [j]}^{T}(\wmj)^{-1}\uj - X_{ij}\right]\\
& \quad  + b_{j}\cdot\lb\frac{1}{n}\sum_{i=1}^{n}d_{i, j}X_{ij}x_{i}^{T}\rb \gamma_{j}\\
= & \frac{1}{\sqrt{n}}N_{j} + b_{j}\cdot \lb\frac{1}{n}X_{j}^{T}D_{[j]}X_{[j]}(\wmj)^{-1}\uj - \frac{1}{n}\sum_{i=1}^{n}\psi'(r_{i, [j]})X_{ij}^{2}\rb\\
& \quad  + b_{j}\cdot\lb\frac{1}{n}\sum_{i=1}^{n}d_{i, j}X_{ij}x_{i}^{T}\rb \gamma_{j}\\
& = \frac{1}{\sqrt{n}}N_{j} - b_{j}\cdot \xi_{j} + b_{j}\cdot\lb\frac{1}{n}\sum_{i=1}^{n}d_{i, j}X_{ij}x_{i}^{T}\rb \gamma_{j}\\
& = b_{j}\cdot\lb\frac{1}{n}\sum_{i=1}^{n}d_{i, j}X_{ij}x_{i}^{T}\rb \gamma_{j}
\end{align*}
where the second last line uses the definition of $b_{j}$. Putting the results together, we obtain that
\begin{equation*}
  f(\mathbf{\td{b}_{j}}) = b_{j}\cdot\lb\frac{1}{n}\sum_{i=1}^{n}d_{i,j}x_{i}x_{i}^{T}\rb\cdot \gamma_{j}.
\end{equation*}
This entails that 
\begin{equation}
  \label{eq:fbound}
\|f(\mathbf{\td{b}_{j}})\|_{2}\le |b_{j}|\cdot \max_{i}|d_{i,j}|\cdot \lammax \cdot \|\gamma_{j}\|_{2}.
\end{equation}
Now we derive a bound for $\max_{i}|d_{i,j}|$, where $d_{i,j}$ is defined in \eqref{eq:dij}. By Lemma \ref{lem:psi},
\[|d_{i,j}| = |\psi'(\nu_{i,j}) - \psi'(r_{i, [j]})|\le K_{3}|\nu_{i, j} - r_{i, [j]}| = K_{3}|x_{i, [j]}^{T}\hat{\beta}_{[j]} - x_{i}^{T}\mathbf{\td{b}_{j}}|.\]
By definition of $\mathbf{\td{b}_{j}}$ and $h_{j, 1, i}$,
\begin{align}
&|x_{i, [j]}^{T}\hat{\beta}_{[j]} - x_{i}^{T}\mathbf{\td{b}_{j}}| = |b_{j}|\cdot \big|x_{i, [j]}^{T}(\wmj)^{-1}\uj - X_{ij}\big|\nonumber\\
= & |b_{j}| \cdot |e_{i}^{T}(I - X_{[j]}(\wmj)^{-1}X_{[j]}^{T}D_{[j]})X_{j}| \nonumber\\
= & |b_{j}|\cdot |h_{j, 1, i}^{T}X_{j}| \le |b_{j}|\cdot \Delta_{C}\norm{2}{h_{j, 1, i}},\label{eq:piij}
\end{align}
where the last inequality is derived by definition of $\Delta_{C}$, see \eqref{eq:DeltaC}. Since $h_{j, 1, i}$ is the $i$-th column of matrix $I - D_{[j]}X_{[j]}(\wmj)^{-1}X_{[j]}^{T}$, its $L_{2}$ norm is upper bounded by the operator norm of this matrix. Notice that
\[I - D_{[j]}X_{[j]}(\wmj)^{-1}X_{[j]}^{T} = D_{[j]}^{\frac{1}{2}}\lb I - D_{[j]}^{\frac{1}{2}}X_{[j]}(\wmj)^{-1}X_{[j]}^{T}D_{[j]}^{\frac{1}{2}}\rb D_{[j]}^{-\frac{1}{2}}.\]
The middle matrix in RHS of the displayed atom is an orthogonal projection matrix and hence 
\begin{equation}\label{eq:Gj}
\|I - D_{[j]}X_{[j]}(\wmj)^{-1}X_{[j]}^{T}\|_{\mathrm{op}}\le \|D_{[j]}^{\frac{1}{2}}\|_{\mathrm{op}}\cdot \|D_{[j]}^{-\frac{1}{2}}\|_{\mathrm{op}} \le \lb\frac{K_{1}}{K_{0}}\rb^{\frac{1}{2}}.
\end{equation}
Therefore,
\begin{equation}
  \label{eq:hj1i}
  \max_{i, j}\|h_{j, 1, i}\|_{2}\le \cmax \|I - D_{[j]}X_{[j]}(\wmj)^{-1}X_{[j]}^{T}\|_{\mathrm{op}}\le \lb\frac{K_{1}}{K_{0}}\rb^{\frac{1}{2}},
\end{equation}
and thus
\begin{equation}\label{eq:dij}
\max_{i}|d_{i,j}|\le K_{3}\sqrt{\frac{K_{1}}{K_{0}}}\cdot |b_{j}|\cdot  \Delta_{C}.
\end{equation}
As for $\gamma_{j}$, we have
\begin{align*}
&K_{0}\lammin\|\gamma_{j}\|_{2}^{2} \le \gamma_{j}^{T}\lb\frac{X^{T}D_{[j]}X}{n}\rb\gamma_{j} \\
=& (\gamma_{j})_{j}^{2}\cdot \frac{X_{j}^{T}D_{j}X_{j}}{n} + (\gamma_{j})_{[j]}^{T}\lb\frac{X_{[j]}^{T}D_{[j]}X_{[j]}}{n}\rb(\gamma_{j})_{[j]} + 2\gamma_{j}\frac{X_{j}^{T}D_{[j]}X_{[j]}}{n}(\gamma_{j})_{[j]}
\end{align*}
Recall the definition of $\gamma_{j}$ in \eqref{eq:gammaj}, we have
\[(\gamma_{j})_{[j]}^{T}\lb\frac{X_{[j]}^{T}D_{[j]}X_{[j]}}{n}\rb(\gamma_{j})_{[j]} = \frac{1}{n}X_{j}^{T}D_{[j]}X_{[j]}(X_{[j]}^{T}D_{[j]}X_{[j]})^{-1}X_{[j]}^{T}D_{[j]}X_{j}\]
and 
\[\gamma_{j}\frac{X_{j}^{T}D_{[j]}X_{[j]}}{n}(\gamma_{j})_{[j]} = - \frac{1}{n}X_{j}^{T}D_{[j]}X_{[j]}(X_{[j]}^{T}D_{[j]}X_{[j]})^{-1}X_{[j]}^{T}D_{[j]}X_{j}.\]
As a result, 
\begin{align*}
& K_{0}\lammin\|\gamma_{j}\|_{2}^{2}\\
\le & \frac{1}{n}X_{j}^{T}D_{[j]}^{\frac{1}{2}}(I - D_{[j]}^{\frac{1}{2}}X_{[j]}(\wmj)^{-1}X_{[j]}^{T}D_{[j]}^{\frac{1}{2}})D_{[j]}^{\frac{1}{2}}X_{j}\\
\le & \frac{\|D_{[j]}^{\frac{1}{2}}X_{j}\|_{2}^{2}}{n}\cdot\norm{op}{I - D_{[j]}^{\frac{1}{2}}X_{[j]}(\wmj)^{-1}X_{[j]}^{T}D_{[j]}^{\frac{1}{2}}}\\
\le & \frac{\|D_{[j]}^{\frac{1}{2}}X_{j}\|_{2}^{2}}{n} \le \frac{K_{1}\|X_{j}\|_{2}^{2}}{n}\le T^{2}K_{1},
\end{align*}
where $T$ is defined in \eqref{eq:TE}. Therefore we have
\begin{equation}\label{eq:gammaj}
\norm{2}{\gamma_{j}}\le \sqrt{\frac{K_{1}}{K_{0}\lammin}}T.
\end{equation}
Putting (\ref{eq:fbound}), (\ref{eq:dij}), (\ref{eq:gammaj}) and part (ii) together, we obtain that
\begin{align*}
  \|f(\mathbf{\td{b}_{j}})\|_{2} &\le \lammax\cdot|b_{j}|\cdot K_{3}\sqrt{\frac{K_{1}}{K_{0}}} \Delta_{C}|b_{j}| \cdot \sqrt{\frac{K_{1}}{K_{0}\lammin}}T\\
&\le \lammax\cdot \frac{1}{n}\frac{2K_{1}}{(K_{0}\lammin)^{2}}\Delta_{C}^{2}\cure\cdot K_{3}\sqrt{\frac{K_{1}}{K_{0}}} \Delta_{C}\cdot \sqrt{\frac{K_{1}}{K_{0}\lammin}}T\\
& = \frac{1}{n}\cdot \frac{2K_{1}^{2}K_{3}\lammax T}{K_{0}^{3}\lammin^{\frac{5}{2}}}\cdot \Delta_{C}^{3}\cdot \cure.
\end{align*}
By Lemma \ref{lem:keylemma}, 
\begin{align*}
  \|\hat{\beta} - \mathbf{\td{b}_{j}}\|_{2} &\le \frac{\|f(\hat{\beta}) - f(\mathbf{\td{b}_{j}})\|_{2}}{K_{0}\lammin} = \frac{\|f(\mathbf{\td{b}_{j}})\|_{2}}{K_{0}\lammin} \le \frac{1}{n}\cdot \frac{2K_{1}^{2}K_{3}\lammax T}{K_{0}^{4}\lammin^{\frac{7}{2}}}\cdot \Delta_{C}^{3}\cdot \cure.
\end{align*}
Since $\hat{\beta}_{j} - b_{j}$ is the $j$-th entry of $\hat{\beta} - \mathbf{\td{b}_{j}}$, we have
\[|\hat{\beta}_{j} - b_{j}|\le \|\hat{\beta} - \mathbf{\td{b}_{j}}\|_{2} \le \frac{1}{n}\cdot \frac{2K_{1}^{2}K_{3}\lammax T}{K_{0}^{4}\lammin^{\frac{7}{2}}}\cdot \Delta_{C}^{3} \cdot \cure.\]
\item Similar to part (iii), this result has been shown by \citeA{elkaroui13}. Here we state a refined version for the sake of completeness. Let $\mathbf{\td{b}_{j}}$ be defined as in \eqref{eq:boldbj}, then
  \begin{align*}
    |R_{i} - r_{i, [j]}| & = |x_{i}^{T}\hat{\beta} - x_{i, [j]}^{T}\hat{\beta}_{[j]}| = |x_{i}^{T}(\hat{\beta} - \mathbf{\td{b}_{j}}) + x_{i}^{T}\mathbf{\td{b}_{j}} - x_{i, [j]}^{T}\hat{\beta}_{[j]}|\\ 
& \le \| x_{i}\|_{2} \cdot \| \hat{\beta} - \mathbf{\td{b}_{j}}\|_{2} + |x_{i}^{T}\mathbf{\td{b}_{j}} - x_{i, [j]}^{T}\hat{\beta}_{[j]}|.
\end{align*}
Note that $\norm{2}{x_{i}}\le \sqrt{n}T$, by part (iii), we have
\begin{equation}\label{eq:leave-j-out-err_part1}
\| x_{i}\|_{2} \cdot \| \hat{\beta} - \mathbf{\td{b}_{j}}\|_{2}\le \frac{1}{\sqrt{n}}\frac{2K_{1}^{2}K_{3}\lammax T^{2}}{K_{0}^{4}\lammin^{\frac{7}{2}}}\cdot \Delta_{C}^{3}\cdot \cure.
\end{equation}
On the other hand, similar to \eqref{eq:dij}, by \eqref{eq:piij},
\begin{equation}\label{eq:leave-j-out-err_part2}
|x_{i}^{T}\mathbf{\td{b}_{j}} - x_{i, [j]}^{T}\hat{\beta}_{[j]}|\le \sqrt{\frac{K_{1}}{K_{0}}}\cdot |b_{j}|\cdot \Delta_{C} \le \frac{1}{\sqrt{n}}\cdot \frac{\sqrt{2}K_{1}}{K_{0}^{\frac{3}{2}}\lammin}\cdot \Delta_{C}^{2}\cdot \sqrt{\cure}.
\end{equation}
Therefore, 
\begin{align*}
|R_{i} - r_{i, [j]}|\le \frac{1}{\sqrt{n}}\lb \frac{2K_{1}^{2}K_{3}\lammax T^{2}}{K_{0}^{4}\lammin^{\frac{7}{2}}}\cdot \Delta_{C}^{3}\cdot\cure + \frac{\sqrt{2}K_{1}}{K_{0}^{\frac{3}{2}}\lammin}\cdot\Delta_{C}^{2}\cdot\sqrt{\cure}\rb.
\end{align*}
\end{enumerate}
\end{proof}

\subsection{Summary of Approximation Results}
Under our technical assumptions, we can derive the rate for approximations via Proposition \ref{prop:deterministic}. This justifies all approximations in Appendix \ref{app:heuristic}.
\begin{theorem}\label{thm:mainapprox}
  Under the assumptions \textbf{A}1 - \textbf{A}5, 
  \begin{enumerate}[(i)]
  \item 
\[T\le \lammax = O\lb\polyLog\rb;\]
  \item 
\[\cmax |\hat{\beta}_{j}|\le \|\hat{\beta}\|_{2} = \kbigo{4}{\polyLog};\]
\item
\[\cmax|b_{j}| = \kbigo{2}{\frac{\polyLog}{\sqrt{n}}};\]
\item 
\[\cmax|\hat{\beta}_{j} - b_{j}| = \kbigo{2}{\frac{\polyLog}{n}};\]
\item 
\[\cmax\rmax|R_{i} - r_{i, [j]}| = \kbigo{2}{\frac{\polyLog}{\sqrt{n}}}.\]
  \end{enumerate}
\end{theorem}

\begin{proof}
  \begin{enumerate}[(i)]
  \item Notice that $X_{j} = Xe_{j}$, where $e_{j}$ is the $j$-th canonical basis vector in $\R^{p}$, we have
\[\frac{\|X_{j}\|^{2}}{n} = e_{j}^{T}\frac{X^{T}X}{n}e_{j}\le \lammax.\]
Similarly, consider the $X^{T}$ instead of $X$, we conclude that 
\[\frac{\|x_{i}\|^{2}}{n} \le \lambda_{\max}\lb\frac{XX^{T}}{n}\rb = \lammax.\]
Recall the definition of $T$ in \eqref{eq:TE}, we conclude that 
\[T \le \sqrt{\lammax} = O\lb\polyLog\rb.\]
  \item Since $\eps_{i} = u_{i}(W_{i})$ with $\|u'_{i}\|_{\infty}\le c_{1}$, the gaussian concentration property (\citeNP{ledoux01}, chapter 1.3) implies that $\eps_{i}$ is $c_{1}^{2}$-sub-gaussian and hence $\E |\eps_{i}|^{k} = O(c_{1}^{k})$ for any finite $k > 0$. By Lemma \ref{lem:psi}, $|\psi(\eps_{i})|\le K_{1}|\eps_{i}|$ and hence for any finite $k$,
\[\E |\psi(\eps_{i})|^{k}\le K_{1}^{k}\E |\eps_{i}|^{k} = O(c_{1}^{k}).\]
By part (i) of Proposition \ref{prop:deterministic}, using the convexity of $x^{4}$ and hence $\lb\frac{a + b}{2}\rb^{4} \le \frac{a^{4} + b^{4}}{2}$, 
\begin{align*}
  \E \|\hat{\beta}\|_{2}^{4}\le \frac{1}{(K_{0}\lammin)^{4}}\E (U + U_{0})^{4}\le \frac{8}{(K_{0}\lammin)^{4}}(\E U^{4} + U_{0}^{4}).
\end{align*}
Recall \eqref{eq:UU0} that $U = \norm{2}{\frac{1}{n}\sum_{i=1}^{n}x_{i}(\psi(\eps_{i}) - \E\psi(\eps_{i}))}$, 
\begin{align*}
&  U^{4} = (U^{2})^{2} = \frac{1}{n^{4}}\lb\sum_{i,i'=1}^{n}x_{i}^{T}x_{i'}(\psi(\eps_{i}) - \E\psi(\eps_{i}))(\psi(\eps_{i'}) - \E\psi(\eps_{i'}))\rb^{2}\\
= & \frac{1}{n^{4}}\lb\sum_{i=1}^{n}\|x_{i}\|_{2}^{2}(\psi(\eps_{i}) - \E\psi(\eps_{i}))^{2} + \sum_{i\not= i'}|x_{i}^{T}x_{i'}|(\psi(\eps_{i}) - \E\psi(\eps_{i}))(\psi(\eps_{i'}) - \E\psi(\eps_{i'}))\rb^{2}\\
= &\frac{1}{n^{4}}\bigg\{\sum_{i=1}^{n}\|x_{i}\|_{2}^{4}(\psi(\eps_{i}) - \E\psi(\eps_{i}))^{4} + \sum_{i\not = i'}(2|x_{i}^{T}x_{i'}|^{2} + \|x_{i}\|_{2}^{2}\|x_{i'}\|_{2}^{2})(\psi(\eps_{i}) - \E\psi(\eps_{i}))^{2}(\psi(\eps_{i'}) - \E\psi(\eps_{i'}))^{2}\\
&+ \sum_{\mbox{others}}|x_{i}^{T}x_{i'}|\cdot|x_{k}^{T}x_{k'}|\cdot(\psi(\eps_{i}) - \E\psi(\eps_{i}))(\psi(\eps_{i'}) - \E\psi(\eps_{i'}))(\psi(\eps_{k}) - \E\psi(\eps_{k}))(\psi(\eps_{k'}) - \E\psi(\eps_{k'}))\bigg\}
\end{align*}
Since $\psi(\eps_{i}) - \E\psi(\eps_{i})$ has a zero mean, we have
\[\E (\psi(\eps_{i}) - \E\psi(\eps_{i}))(\psi(\eps_{i'}) - \E\psi(\eps_{i'}))(\psi(\eps_{k}) - \E\psi(\eps_{k}))(\psi(\eps_{k'}) - \E\psi(\eps_{k'})) = 0\]
for any $(i, i')\not = (k, k')\mbox{ or }(k', k)$ and $i\not= i'$. As a consequence, 
\begin{align*}
\E U^{4} &= \frac{1}{n^{4}}\bigg(\sum_{i=1}^{n}\|x_{i}\|_{2}^{4}\E (\psi(\eps_{i}) - \E \psi(\eps_{i}))^{4}\\
&  + \sum_{i\not =i'}(2|x_{i}^{T}x_{i'}|_{2}^{2} + \|x_{i}\|_{2}^{2}\|x_{i'}\|_{2}^{2})\E (\psi(\eps_{i}) - \E \psi(\eps_{i}))^{2}\E (\psi(\eps_{i'}) - \E \psi(\eps_{i'}))^{2}\bigg)\\
&\le \frac{1}{n^{4}}\lb\sum_{i=1}^{n}\|x_{i}\|_{2}^{4}\E (\psi(\eps_{i}) - \E \psi(\eps_{i}))^{4} + 3\sum_{i\not =i'}\|x_{i}\|_{2}^{2}\|x_{i'}\|_{2}^{2}\E (\psi(\eps_{i}) - \E \psi(\eps_{i}))^{2}\E (\psi(\eps_{i'}) - \E \psi(\eps_{i'}))^{2}\rb.
\end{align*}
For any $i$, using the convexity of $x^{4}$, hence $(\frac{a + b}{2})^{4}\le \frac{a^{4} + b^{4}}{2}$, we have
\[\E (\psi(\eps_{i}) - \E \psi(\eps_{i}))^{4}\le 8\E \lb \psi(\eps_{i})^{4} + (\E \psi(\eps_{i}))^{4}\rb \le 16 \E \psi(\eps_{i})^{4}\le 16\max_{i}\E \psi(\eps_{i})^{4}.\]
By Cauchy-Schwartz inequality,
\[\E(\psi(\eps_{i}) - \E \psi(\eps_{i}))^{2}\le \E \psi(\eps_{i})^{2}\le \sqrt{\E \psi(\eps_{i})^{4}}\le \sqrt{\max_{i}\E \psi(\eps_{i})^{4}}.\]
Recall \eqref{eq:TE} that $\|x_{i}\|_{2}^{2} \le nT^{2}$ and thus,
\begin{align*}
\E U^{4} & \le \frac{1}{n^{4}}\lb 16 n\cdot n^{2}T^{4} + 3n^{2}\cdot n^{2}T^{4}\rb\cdot \max_{i}\E \psi(\eps_{i})^{4}\\
& \le \frac{1}{n^{4}}\cdot (16 n^{3} + 3n^{4})T^{4}\max_{i}\E \psi(\eps_{i})^{4} = O\lb\polyLog\rb.
\end{align*}
On the other hand, let $\mu^{T} = (\E \psi(\eps_{1}), \ldots, \E \psi(\eps_{n}))$, then $\|\mu\|_{2}^{2} = O(n\cdot\polyLog)$ and hence by definition of $U_{0}$ in \eqref{eq:UU0},
\[U_{0} = \frac{\|\mu^{T}X\|_{2}}{n} = \frac{1}{n}\sqrt{\mu^{T}XX^{T}\mu}\le \sqrt{\frac{\|\mu\|_{2}^{2}}{n}\cdot \lammax} = O\lb\polyLog\rb.\]
In summary,
\[\E \|\hat{\beta}\|_{2}^{4} = O\lb\polyLog\rb.\]
\item By mean-value theorem, there exists $a_{x}\in (0, x)$ such that 
\[\rho(x) = \rho(0) + x\psi(0) + \frac{x^{2}}{2}\psi'(a_{x}).\]
By assumption \textbf{A}1 and Lemma \ref{lem:psi}, we have
\[\rho(x) = \frac{x^{2}}{2}\psi'(a_{x})\le \frac{x^{2}}{2}\|\psi'\|_{\infty} \le \frac{K_{3}x^{2}}{2},\]
where $K_{3}$ is defined in Lemma \ref{lem:psi}. As a result, 
\[\E \rho(\eps_{i})^{8} \le \lb\frac{K_{3}}{2}\rb^{8}\E\eps_{i}^{16} = O(c_{1}^{16}).\] 
Recall the definition of $\cure$ in \eqref{eq:TE} and the convexity of $x^{8}$, we have
\begin{equation}\label{eq:cure}
\E \cure^{8} \le \frac{1}{n}\sum_{i=1}^{n}\E \rho(\eps_{i})^{8} = O(c_{1}^{16}) = O\lb\polyLog\rb.
\end{equation}
Under assumption \textbf{A}5, by Cauchy-Schwartz inequality,
\[\E (\Delta_{C}\sqrt{\cure})^{2} = \E \Delta_{C}^{2}\cure\le \sqrt{\E \Delta_{C}^{4}}\cdot \sqrt{\E \cure^{2}} = O\lb\polyLog\rb.\]
Under assumptions \textbf{A}1 and \textbf{A}3, 
\[\frac{\sqrt{2K_{1}}}{K_{0}\lammin} = O\lb\polyLog\rb.\]
Putting all the pieces together, we obtain that
\[\cmax |b_{j}| = \kbigo{2}{\frac{\polyLog}{\sqrt{n}}}.\]
\item Similarly, by Holder's inequality,
\[\E (\Delta_{C}^{3}\cure)^{2} = \E \Delta_{C}^{6}\cure^{2}\le \lb\E \Delta_{C}^{8}\rb^{\frac{3}{4}}\cdot \lb\E \cure^{8}\rb^{\frac{1}{4}} = O\lb\polyLog\rb,\]
and under assumptions \textbf{A}1 and \textbf{A}3,
\[\frac{2K_{1}^{2}K_{3}\lammax T}{K_{0}^{4}\lammin^{\frac{7}{2}}} = O\lb\polyLog\rb.\]
Therefore,
\[\cmax|\hat{\beta}_{j} - b_{j}| = \kbigo{2}{\frac{\polyLog}{n}}.\]
\item It follows from the previous part that 
\[\E (\Delta_{C}^{2}\cdot\sqrt{\cure})^{2} = O\lb\polyLog\rb.\]
Under assumptions \textbf{A}1 and \textbf{A}3, the multiplicative factors are also $O\lb\polyLog\rb$, i.e.
\[\frac{2K_{1}^{2}K_{3}\lammax T^{2}}{K_{0}^{4}\lammin^{\frac{7}{2}}} = O\lb\polyLog\rb, \quad \frac{\sqrt{2}K_{1}}{K_{0}^{\frac{3}{2}}\lammin} = O\lb\polyLog\rb.\]
Therefore, 
\[\cmax\rmax|R_{i} - r_{i, [j]}| = \kbigo{2}{\frac{\polyLog}{\sqrt{n}}}.\]
  \end{enumerate}
\end{proof}

\subsection{Controlling Gradient and Hessian}
\begin{proof}[\textbf{Proof of Lemma \ref{lem:beta_deriv}}]
Recall that $\hat{\beta}$ is the solution of the following equation
\begin{equation}
  \label{eq:first_order_cond}
  \frac{1}{n}\sum_{i=1}^{n}x_{i}\psi(\eps_{i} - x_{i}^{T}\hat{\beta}) = 0.
\end{equation}
Taking derivative of (\ref{eq:first_order_cond}), we have
\[X^{T}D\lb I - X\fd\rb = 0\lra \fd = (\wm)^{-1}X^{T}D.\]
This establishes (\ref{eq:fd}). To establishes (\ref{eq:sd}), note that (\ref{eq:fd}) can be rewritten as
\begin{equation}\label{eq:gradient}
(\wm)\fd = X^{T}D.
\end{equation}
Fix $k\in\{1, \cdots, n\}$. Note that 
\[\frac{\partial R_{i}}{\partial \eps_{k}} = \frac{\partial \eps_{i}}{\partial \eps_{k}} - x_{i}^{T}\frac{\partial \hat{\beta}}{\partial \eps_{k}} = I(i = k) - x_{i}^{T}(X^{T}DX)^{-1}X^{T}D.\]
Recall that $G = I - X(X^{T}DX)^{-1}X^{T}D$, we have
\begin{equation}\label{eq:gradient_Ri}
\frac{\partial R_{i}}{\partial \eps_{k}} = e_{i}^{T}Ge_{k},
\end{equation}
where $e_{i}$ is the $i$-th canonical basis of $\R^{n}$. As a result,
\begin{equation}\label{eq:gradient_D}
\frac{\partial D}{\partial \eps_{k}} = \td{D}\diag(Ge_{k}).
\end{equation}
Taking derivative of \eqref{eq:gradient}, we have
\begin{align*}
&X^{T}\frac{\partial D}{\partial \epsilon_{k}}X\fd + (\wm)\sd{k} = X^{T}\frac{\partial D}{\partial \epsilon_{k}}\\
\lra & \sd{k} = (\wm)^{-1}X^{T}\frac{\partial D}{\partial \epsilon_{k}}\lb I - X(\wm)^{-1}X^{T}D\rb\\
\lra & \sd{k} = (\wm)^{-1}X^{T} \tilde{D}\diag(Ge_{k})G,
\end{align*}
where $G = I - X(\wm)^{-1}X^{T}D$ is defined in \eqref{eq:GGj} in p.\pageref{eq:GGj}. Then for each $j\in\{1, \cdots, p\}$ and $k\in \{1, \ldots, n\}$,
\[\Sd{k}{j} = e_{j}^{T}(\wm)^{-1}X^{T}\tilde{D}\diag(Ge_{k})G = e_{k}^{T}G^{T}\diag(e_{j}^{T}(\wm)^{-1}X^{T}\tilde{D})G\]
where we use the fact that $a^{T}\diag(b) = b^{T}\diag(a)$ for any vectors $a, b$. This implies that
\[\SD{j} = G^{T}\diag(e_{j}^{T}(\wm)^{-1}X^{T}\tilde{D})G\]
\end{proof}

\begin{proof}[\textbf{Proof of Lemma \ref{lem:beta_deriv_bound}}]
Throughout the proof we are using the simple fact that $\norm{\infty}{a}\le \norm{2}{a}$. Based on it, we found that
\begin{align}
&\norm{\infty}{e_{j}^{T}(\wm)^{-1}X^{T}D^{\frac{1}{2}}}\le \norm{2}{e_{j}^{T}(\wm)^{-1}X^{T}D^{\frac{1}{2}}}\nonumber\\
= & \sqrt{e_{j}^{T}(\wm)^{-1}\wm(\wm)^{-1}e_{j}}\nonumber\\ 
= & \sqrt{e_{j}^{T}(\wm)^{-1}e_{j}} \le  \frac{1}{(nK_{0}\lammin)^{\frac{1}{2}}}.\label{eq:2_norm}
\end{align}
Thus for any $m > 1$, recall that $M_{j} = \E \norm{\infty}{e_{j}^{T}(\wm)^{-1}X^{T}D^{\frac{1}{2}}}$, 
\begin{align}
&\E \norm {\infty}{e_{j}^{T}(\wm)^{-1}X^{T}D^{\frac{1}{2}}}^{m}\nonumber\\
\le & \E \norm {\infty}{e_{j}^{T}(\wm)^{-1}X^{T}D^{\frac{1}{2}}}\cdot \norm {2}{e_{j}^{T}(\wm)^{-1}X^{T}D^{\frac{1}{2}}}^{m - 1}\nonumber\\
\le & \frac{M_{j}}{(nK_{0}\lammin)^{\frac{m - 1}{2}}}.\label{eq:Mjmoment}
\end{align}
We should emphasize that we cannot use the naive bound that
\begin{align}
&\E \norm {\infty}{e_{j}^{T}(\wm)^{-1}X^{T}D^{\frac{1}{2}}}^{m} \le \E \norm {2}{e_{j}^{T}(\wm)^{-1}X^{T}D^{\frac{1}{2}}}^{m}\le \frac{1}{(nK_{0}\lammin)^{\frac{m}{2}}},\label{eq:naive_bound}\\
& \lra \norm {\infty}{e_{j}^{T}(\wm)^{-1}X^{T}D^{\frac{1}{2}}} = \kbigo{m}{\frac{\polyLog}{\sqrt{n}}}\nonumber
\end{align}
since it fails to guarantee the convergence of TV distance. We will address this issue after deriving Lemma \ref{lem:tv_bound}. 

By contrast, as proved below, 
\begin{equation}\label{eq:inf_norm}
\norm {\infty}{e_{j}^{T}(\wm)^{-1}X^{T}D^{\frac{1}{2}}} = O_{p}(M_{j}) = O_{p}\lb\frac{\polyLog}{n}\rb <\!\!< \frac{1}{\sqrt{nK_{0}\lammin}}.
\end{equation}
Thus \eqref{eq:Mjmoment} produces a slightly tighter bound 
\[\norm {\infty}{e_{j}^{T}(\wm)^{-1}X^{T}D^{\frac{1}{2}}} = \kbigo{m}{\frac{\polyLog}{n^{\frac{m + 1}{2m}}}}.\]
It turns out that the above bound suffices to prove the convergence. Although (\ref{eq:inf_norm}) implies the possibility to sharpen the bound from $n^{-\frac{m + 1}{2m}}$ to $n^{-1}$ using refined analysis, we do not explore this to avoid extra conditions and notation.

~\\
\noindent $\bullet$ \textbf{Bound for $\kappa_{0j}$}

First we derive a bound for $\kappa_{0j}$. By definition, 
\[\kappa_{0j}^{2} = \E \norm{4}{\pd{\hat{\beta}_{j}}{\eps^{T}}}^{4}\le \E\lb\norm{\infty}{\pd{\hat{\beta}_{j}}{\eps^{T}}}^{2}\cdot \norm{2}{\pd{\hat{\beta}_{j}}{\eps^{T}}}^{2}\rb.\]
By Lemma \ref{lem:beta_deriv} and \eqref{eq:Mjmoment} with $m = 2$, 
\begin{align*}
\E\norm{\infty}{\pd{\hat{\beta}_{j}}{\eps^{T}}}^{2} \le \E\norm{\infty}{e_{j}^{T}(\wm)^{-1}X^{T}D^{\frac{1}{2}}}^{2}\cdot K_{1} = \frac{K_{1}M_{j}}{(nK_{0}\lammin)^{\frac{1}{2}}}.
\end{align*}
On the other hand, it follows from \eqref{eq:2_norm} that
\begin{align}
\norm{2}{\pd{\hat{\beta}_{j}}{\eps^{T}}}^{2} &= \norm{2}{e_{j}^{T}(\wm)^{-1}X^{T}D}^{2} \le K_{1}\cdot \norm{2}{e_{j}^{T}(\wm)^{-1}X^{T}D^{\frac{1}{2}}}^{2} \le \frac{K_{1}}{nK_{0}\lammin}.\label{eq:2_norm_sq}
\end{align}
Putting the above two bounds together we have
\begin{equation}\label{eq:kappa0j}
\kappa_{0j}^{2}\le \frac{K_{1}^{2}}{(nK_{0}\lammin)^{\frac{3}{2}}}\cdot M_{j}.
\end{equation}

~\\
\noindent $\bullet$ \textbf{Bound for $\kappa_{1j}$}

As a by-product of \eqref{eq:2_norm_sq}, we obtain that
\begin{align}
\kappa_{1j}^{4} & = \E \norm{2}{\pd{\hat{\beta}_{j}}{\eps^{T}}}^{4} \le \frac{K_{1}^{2}}{(nK_{0}\lammin)^{2}}.\label{eq:kappa1j}
\end{align}

~\\
\noindent $\bullet$ \textbf{Bound for $\kappa_{2j}$}

Finally, we derive a bound for $\kappa_{2j}$. By Lemma \ref{lem:beta_deriv}, $\kappa_{2j}$ involves the operator norm of a symmetric matrix with form $G^{T}MG$ where $M$ is a diagonal matrix. Then by the triangle inequality, 
\[\norm{op}{G^{T}MG}\le \|M\|_{\mathrm{op}}\cdot\norm{op}{G^{T}G} = \|M\|_{\mathrm{op}}\cdot\norm{op}{G}^{2}.\]
Note that 
\[D^{\frac{1}{2}}GD^{-\frac{1}{2}} = I - D^{\frac{1}{2}}X(\wm)^{-1}X^{T}D^{\frac{1}{2}}\]
is a projection matrix, which is idempotent. This implies that
\[\norm{op}{D^{\frac{1}{2}}GD^{-\frac{1}{2}}} = \lambda_{\mathrm{max}}\lb D^{\frac{1}{2}}GD^{-\frac{1}{2}}\rb\le 1.\]
Write $G$ as $D^{-\frac{1}{2}}(D^{\frac{1}{2}}GD^{-\frac{1}{2}})D^{\frac{1}{2}}$, then we have
\[\norm{op}{G} \le \norm{op}{D^{-\frac{1}{2}}}\cdot \norm{op}{D^{\frac{1}{2}}GD^{-\frac{1}{2}}}\cdot \norm{op}{D^{\frac{1}{2}}}\le \sqrt{\frac{K_{1}}{K_{0}}}.\]
Returning to $\kappa_{2j}$, we obtain that
\begin{align*}
 \kappa_{2j}^{4} & = \E \norm{op}{G^{T}\diag(e_{j}^{T}(\wm)^{-1}X^{T}\td{D})G}^{4} \\
& \le \E \lb\norm{\infty}{e_{j}^{T}(\wm)^{-1}X^{T}\td{D}}^{4}\cdot \norm{op}{G}^{8}\rb\\
& \le \E \lb\norm{\infty}{e_{j}^{T}(\wm)^{-1}X^{T}\td{D}}^{4}\rb \lb\frac{K_{1}}{K_{0}}\rb^{4}\\
& = \E \lb\norm{\infty}{e_{j}^{T}(\wm)^{-1}X^{T}D^{\frac{1}{2}}D^{-\frac{1}{2}}\td{D}}^{4}\rb\cdot \lb\frac{K_{1}}{K_{0}}\rb^{4}
\end{align*}
Assumption \textbf{A}1 implies that 
\[\forall i, \,\, \frac{|\psi''(R_{i})|}{\sqrt{\psi'(R_{i})}} \le K_{2}\mbox{ \& hence } \|D^{-\frac{1}{2}}\td{D}\|_{\mathrm{op}}\le K_{2}.\]
Therefore,
\[\norm{\infty}{e_{j}^{T}(\wm)^{-1}X^{T}D^{\frac{1}{2}}D^{-\frac{1}{2}}\td{D}}^{4}\le K_{2}^{4}\cdot\norm{\infty}{e_{j}^{T}(\wm)^{-1}X^{T}D^{\frac{1}{2}}}^{4}.\]
By \eqref{eq:Mjmoment} with $m = 4$,
\begin{align}
 \kappa_{2j}^{4} & \le \frac{K_{2}^{4}}{(n\lammin)^{\frac{3}{2}}}\cdot \lb\frac{K_{1}}{K_{0}}\rb^{4}\cdot M_{j}.\label{eq:kappa2j}
\end{align}
\end{proof}

\begin{proof}[\textbf{Proof of Lemma \ref{lem:tv_bound}}]
By Theorem \ref{thm:mainapprox}, for any $j$, 
\[\E \hat{\beta}_{j}^{4} \le \E \|\hat{\beta}\|_{2}^{4} < \infty.\]
Then using the \SOPI (Proposition \ref{prop:sopi}), 
\begin{align*}
&\cmax d_{TV}\lb\mathcal{L}\lb\frac{\hat{\beta}_{j} - \E \hat{\beta}_{j}}{\sqrt{\Var(\hat{\beta}_{j})}}\rb, N(0, 1)\rb = O\lb\frac{c_{1}c_{2}\kappa_{0j}+ c_{1}^{3}\kappa_{1j}\kappa_{2j}}{\Var(\hat{\beta}_{j})}\rb \\
= &O\lb \frac{\frac{M_{j}^{\frac{1}{2}}}{n^{\frac{3}{4}}} + \frac{M_{j}^{\frac{1}{4}}}{n^{\frac{7}{8}}}}{\Var(\hat{\beta}_{j})}\cdot \polyLog\rb =  O\lb \frac{(nM_{j}^{2})^{\frac{1}{4}} + (nM_{j}^{2})^{\frac{1}{8}}}{n\Var(\hat{\beta}_{j})}\cdot \polyLog\rb.
\end{align*}
It follows from \eqref{eq:2_norm} that $nM_{j}^{2} = O\lb\polyLog\rb$ and the above bound can be simplified as
\[\cmax d_{TV}\lb\mathcal{L}\lb\frac{\hat{\beta}_{j} - \E \hat{\beta}_{j}}{\sqrt{\Var(\hat{\beta}_{j})}}\rb, N(0, 1)\rb = O\lb \frac{(nM_{j}^{2})^{\frac{1}{8}}}{n\Var(\hat{\beta}_{j})}\cdot \polyLog\rb.\]
\end{proof}
\begin{remark}
  If we use the naive bound \eqref{eq:naive_bound}, by repeating the above derivation, we obtain a worse bound for $\kappa_{0, j} = O(\frac{\polyLog}{n})$ and $\kappa_{2} = O(\frac{\polyLog}{\sqrt{n}})$, in which case,
\[\cmax d_{TV}\lb\mathcal{L}\lb\frac{\hat{\beta}_{j} - \E \hat{\beta}_{j}}{\sqrt{\Var(\hat{\beta}_{j})}}\rb, N(0, 1)\rb = O\lb\frac{\polyLog}{n\Var(\hat{\beta}_{j})}\rb.\]
However, we can only prove that  $\Var(\hat{\beta}_{j}) = \Omega(\frac{1}{n})$. Without the numerator $(nM_{j}^{2})^{\frac{1}{8}}$, which will be shown to be $O(n^{-\frac{1}{8}}\polyLog)$ in the next subsection, the convergence cannot be proved.
\end{remark}

\subsection{Upper Bound of $M_{j}$}
As mentioned in Appendix \ref{app:heuristic}, we should approximate $D$ by $D_{[j]}$ to remove the functional dependence on $X_{j}$ . To achieve this, we introduce two terms, $M_{j}^{(1)}$ and $M_{j}^{(2)}$, defined as
\[M_{j}^{(1)} = \E (\| e_{j}^{T}(\wm)^{-1}X^{T} D_{[j]}^{\frac{1}{2}}\| _{\infty}), \quad M_{j}^{(2)} = \E (\| e_{j}^{T}(X^{T}D_{[j]}X)^{-1}X^{T} D_{[j]}^{\frac{1}{2}}\| _{\infty}).\]
We will first prove that both $|M_{j} - M_{j}^{(1)}|$ and $|M_{j}^{(1)} - M_{j}^{(2)}|$ are negligible and then derive an upper bound for $M_{j}^{(2)}$. 
\subsubsection{Controlling $|M_{j} - M_{j}^{(1)}|$}
By Lemma \ref{lem:psi}, 
\[\| D^{\frac{1}{2}} - D_{[j]}^{\frac{1}{2}}\| _{\infty} \le K_{2}\max_{i}|R_{i} - r_{i, [j]}| \triangleq K_{2}\mathcal{R}_{j},\]
and by Theorem \ref{thm:mainapprox}, 
\[\sqrt{\E\mathcal{R}_{j}^{2}} = O\lb\frac{\polyLog}{\sqrt{n}}\rb.\]
Then we can bound $|M_{j} - M_{j}^{(1)}|$ via the fact that $\norm{\infty}{a}\le \norm{2}{a}$ and algebra as follows.
\begin{align*}
|M_{j} - M_{j}^{(1)}|&\le \E (\|  e_{j}^{T}(\wm)^{-1}X^{T} (D^{\frac{1}{2}} - D_{[j]}^{\frac{1}{2}})\| _{\infty} )\\
&\le \E (\|  e_{j}^{T}(\wm)^{-1}X^{T} (D^{\frac{1}{2}} - D_{[j]}^{\frac{1}{2}})\| _{2} )\\
&\le \sqrt{\E (\|  e_{j}^{T}(\wm)^{-1}X^{T} (D^{\frac{1}{2}} - D_{[j]}^{\frac{1}{2}})\| _{2}^{2} )}\\
& = \sqrt{\E ( e_{j}^{T}(\wm)^{-1}X^{T} (D^{\frac{1}{2}} - D_{[j]}^{\frac{1}{2}})^{2} X(\wm)^{-1} e_{j})}.
\end{align*}
By Lemma \ref{lem:psi}, 
\[|\sqrt{\psi'(R_{i})} - \sqrt{\psi'(r_{i, [j]})}|\le K_{2}|R_{i} - r_{i, [j]}|\le K_{2}\mathcal{R}_{j},\]
thus 
\[(D^{\frac{1}{2}} - D_{[j]}^{\frac{1}{2}})^{2}\preceq K_{2}^{2}\mathcal{R}_{j}^{2}I\preceq \frac{K_{2}^{2}}{K_{0}}\mathcal{R}_{j}^{2}D.\]
This entails that 
\begin{align*}
|M_{j} - M_{j}^{(1)}| & \le K_{2}K_{0}^{-\frac{1}{2}}\sqrt{\E (\mathcal{R}_{j}^{2}\cdot e_{j}^{T}(\wm)^{-1}\wm(\wm)^{-1} e_{j})}\\
& = K_{2}K_{0}^{-\frac{1}{2}}\sqrt{\E (\mathcal{R}_{j}^{2}\cdot e_{j}^{T}(\wm)^{-1} e_{j})}\\
&\le \frac{K_{2}}{\sqrt{n}K_{0}\sqrt{\lammin}}\sqrt{\E (\mathcal{R}_{j}^{2})} = O\lb\frac{\polyLog}{n}\rb.
\end{align*}

\subsubsection{Bound of $|M_{j}^{(1)} - M_{j}^{(2)}|$}
First we prove a useful lemma.
\begin{lemma}\label{lem:I_plus_M_inv}
For any symmetric matrix $N$ with $\|N\|_{\mathrm{op}} <  1$,
\[(I - (I + N)^{-1})^{2}\preceq \frac{N^{2}}{(1 - \|N\|_{\mathrm{op}})^{2}}.\]
\end{lemma}
\begin{proof}
First, notice that
\[I - (I + N)^{-1} = (I + N - I)(I + N)^{-1} = N(I + N)^{-1},\]
and therefore
\[(I - (I + N)^{-1})^{2} = N(I + N)^{-2}N.\]
Since $\|N\|_{\mathrm{op}} < 1$, $I + N$ is positive semi-definite and 
\[(I + N)^{-2}\preceq \frac{1}{(1 - \|N\|_{\mathrm{op}})^{2}}I.\]
Therefore,
\[N(I + N)^{-2}N\preceq \frac{N^{2}}{(1 - \|N\|_{\mathrm{op}})^{2}}.\]
\end{proof}

We now back to bounding $|M_{j}^{(1)} - M_{j}^{(2)}|$. Let $A_{j} = X^{T}D_{[j]}X$, $B_{j} = X^{T}(D - D_{[j]})X$. By Lemma \ref{lem:psi}, 
\[\|D - D_{[j]}\|_{\infty}\le K_{3}\max_{i}|R_{i} - r_{i, [j]}| = K_{3}\mathcal{R}_{j}\]
and hence 
\[\|B_{j}\|_{\mathrm{op}}\le K_{3}\mathcal{R}_{j}\cdot n\lammax I\triangleq n\eta_{j}.\]
where $\eta_{j} = K_{3}\lammax\cdot \mathcal{R}_{j}$. Then by Theorem \ref{thm:mainapprox}.(v),
\[\E(\eta_{j}^{2}) = O\lb\frac{\polyLog}{n}\rb.\]
Using the fact that $\norm{\infty}{a}\le \norm{2}{a}$, we obtain that
\begin{align*}
|M_{j}^{(1)} - M_{j}^{(2)}| & \le \E (\|  e_{j}^{T}A_{j}^{-1}X^{T}D_{[j]}^{\frac{1}{2}} -  e_{j}^{T}(A_{j} + B_{j})^{-1}X^{T}D_{[j]}^{\frac{1}{2}}\| _{\infty} )\\
& \le \sqrt{\E (\|  e_{j}^{T}A_{j}^{-1}X^{T}D_{[j]}^{\frac{1}{2}} -  e_{j}^{T}(A_{j} + B_{j})^{-1}X^{T}D_{[j]}^{\frac{1}{2}}\|_{2}^{2} )}\\
& = \sqrt{\E \left[e_{j}^{T}(A_{j}^{-1} - (A_{j} + B_{j})^{-1})X^{T}D_{[j]}X(A_{j}^{-1} - (A_{j} + B_{j})^{-1}) e_{j}\right]}\\
& = \sqrt{\E  \left[e_{j}^{T}(A_{j}^{-1} - (A_{j} + B_{j})^{-1})A_{j}(A_{j}^{-1} - (A_{j} + B_{j})^{-1}) e_{j}\right]}
\end{align*}
The inner matrix can be rewritten as
\begin{align}
&(A_{j}^{-1} - (A_{j} + B_{j})^{-1})A_{j}(A_{j}^{-1} - (A_{j} + B_{j})^{-1})\nonumber\\
=& A_{j}^{-\frac{1}{2}}(I - (I + A_{j}^{-\frac{1}{2}}B_{j}A_{j}^{-\frac{1}{2}})^{-1})A_{j}^{-\frac{1}{2}}A_{j}A_{j}^{-\frac{1}{2}}(I - (I + A_{j}^{-\frac{1}{2}}B_{j}A_{j}^{-\frac{1}{2}})^{-1})A_{j}^{-\frac{1}{2}}\nonumber\\
=& A_{j}^{-\frac{1}{2}}(I - (I + A_{j}^{-\frac{1}{2}}B_{j}A_{j}^{-\frac{1}{2}})^{-1})^{2}A_{j}^{-\frac{1}{2}}.\label{eq:matrix_game}
\end{align}
Let $N_{j} = A_{j}^{-\frac{1}{2}}B_{j}A_{j}^{-\frac{1}{2}}$, then 
\[\|N_{j}\|_{\mathrm{op}}\le \|A_{j}^{-\frac{1}{2}}\|_{\mathrm{op}}\cdot \|B_{j}\|_{\mathrm{op}}\cdot \|A_{j}^{-\frac{1}{2}}\|_{\mathrm{op}}\le (nK_{0}\lammin)^{-\frac{1}{2}}\cdot n\eta_{j}\cdot (nK_{0}\lammin)^{-\frac{1}{2}} = \frac{\eta_{j}}{K_{0}\lammin}.\]

\noindent On the event $\{\eta_{j} \le \frac{1}{2}K_{0}\lammin\}$, $\|N_{j}\|_{\mathrm{op}}\le \frac{1}{2}$. By Lemma \ref{lem:I_plus_M_inv},
\[(I - (I + N_{j})^{-1})^{2}\preceq 4N_{j}^{2}.\]
This together with \eqref{eq:matrix_game} entails that 
\begin{align*}
&e_{j}^{T}(A_{j}^{-1} - (A_{j} + B_{j})^{-1})A_{j}(A_{j}^{-1} - (A_{j} + B_{j})^{-1}) e_{j} = e_{j}^{T}A_{j}^{-\frac{1}{2}}(I - (I + N_{j})^{-1})^{2}A_{j}^{-\frac{1}{2}}e_{j}\\
\le & 4e_{j}^{T}A_{j}^{-\frac{1}{2}}N_{j}^{2}A_{j}^{-\frac{1}{2}}e_{j} = e_{j}^{T}A_{j}^{-1}B_{j}A_{j}^{-1}B_{j}A_{j}^{-1}e_{j} \le \|A_{j}^{-1}B_{j}A_{j}^{-1}B_{j}A_{j}^{-1}\|_{\mathrm{op}}.
\end{align*}
Since $A_{j}\succeq nK_{0}\lammin I$, and $\|B_{j}\|_{\mathrm{op}}\le  n\eta_{j}$, we have
\[\|A_{j}^{-1}B_{j}A_{j}^{-1}B_{j}A_{j}^{-1}\|_{\mathrm{op}}\le \|A_{j}^{-1}\|_{\mathrm{op}}^{3} \cdot \|B_{j}\|_{\mathrm{op}}^{2}\le \frac{1}{n}\cdot \frac{1}{(K_{0}\lammin)^{3}}\cdot \eta_{j}^{2}.\]
Thus, 
\begin{align*}
&\E\left[ e_{j}^{T}(A_{j}^{-1} - (A_{j} + B_{j})^{-1})A_{j}(A_{j}^{-1} - (A_{j} + B_{j})^{-1}) e_{j} \cdot I\lb\eta_{j}\le \frac{K_{0}\lammin}{2}\rb\right]\\
\le & \E \left[e_{j}^{T}A_{j}^{-1}B_{j}A_{j}^{-1}B_{j}A_{j}^{-1}e_{j}\right] \le \frac{1}{n}\cdot \frac{1}{(K_{0}\lammin)^{3}}\cdot \E \eta_{j}^{2} = O\lb\frac{\polyLog}{n^{2}}\rb.
\end{align*}

\noindent On the event $\{\eta_{j} > \frac{1}{2}K_{0}\lammin\}$, since $nK_{0}\lammin I \preceq A_{j}\preceq nK_{1}\lammax I$ and $A_{j} + B_{j}\succeq nK_{0}\lammin I$, 
\begin{align*}
&|e_{j}^{T}(A_{j}^{-1} - (A_{j} + B_{j})^{-1})A_{j}(A_{j}^{-1} - (A_{j} + B_{j})^{-1}) e_{j}|\\
\le & nK_{1}\lammax\cdot |e_{j}^{T}(A_{j}^{-1} - (A_{j} + B_{j})^{-1})^{2} e_{j}|\\
\le & nK_{1}\lammax\cdot \lb 2|e_{j}^{T}A_{j}^{-2}e_{j}| + 2|e_{j}^{T}(A_{j} + B_{j})^{-2}e_{j}|\rb\\
\le & \frac{4nK_{1}\lammax}{(nK_{0}\lammin)^{2}} = \frac{1}{n}\cdot \frac{4K_{1}\lammax}{(K_{0}\lammin)^{2}}.
\end{align*}
This together with Markov inequality implies htat 
\begin{align*}
&\E\left[ e_{j}^{T}(A_{j}^{-1} - (A_{j} + B_{j})^{-1})A_{j}(A_{j}^{-1} - (A_{j} + B_{j})^{-1}) e_{j} \cdot I\lb\eta_{j} > \frac{K_{0}\lammin}{2}\rb\right]\\
\le & \frac{1}{n}\cdot \frac{4K_{1}\lammax}{(K_{0}\lammin)^{2}} \cdot P\lb\eta_{j} > \frac{K_{0}\lammin}{2}\rb\\
\le & \frac{1}{n}\cdot \frac{4K_{1}\lammax}{(K_{0}\lammin)^{2}} \cdot \frac{4}{(K_{0}\lammin)^{2}}\cdot \E \eta_{j}^{2}\\
= & O\lb\frac{\polyLog}{n^{2}}\rb.
\end{align*}
Putting pieces together, we conclude that
\begin{align*}
|M_{j}^{(1)} - M_{j}^{(2)}| & \le \sqrt{\E\left[ e_{j}^{T}(A_{j}^{-1} - (A_{j} + B_{j})^{-1})A_{j}(A_{j}^{-1} - (A_{j} + B_{j})^{-1}) e_{j}\right]}\\
& \le \sqrt{\E\left[ e_{j}^{T}(A_{j}^{-1} - (A_{j} + B_{j})^{-1})A_{j}(A_{j}^{-1} - (A_{j} + B_{j})^{-1}) e_{j} \cdot I\lb\eta_{j} > \frac{K_{0}\lammin}{2}\rb\right]}\\
& \quad + \sqrt{\E\left[ e_{j}^{T}(A_{j}^{-1} - (A_{j} + B_{j})^{-1})A_{j}(A_{j}^{-1} - (A_{j} + B_{j})^{-1}) e_{j} \cdot I\lb\eta_{j} \le \frac{K_{0}\lammin}{2}\rb\right]}\\
= & O\lb\frac{\polyLog}{n}\rb.
\end{align*}

\subsubsection{Bound of $M_{j}^{(2)}$}
Similar to \eqref{eq:M1}, by block matrix inversion formula (See Proposition \ref{prop:block_inv}),
\[e_{j}^{T}(X^{T}D_{[j]}X)^{-1}X^{T}D_{[j]}^{\frac{1}{2}} = \frac{X_{j}^{T}D_{[j]}^{\frac{1}{2}}(I - H_{j})}{X_{j}^{T}D_{[j]}^{\frac{1}{2}}(I - H_{j})D_{[j]}^{\frac{1}{2}} X_{j}},\]
where $H_{j} = D_{[j]}^{\frac{1}{2}} X_{[j]} (\wmj)^{-1}X_{[j]}^{T}D_{[j]}^{\frac{1}{2}}$. Recall that $\xi_{j}\ge K_{0}\lammin$ by (\ref{eq:xij}), so we have
\[X_{j}^{T}D_{[j]}^{\frac{1}{2}}(I - H_{j})D_{[j]}^{\frac{1}{2}} X_{j} = n\xi_{j}\ge n\lammin.\]
As for the numerator, recalling the definition of $h_{j, 1, i}$, we obtain that
\begin{align*}
\| X_{j}^{T}D_{[j]}^{\frac{1}{2}}(I - H_{j})\|_{\infty} &= \norm{\infty}{\frac{1}{n} X_{j}^{T}(I - D_{[j]}X_{[j]}(\wmj)^{-1}X_{[j]})\cdot D_{[j]}^{\frac{1}{2}}}\\
&\le \sqrt{K_{1}}\cdot\norm{\infty}{\frac{1}{n} X_{j}^{T}(I - D_{[j]}X_{[j]}(\wmj)^{-1}X_{[j]})}\\
& = \sqrt{K_{1}}\max_{i}\big|h_{j, 1, i}^{T}X_{j}\big| \le \sqrt{K_{1}}\Delta_{C}\max_{i}\| h_{j, 1, i}\|_{2}.
% & = \sqrt{K_{1}}\Delta_{C}\cdot\frac{1}{n}\sqrt{e_{i}^{T}X_{[j]}(\wmj)^{-1}X_{[j]}^{T}D_{[j]}^{2} X_{[j]}(\wmj)^{-1}X_{[j]}^{T}e_{i}}\\
% & = K_{1}\Delta_{C}\cdot \frac{1}{\sqrt{n}}\sqrt{e_{i}^{T}X_{[j]}(\wmj)^{-1}X_{[j]}^{T}e_{i}}\\
% & = K_{1}\Delta_{C}\cdot\frac{1}{\sqrt{n\lammin}}\sqrt{e_{i}^{T}\lb\frac{X_{[j]}X_{[j]}^{T}}{n}\rb e_{i}}\\
% & = K_{1}\sqrt{\frac{\lammax}{\lammin}}\cdot \Delta_{C}.
\end{align*}
As proved in \eqref{eq:hj1i}, 
\[\rmax \|h_{j, 1, i}\|_{2}\le \lb\frac{K_{1}}{K_{0}}\rb^{\frac{1}{2}}.\]
This entails that
\[\| X_{j}^{T}D_{[j]}^{\frac{1}{2}}(I - H_{j})\|_{\infty}\le \frac{K_{1}}{\sqrt{K_{0}}}\cdot \Delta_{C} = \kbigo{1}{\polyLog}.\]
Putting the pieces together we conclude that
\begin{align*}
M_{j}^{(2)}\le & \frac{\E\| X_{j}^{T}D_{[j]}^{\frac{1}{2}}(I - H_{j})\| _{\infty}}{n\lammin} = O\lb\frac{\polyLog}{n}\rb.
\end{align*}

\subsubsection{Summary}
Based on results from Section B.5.1 - Section B.5.3, we have
\[M_{j} = O\lb\frac{\polyLog}{n}\rb.\]
Note that the bounds we obtained do not depend on $j$, so we conclude that
\[\cmax M_{j} = O\lb\frac{\polyLog}{n}\rb.\]

\subsection{Lower Bound of $\Var(\hat{\beta}_{j})$}

\subsubsection{Approximating $\Var(\hat{\beta}_{j})$ by $\Var(b_{j})$}
By Theorem \ref{thm:mainapprox}, 
\[\max_{j}\E (\hat{\beta}_{j} - b_{j})^{2} = O\lb\frac{\polyLog}{n^{2}}\rb, \quad \max_{j} \E b_{j}^{2} = O\lb\frac{\polyLog}{n}\rb.\]
Using the fact that 
\[\hat{\beta}_{j}^{2} - b_{j}^{2} = (\hat{\beta}_{j} - b_{j} + b_{j})^{2} - b_{j}^{2} = (\hat{\beta}_{j} - b_{j})^{2} + 2(\hat{\beta}_{j} - b_{j})b_{j},\]
we can bound the difference between $\E\hat{\beta}_{j}^{2}$ and $\E b_{j}^{2}$ by
\begin{align*}
\big|\E \hat{\beta}_{j}^{2} - \E b_{j}^{2}\big| & = \E (\hat{\beta}_{j} - b_{j})^{2} + 2|\E (\hat{\beta}_{j} - b_{j})b_{j}| \le \E (\hat{\beta}_{j} - b_{j})^{2} + 2\sqrt{\E (\hat{\beta}_{j} - b_{j})^{2}}\sqrt{\E b_{j}^{2}} = O\lb\frac{\polyLog}{n^{\frac{3}{2}}}\rb.
\end{align*}
Similarly, since $|a^{2} - b^{2}| = |a - b|\cdot |a + b|\le |a - b|(|a - b| + 2|b|)$, 
\begin{align*}
|(\E \hat{\beta}_{j})^{2} - (\E b_{j})^{2}| & \le \E |\hat{\beta}_{j} - b_{j}|\cdot \lb\E |\hat{\beta}_{j} - b_{j}| + 2 \E |b_{j}| \rb = O\lb\frac{\polyLog}{n^{\frac{3}{2}}}\rb.
\end{align*}
Putting the above two results together, we conclude that
\begin{equation}\label{eq:approxbetajbj}
\big|\Var(\hat{\beta}_{j}) - \Var(b_{j})\big| = O\lb\frac{\polyLog}{n^{\frac{3}{2}}}\rb.
\end{equation}
Then it is left to show that
\[\Var(b_{j}) = \Omega\lb\frac{1}{n\cdot \polyLog}\rb.\]

\subsubsection{Controlling $\Var(b_{j})$ by $\Var(N_{j})$}
Recall that
\[b_{j} = \frac{1}{\sqrt{n}}\frac{N_{j}}{\xi_{j}}\]
where 
\[N_{j} = \frac{1}{\sqrt{n}}\sum_{i=1}^{n}X_{ij}\psi(r_{i, [j]}), \quad \xi_{j} = \frac{1}{n}X_{j}^{T}(D_{[j]} - D_{[j]}X_{[j]}(\wmj)^{-1}X_{[j]}^{T}D_{[j]})X_{j}.\]
% For any random variable $Z$ and real number $c$, we have
% \begin{equation}\label{eq:var_decomp}
% \Var(Z) = \E(Z - c)^{2} - (\E Z - c)^{2}.
% \end{equation}
% Apply the above result to $b_{j}$, we have
Then 
\[n\Var(b_{j}) = \E\lb \frac{N_{j}}{\xi_{j}} - \E \frac{N_{j}}{\xi_{j}}\rb^{2} = \E \lb \frac{N_{j} - \E N_{j}}{\xi_{j}} + \frac{\E N_{j}}{\xi_{j}} - \E \frac{N_{j}}{\xi_{j}}\rb^{2}.\]
Using the fact that $(a + b)^{2} - (\frac{1}{2}a^{2} - b^{2}) = \frac{1}{2}(a + 2b)^{2}\ge 0$, we have 
\begin{equation}\label{eq:var_bj}
n\Var(b_{j}) \ge \frac{1}{2}\E \lb \frac{N_{j} - \E N_{j}}{\xi_{j}}\rb^{2} - \E \lb \frac{\E N_{j}}{\xi_{j}} - \E \frac{N_{j}}{\xi_{j}}\rb^{2}\triangleq \frac{1}{2}I_{1} - I_{2}.
\end{equation}

\subsubsection{Controlling $I_{1}$}\label{subsubsec:I1}
The Assumption $\textbf{A}4$ implies that
\[\Var(N_{j}) = \frac{1}{n}X_{j}^{T}Q_{j}X_{j} = \Omega\lb \frac{\tr(\Cov(h_{j, 0}))}{n \polyLog} \rb.\]
It is left to show that $ \tr(\Cov(h_{j, 0})) / n = \Omega\lb\frac{1}{\polyLog}\rb$. Since this result will also be used later in Appendix \ref{app:others}, we state it in the following the lemma.
\begin{lemma}\label{lem:var_rij}
Under assumptions \textbf{A}1 - \textbf{A}3,
  \[\frac{\tr(\Cov(\psi(h_{j, 0})))}{n}\ge \frac{K_{0}^{4}}{K_{1}^{2}}\cdot \lb \frac{n - p + 1}{n}\rb^{2}\cdot \min_{i}\Var(\eps_{i}) = \Omega\lb\frac{1}{\polyLog}\rb.\]
\end{lemma}
\begin{proof}
The \eqref{eq:varfunc} implies that
\begin{equation}\label{eq:remove_psi}
\Var(\psi(r_{i, [j]}))\ge K_{0}^{2}\Var(r_{i, [j]}).
\end{equation}
Note that $r_{i, [j]}$ is a function of $\eps$, we can apply \eqref{eq:varfunc} again to obtain a lower bound for $\Var(r_{i, [j]})$. In fact, by variance decomposition formula, using the independence of $\eps'_{i}$s,
\begin{align*}
\Var(r_{i, [j]}) &= \E \lb\Var\lb r_{i, [j]} \big| \eps_{(i)}\rb\rb + \Var\lb\E \lb r_{i, [j]}\big| \eps_{(i)}\rb\rb \ge \E \lb\Var\lb r_{i, [j]} \big| \eps_{(i)}\rb\rb,
\end{align*}
where $\eps_{(i)}$ includes all but the $i$-th entry of $\eps$. Apply \ref{eq:varfunc} again, 
\[\Var\lb r_{i, [j]} \big|\eps_{(i)}\rb\ge \inf_{\eps_{i}} \bigg|\pd{r_{i, [j]}}{\eps_{i}}\bigg|^{2}\cdot \Var(\eps_{i}),\]
and hence
\begin{equation}\label{eq:var_low}
\Var(r_{i, [j]})\ge \E \Var\lb r_{i, [j]} \big|\eps_{(i)}\rb\ge \E\inf_{\eps} \bigg|\pd{r_{i, [j]}}{\eps_{i}}\bigg|^{2}\cdot \Var(\eps_{i}).
\end{equation}
Now we compute $\pd{r_{i, [j]}}{\eps_{i}}$. Similar to \eqref{eq:gradient_Ri} in p.\pageref{eq:gradient_Ri}, % Lemma \ref{lem:beta_deriv}, 
% \[\pd{\hat{\beta}_{[j]}}{\eps_{k}} = (\wmj)^{-1}X_{[j]}^{T}D_{[j]}e_{k}.\]
% This implies that
% \begin{align*}
% \pd{r_{i, [j]}}{\eps_{k}} &= \pd{(\eps_{i} - x_{i, [j]}^{T}\hat{\beta}_{[j]})}{\eps_{k}} = I(i = k) - x_{i, [j]}^{T}\pd{\hat{\beta}_{[j]}}{\eps_{k}}\\
% & = I(i = k) - x_{i, [j]}^{T}(\wmj)^{-1}X_{[j]}^{T}D_{[j]}e_{k}\\
% & = e_{i}^{T}\lb I - X_{[j]}(\wmj)^{-1}X_{[j]}D_{[j]}\rb e_{k}\\
% & = e_{i}^{T}G_{[j]}e_{k}
% \end{align*}
we have
\begin{equation}
  \label{eq:gradient_rij}
  \pd{r_{k, [j]}}{\eps_{i}} = e_{i}^{T}G_{[j]}e_{k},
\end{equation}
where $G_{[j]}$ is defined in \eqref{eq:GGj} in p.\pageref{eq:GGj}. When $k = i$, 
\begin{equation}\label{eq:deriv_rij}
\pd{r_{i, [j]}}{\eps_{i}} = e_{i}^{T}G_{[j]}e_{i} = e_{i}^{T}D_{[j]}^{-\frac{1}{2}}D_{[j]}^{\frac{1}{2}}G_{[j]}D_{[j]}^{-\frac{1}{2}}D_{[j]}^{\frac{1}{2}}e_{i} = e_{i}^{T}D_{[j]}^{\frac{1}{2}}G_{[j]}D_{[j]}^{-\frac{1}{2}}e_{i}.
\end{equation}
By definition of $G_{[j]}$,
\[D_{[j]}^{\frac{1}{2}}G_{[j]}D_{[j]}^{-\frac{1}{2}} = I - D_{[j]}^{\frac{1}{2}}X_{[j]}(\wmj)^{-1}X_{[j]}^{T}D_{[j]}^{\frac{1}{2}}.\]
Let $\td{X}_{[j]} = D_{[j]}^{\frac{1}{2}}X_{[j]}$ and $H_{j} = \td{X}_{[j]}(\td{X}_{[j]}^{T}\td{X}_{[j]})^{-1}\td{X}_{[j]}^{T}$. Denote by $\td{X}_{(i), [j]}$ the matrix $\td{X}_{[j]}$ after removing $i$-th row, then by block matrix inversion formula (See Proposition \ref{prop:block_inv}), 
\begin{align*}
e_{i}^{T}H_{j}e_{i} &= \td{x}_{i, [j]}^{T}(\td{X}_{(i), [j]}^{T}\td{X}_{(i), [j]} + \td{x}_{i, [j]}\td{x}_{i, [j]}^{T})^{-1}\td{x}_{i, [j]}\\
& = \td{x}_{i, [j]}^{T}\lb (\td{X}_{(i), [j]}^{T}\td{X}_{(i), [j]})^{-1} - \frac{(\td{X}_{(i), [j]}^{T}\td{X}_{(i), [j]})^{-1}\td{x}_{i, [j]}\td{x}_{i, [j]}^{T}(\td{X}_{(i), [j]}^{T}\td{X}_{(i), [j]})^{-1}}{1 + \td{x}_{i, [j]}^{T}(\td{X}_{(i), [j]}^{T}\td{X}_{(i), [j]})^{-1}\td{x}_{i, [j]}}\rb\td{x}_{i, [j]}\\
& = \frac{\td{x}_{i, [j]}^{T}(\td{X}_{(i), [j]}^{T}\td{X}_{(i), [j]})^{-1}\td{x}_{i, [j]}}{1 + \td{x}_{i, [j]}^{T}(\td{X}_{(i), [j]}^{T}\td{X}_{(i), [j]})^{-1}\td{x}_{i, [j]}}.
\end{align*}
This implies that
\begin{align}
e_{i}^{T}D_{[j]}^{\frac{1}{2}}G_{[j]}D_{[j]}^{-\frac{1}{2}}e_{i} & = e_{i}^{T}(I - H_{j})e_{i} = \frac{1}{1 + \td{x}_{i, [j]}^{T}(\td{X}_{(i), [j]}^{T}\td{X}_{(i), [j]})^{-1}\td{x}_{i, [j]}}\nonumber\\
& = \frac{1}{1 + e_{i}^{T}D_{[j]}^{\frac{1}{2}}X_{[j]}(X_{(i), [j]}^{T}D_{(i), [j]}X_{(i), [j]})^{-1}X_{[j]}^{T}D_{[j]}^{\frac{1}{2}}e_{i}}\nonumber\\
& \ge \frac{1}{1 + K_{0}^{-1}e_{i}^{T}D_{[j]}^{\frac{1}{2}}X_{[j]}(X_{(i), [j]}^{T}X_{(i), [j]})^{-1}X_{[j]}^{T}D_{[j]}^{\frac{1}{2}}e_{i}}\nonumber\\
& = \frac{1}{1 + K_{0}^{-1}(D_{[j]})_{i, i}\cdot e_{i}^{T}X_{[j]}(X_{(i), [j]}^{T}X_{(i), [j]})^{-1}X_{[j]}^{T}e_{i}}\nonumber\\
& \ge \frac{1}{1 + K_{0}^{-1}K_{1}e_{i}^{T}X_{[j]}(X_{(i), [j]}^{T}X_{(i), [j]})^{-1}X_{[j]}^{T}e_{i}}\nonumber\\
& \ge \frac{K_{0}}{K_{1}}\cdot \frac{1}{1 + e_{i}^{T}X_{[j]}(X_{(i), [j]}^{T}X_{(i), [j]})^{-1}X_{[j]}^{T}e_{i}}.\label{eq:Qjii}
\end{align}
Apply the above argument to $H_{j} = X_{[j]}(X_{[j]}^{T}X_{[j]})^{-1}X_{[j]}^{T}$, we have
\[\frac{1}{1 + e_{i}^{T}X_{[j]}^{T}(X_{(i), [j]}^{T}X_{(i), [j]})^{-1}X_{[j]}e_{i}} = e_{i}^{T}(I - X_{[j]}(X_{[j]}^{T}X_{[j]})^{-1}X_{[j]}^{T})e_{i}.\]
Thus, by \eqref{eq:remove_psi} and \eqref{eq:var_low},
\[\Var(\psi(r_{i, [j]}))\ge \frac{K_{0}^{4}}{K_{1}^{2}}\cdot [e_{i}^{T}(I - X_{[j]}(X_{[j]}^{T}X_{[j]})^{-1}X_{[j]}^{T})e_{i}]^{2}.\]
Summing $i$ over $1, \ldots, n$, we obtain that 
\begin{align*}
\frac{\tr(\Cov(h_{j, 0}))}{n}&\ge \frac{K_{0}^{4}}{K_{1}^{2}}\cdot \frac{1}{n}\sum_{i=1}^{n}[e_{i}^{T}(I - X_{[j]}(X_{[j]}^{T}X_{[j]})^{-1}X_{[j]}^{T})e_{i}]^{2}\cdot \min_{i}\Var(\eps_{i})\\
& \ge \frac{K_{0}^{4}}{K_{1}^{2}}\cdot \lb\frac{1}{n}\tr(I - X_{[j]}(X_{[j]}^{T}X_{[j]})^{-1}X_{[j]}^{T})\rb^{2}\cdot \min_{i}\Var(\eps_{i})\\
& = \frac{K_{0}^{4}}{K_{1}^{2}}\cdot \lb\frac{n - p + 1}{n}\rb^{2}\cdot \min_{i}\Var(\eps_{i})
\end{align*}
Since $\min_{i}\Var(\eps_{i}) = \Omega\lb\frac{1}{\polyLog}\rb$ by assumption \textbf{A}2, we conclude that 
\[\frac{\tr(\Cov(h_{j, 0}))}{n} = \Omega\lb\frac{1}{\polyLog}\rb.\]
\end{proof}
In summary,
\[\Var(N_{j}) = \Omega\lb\frac{1}{\polyLog}\rb.\]
Recall that 
\[\xi_{j}= \frac{1}{n}X_{j}^{T}(D_{[j]} - D_{[j]}X_{[j]}(\wmj)^{-1}X_{[j]}^{T}D_{[j]})X_{j}\le \frac{1}{n}X_{j}^{T}D_{[j]}X_{j}\le K_{1}T^{2},\]
we conclude that
\begin{equation}\label{eq:I1}
I_{1}\ge \frac{\Var(N_{j})}{(K_{1}T^{2})^{2}} = \Omega\lb\frac{1}{\polyLog}\rb.
\end{equation}

\subsubsection{Controlling $I_{2}$}\label{subsubsec:I2}
By definition,
\begin{align}
I_{2} &= \E \lb \E N_{j}\lb \frac{1}{\xi_{j}} - \E \frac{1}{\xi_{j}}\rb + \E N_{j}\E \frac{1}{\xi_{j}} - \E \frac{N_{j}}{\xi_{j}}\rb^{2}\nonumber\\
& = \Var\lb\frac{\E N_{j}}{\xi_{j}}\rb + \lb\E N_{j}\E \frac{1}{\xi_{j}} - \E \frac{N_{j}}{\xi_{j}}\rb^{2}\nonumber\\
& = (\E N_{j})^{2}\cdot \Var\lb\frac{1}{\xi_{j}}\rb + \Cov\lb N_{j}, \frac{1}{\xi_{j}}\rb^{2}\nonumber\\
& \le (\E N_{j})^{2}\cdot \Var\lb\frac{1}{\xi_{j}}\rb + \Var(N_{j})\Var\lb\frac{1}{\xi_{j}}\rb\nonumber\\
& = \E N_{j}^{2}\cdot \Var\lb\frac{1}{\xi_{j}}\rb.\label{eq:I2}
\end{align}
By \eqref{eq:Nj} in the proof of Theorem \ref{thm:mainapprox},
\[\E N_{j}^{2}\le 2K_{1}\E (\cure\cdot \Delta_{C}^{2})\le 2K_{1}\sqrt{\E \cure^{2}\cdot \E \Delta_{C}^{4}} = O\lb\polyLog\rb,\]
where the last equality uses the fact that $\cure = \kbigo{2}{\polyLog}$ as proved in \eqref{eq:cure}. On the other hand, let $\td{\xi}_{j}$ be an independent copy of $\xi_{j}$, then
\begin{align*}
\Var\lb\frac{1}{\xi_{j}}\rb &= \frac{1}{2}\E \lb\frac{1}{\xi_{j}} - \frac{1}{\td{\xi}_{j}}\rb^{2} = \frac{1}{2}\E \frac{(\xi_{j} - \td{\xi}_{j})^{2}}{\xi_{j}^{2}\td{\xi}_{j}^{2}}.
\end{align*}
Since $\xi_{j}\ge K_{0}\lammin$ as shown in \eqref{eq:xij}, we have
\begin{equation}\label{eq:var_inv_xij}
\Var\lb\frac{1}{\xi_{j}}\rb\le \frac{1}{2(K_{0}\lammin)^{4}}\E (\xi_{j} - \td{\xi}_{j})^{2} = \frac{1}{(K_{0}\lammin)^{4}}\cdot \Var(\xi_{j}).
\end{equation}
To bound $\Var(\xi_{j})$, we propose to using the standard Poincar\'{e} inequality \cite{chernoff81}, which is stated as follows.
\begin{proposition}\label{prop:poincare}
  Let $W = (W_{1}, \ldots, W_{n})\sim N(0, I_{n\times n})$ and $f$ be a twice differentiable function, then
\[\Var(f(W))\le \E \norm{2}{\pd{f(W)}{W}}^{2}.\]
\end{proposition}
In our case, $\eps_{i} = u_{i}(W_{i})$, and hence for any twice differentiable function $g$,
\[\Var(g(\eps))\le \E \norm{2}{\pd{g(\eps)}{W}}^{2} = \E \norm{2}{\pd{g(\eps)}{\eps}\cdot \pd{\eps}{W^{T}}}^{2}\le \rmax\norm{\infty}{u'_{i}}^{2}\cdot \E \norm{2}{\pd{g(\eps)}{\eps}}^{2}.\]
Applying it to $\xi_{j}$, we have
\begin{equation}\label{eq:poincarexij}
\Var(\xi_{j})\le c_{1}^{2}\cdot\E \norm{2}{\pd{\xi_{j}}{\eps}}^{2}.
\end{equation}
For given $k\in\{1, \ldots, n\}$, using the chain rule and the fact that $dB^{-1} = -B^{-1}dB B^{-1}$ for any square matrix $B$, we obtain that
\begin{align*}
&\pd{}{\eps_{k}}\lb D_{[j]} - D_{[j]}X_{[j]}(\wmj)^{-1}X_{[j]}^{T}D_{[j]}\rb\\
= &\pd{D_{[j]}}{\eps_{k}} - \pd{D_{[j]}}{\eps_{k}}X_{[j]}(\wmj)^{-1}X_{[j]}^{T}D_{[j]} - D_{[j]}X_{[j]}(\wmj)^{-1}X_{[j]}^{T}\pd{D_{[j]}}{\eps_{k}}\\
&\quad + D_{[j]}X_{[j]}(\wmj)^{-1}X_{[j]}^{T}\pd{D_{[j]}}{\eps_{k}}X_{[j]}(\wmj)^{-1}X_{[j]}^{T}D_{[j]}\\
= & G_{[j]}^{T}\pd{D_{[j]}}{\eps_{k}}G_{[j]}
\end{align*}
where $G_{[j]} = I - X_{[j]}(\wmj)^{-1}X_{[j]}^{T}D_{[j]}$ as defined in last subsection. This implies that
\[\pd{\xi_{j}}{\eps_{k}} = \frac{1}{n}X_{j}^{T}G_{[j]}^{T}\pd{D_{[j]}}{\eps_{k}}G_{[j]}X_{j}.\]
Then (\ref{eq:poincarexij}) entails that
\begin{equation}\label{eq:poincarexij2}
\Var(\xi_{j})\le \frac{1}{n^{2}}\sum_{k=1}^{n}\E\lb X_{j}^{T}G_{[j]}^{T}\pd{D_{[j]}}{\eps_{k}}G_{[j]}X_{j}\rb^{2}
\end{equation}
First we compute $\pd{D_{[j]}}{\eps_{k}}$. Similar to \eqref{eq:gradient_D} in p.\pageref{eq:gradient_D} and recalling the definition of $D_{[j]}$ in \eqref{eq:DjtdDj} and that of $G_{[j]}$ in \eqref{eq:GGj} in p.\pageref{eq:GGj}, we have
% For given $i$, it follows from \eqref{eq:deriv_rij} that
% \[\pd{\psi'(r_{i, [j]})}{\eps_{k}} = \psi''(r_{i, [j]})\cdot \pd{r_{i, [j]}}{\eps_{k}} = \psi''(r_{i, [j]})e_{i}^{T}G_{[j]}e_{k}.\]
% Rewriting it into matrix form, we have
\[\pd{D_{[j]}}{\eps_{k}} =\td{D}_{[j]}\diag(G_{[j]}e_{k}) \diag(\td{D}_{[j]}G_{[j]}e_{k}),\]
% where $\td{D}_{[j]} = \diag(\psi''(r_{i, [j]}))$ as defined in Appendix \ref{app:notation}.1.
Let $\mathcal{X}_{j} = G_{[j]}X_{j}$ and $\td{\mathcal{X}}_{j} = \mathcal{X}_{j}\circ \mathcal{X}_{j}$ where $\circ$ denotes Hadamard product. Then
\begin{align*}
X_{j}^{T}G_{[j]}^{T}\pd{D_{[j]}}{\eps_{k}}G_{[j]}X_{j} & = \mathcal{X}_{j}^{T}\pd{D_{[j]}}{\eps_{k}}\mathcal{X}_{j} = \mathcal{X}_{j}^{T}\diag(\td{D}_{[j]}G_{[j]}e_{k})\mathcal{X}_{j} = \td{\mathcal{X}}_{j}^{T}\td{D}_{[j]}G_{[j]}e_{k}.
\end{align*}
Here we use the fact that for any vectors $x, a\in\R^{n}$, 
\[x^{T}\diag(a)x = \sum_{i=1}^{n}a_{i}x_{i}^{2} = (x\circ x)^{T}a.\]
This together with (\ref{eq:poincarexij2}) imply that
\begin{align*}
\Var(\xi_{j})&\le \frac{1}{n^{2}}\sum_{k=1}^{n}\E(\td{\mathcal{X}}_{j}^{T}\td{D}_{[j]}G_{[j]}e_{k})^{2} = \frac{1}{n^{2}}\E\norm{2}{\td{\mathcal{X}}_{j}^{T}\td{D}_{[j]}G_{[j]}}^{2} = \frac{1}{n^{2}}\E(\td{\mathcal{X}}_{j}^{T}\td{D}_{[j]}G_{[j]}G_{[j]}^{T}\td{D}_{[j]}\td{\mathcal{X}}_{j})
\end{align*}
Note that $G_{[j]}G_{[j]}^{T}\preceq \|G_{[j]}\|_{\mathrm{op}}^{2} I$, and $\td{D}_{[j]}\preceq K_{3}I$ by Lemma \ref{lem:psi} in p.\eqref{lem:psi}. Therefore we obtain that 
\begin{align*}
\Var(\xi_{j})&\le \frac{1}{n^{2}}\E \lb\norm{op}{G_{[j]}}^{2}\cdot \td{\mathcal{X}}_{j}^{T}\td{D}_{[j]}^{2}\td{\mathcal{X}}_{j} \rb\le \frac{K_{3}^{2}}{n^{2}}\cdot \E\lb\norm{op}{G_{[j]}}^{2}\cdot \|\td{\mathcal{X}}_{j}\|_{2}^{2}\rb\\ 
& = \frac{K_{3}^{2}}{n^{2}}\E \lb\norm{op}{G_{[j]}}^{2}\cdot \norm{4}{\mathcal{X}_{j}}^{4}\rb\le \frac{K_{3}^{2}}{n}\E \lb\norm{op}{G_{[j]}}^{2}\cdot \norm{\infty}{\mathcal{X}_{j}}^{4}\rb
\end{align*}
As shown in \eqref{eq:Gj}, 
\[\|G_{[j]}\|_{\mathrm{op}}\le \lb\frac{K_{1}}{K_{0}}\rb^{\frac{1}{2}}.\]
On the other hand, notice that the $i$-th row of $G_{[j]}$ is $h_{j, 1, i}$ (see \eqref{eq:hj} for definition), by definition of $\Delta_{C}$ we have
\[\|\mathcal{X}_{j}\|_{\infty} = \|G_{[j]}X_{j}\|_{\infty} = \max_{i}|h_{j, 1, i}^{T}X_{j}|\le \Delta_{C}\cdot \max\|h_{j, 1, i}\|_{2}.\]
By \eqref{eq:hj1i} and assumption \textbf{A}5, 
\[\|\mathcal{X}_{j}\|_{\infty}\le \Delta_{C}\cdot \lb\frac{K_{1}}{K_{0}}\rb^{\frac{1}{2}} = \kbigo{4}{\polyLog}.\]
This entails that
\[\Var(\xi_{j}) = O\lb\frac{\polyLog}{n}\rb.\]
Combining with \eqref{eq:I2} and \eqref{eq:var_inv_xij}, we obtain that
\[I_{2} = O\lb\frac{\polyLog}{n}\rb.\]

\subsubsection{Summary}
Putting \eqref{eq:var_bj}, \eqref{eq:I1} and \eqref{eq:I2} together, we conclude that
\[n\Var(b_{j}) = \Omega\lb\frac{1}{\polyLog}\rb - O\lb\frac{1}{n\cdot \polyLog}\rb = \Omega\lb\frac{1}{\polyLog}\rb\lra \Var(b_{j}) = \Omega\lb\frac{\polyLog}{n}\rb.\]
Combining with \eqref{eq:approxbetajbj}, 
\[\Var(\hat{\beta}_{j}) = \Omega\lb\frac{\polyLog}{n}\rb.\]

\section{Proof of Other Results}\label{app:others}

\subsection{Proofs of Propositions in Section \ref{subsec:Identifiability}}

\begin{proof}[\textbf{Proof of Proposition \ref{prop:mest_id_fix_sym}}]
Let $H_{i}(\alpha) = \E \rho(\eps_{i} - \alpha)$. First we prove that the conditions imply that $0$ is the unique minimizer of $H_{i}(\alpha)$ for all $i$. In fact, since $\eps_{i}\stackrel{d}{=}-\eps_{i}$,
\[H_{i}(\alpha) = \E \rho(\eps_{i} - \alpha) = \frac{1}{2}\lb \E \rho(\eps_{i} - \alpha) + \rho(-\eps_{i} - \alpha)\rb.\]
Using the fact that $\rho$ is even, we have
\[H_{i}(\alpha) = \E \rho(\eps_{i} - \alpha) = \frac{1}{2}\lb \E \rho(\eps_{i} - \alpha) + \rho(\eps_{i} + \alpha)\rb.\]
By \eqref{eq:symmetric_rho}, for any $\alpha \not= 0$, $H_{i}(\alpha) > H_{i}(0)$. As a result, $0$ is the unique minimizer of $H_{i}$. Then for any $\beta\in \R^{p}$
\[\frac{1}{n}\sum_{i=1}^{n}\E \rho(y_{i} - x_{i}^{T}\beta) = \frac{1}{n}\sum_{i=1}^{n}\E \rho(\eps_{i} - x_{i}^{T}(\beta - \betanull)) = \frac{1}{n}\sum_{i=1}^{n}H_{i}(x_{i}^{T}(\beta - \betanull))\ge \frac{1}{n}\sum_{i=1}^{n}H_{i}(0).\]
The equality holds iff $x_{i}^{T}(\beta - \betanull) = 0$ for all $i$ since $0$ is the unique minimizer of $H_{i}$. This implies that 
\[X(\beta^{*}(\rho) - \betanull) = 0.\]
Since $X$ has full column rank, we conclude that 
\[\beta^{*}(\rho) = \betanull.\]
\end{proof}

\begin{proof}[\textbf{Proof of Proposition \ref{prop:mest_id_fix}}]
For any $\alpha\in \R$ and $\beta\in \R^{p}$, let 
\[G(\alpha; \beta) = \frac{1}{n}\sum_{i=1}^{n}\E \rho(y_{i} - \alpha - x_{i}^{T}\beta).\]
Since $\alpha_{\rho}$ minimizes $\E \rho(\eps_{i} - \alpha)$, it holds that
\[G(\alpha; \beta) = \frac{1}{n}\sum_{i=1}^{n}\E \rho(\eps_{i} - \alpha - x_{i}^{T}(\beta - \betanull))\ge \frac{1}{n}\sum_{i=1}^{n}\E \rho(\eps_{i} - \alpha_{\rho}) = G(\alpha_{\rho}, \betanull).\]
Note that $\alpha_{\rho}$ is the unique minimizer of $\E \rho(\eps_{i} - \alpha)$, the above equality holds if and only if
\[\alpha + x_{i}^{T}(\beta - \betanull) \equiv \alpha_{\rho}\Longrightarrow (\textbf{1}\,\, X) \com{\alpha - \alpha_{\rho}}{\beta - \betanull} = 0.\]
Since $(\textbf{1}\,\, X)$ has full column rank, it must hold that $\alpha = \alpha_{\rho}$ and $\beta = \betanull$.
\end{proof}

\subsection{Proofs of Corollary \ref{cor:transform}}\label{subapp:partial_id}
\begin{proposition}\label{prop:partial_id}
Suppose that $\eps_{i}$ are i.i.d. such that $\E \rho(\eps_{1} - \alpha)$ as a function of $\alpha$ has a unique minimizer $\alpha_{\rho}$. Further assume that $X_{J_{n}^{c}}$ contains an intercept term, $X_{J_{n}}$ has full column rank and 
% \[X_{J_{n}}^{T}(I - X_{J_{n}^{c}}(X_{J_{n}^{c}}^{T}X_{J_{n}^{c}})^{-}X_{J_{n}^{c}})X_{J_{n}}\]
\begin{equation}\label{eq:partial_id}
\spanvec(\{X_{j}: j\in J_{n}\})\cap \spanvec(\{X_{j}: j\in J_{n}^{c}\}) = \{0\}  
\end{equation}
% is non-singular where $A^{-}$ denotes the generalized inverse of $A$. 
Let
\[\beta_{J_{n}}(\rho)= \argmin_{\beta_{J_{n}}}\left\{\min_{\beta_{J_{n}^{c}}}\frac{1}{n}\sum_{i=1}^{n}\E \rho(y_{i} - x_{i}^{T}\beta)\right\}.\]
Then $\beta_{J_{n}}(\rho) = \betanull_{J_{n}}$.
\end{proposition}
\begin{proof}
let 
\[G(\beta) = \frac{1}{n}\sum_{i=1}^{n}\E \rho(y_{i} - x_{i}^{T}\beta).\]
For any minimizer $\beta(\rho)$ of $G$, which might not be unique, we prove that $\beta_{J_{n}}(\rho) = \betanull_{J_{n}}$. It follows by the same argument as in Proposition \ref{prop:mest_id_fix} that 
\[x_{i}^{T}(\beta(\rho) - \betanull)\equiv \alpha_{0}\Longrightarrow X(\beta(\rho) - \betanull) = \alpha_{0}\textbf{1}\Longrightarrow X_{J_{n}}(\beta_{J_{n}}(\rho)) = -X_{J_{n}^{c}}(\beta(\rho)_{J_{n}^{c}} - \betanull_{J_{n}^{c}}) + \alpha_{0}\textbf{1}.\]
Since $X_{J_{n}^{c}}$ contains the intercept term, we have 
\[X_{J_{n}}(\beta_{J_{n}}(\rho) - \betanull_{J_{n}})\in \spanvec(\{X_{j}: j \in J_{n}^{c}\}).\]
It then follows from \eqref{eq:partial_id} that 
\[X_{J_{n}}(\beta_{J_{n}}(\rho) - \betanull_{J_{n}}) = 0.\]
Since $X_{J_{n}}$ has full column rank, we conclude that
\[\beta_{J_{n}}(\rho) = \betanull_{J_{n}}.\]
\end{proof}

The Proposition \ref{prop:partial_id} implies that $\betanull_{J_{n}}$ is identifiable even when $X$ is not of full column rank. A similar conclusion holds for the estimator $\hat{\beta}_{J_{n}}$ and the residuals $R_{i}$. The following two propositions show that under certain assumptions, $\hat{\beta}_{J_{n}}$ and $R_{i}$ are invariant to the choice of $\hat{\beta}$ in the presense of multiple minimizers.

\begin{proposition}\label{prop:partial_id_resid}
Suppose that $\rho$ is convex and twice differentiable with $\rho''(x) > c > 0$ for all $x\in \R$. Let $\hat{\beta}$ be any minimizer, which might not be unique, of 
\[F(\beta)\triangleq\frac{1}{n}\sum_{i=1}^{n}\rho(y_{i} - x_{i}^{T}\beta)\]
Then $R_{i} = y_{i} - x_{i}\hat{\beta}$ is independent of the choice of $\hat{\beta}$ for any $i$.
\end{proposition}
\begin{proof}
The conclusion is obvious if $F(\beta)$ has a unique minimizer. Otherwise, let $\hat{\beta}^{(1)}$ and $\hat{\beta}^{(2)}$ be two different minimizers of $F$ denote by $\eta$ their difference, i.e. $\eta = \hat{\beta}^{(2)} - \hat{\beta}^{(1)}$. Since $F$ is convex, $\hat{\beta}^{(1)} + v\eta$ is a minimizer of $F$ for all $v\in [0, 1]$. By Taylor expansion, 
\[F(\hat{\beta}^{(1)} + v\eta) = F(\hat{\beta}^{(1)}) + v\nabla F(\hat{\beta}^{(1)}) \eta + \frac{v^{2}}{2}\eta^{T}\nabla^{2}F(\hat{\beta}^{(1)})\eta + o(v^{2}).\]
Since both $\hat{\beta}^{(1)} + v\eta$ and $\hat{\beta}^{(1)}$ are minimizers of $F$, we have $F(\hat{\beta}^{(1)} + v\eta) = F(\hat{\beta}^{(1)})$ and $\nabla F(\hat{\beta}^{(1)}) = 0$. By letting $v$ tend to $0$, we conclude that
\[\eta^{T}\nabla^{2}F(\hat{\beta}^{(1)})\eta = 0.\]
The hessian of $F$ can be written as
\[\nabla^{2}F(\hat{\beta}^{(1)}) = \frac{1}{n}X^{T}\diag(\rho''(y_{i} - x_{i}^{T}\hat{\beta}^{(1)}))X\succeq \frac{cX^{T}X}{n}.\]
Thus, $\eta$ satisfies that
\begin{equation}\label{eq:resid_id}
\eta^{T}\frac{cX^{T}X}{n}\eta = 0\Longrightarrow X\eta = 0.
\end{equation}
This implies that 
\[y - X\hat{\beta}^{(1)} = y - X\hat{\beta}^{(2)}\]
and hence $R_{i}$ is the same for all $i$ in both cases.
\end{proof}

\begin{proposition}\label{prop:partial_id_estimator}
Suppose that $\rho$ is convex and twice differentiable with $\rho''(x) > c > 0$ for all $x\in \R$. Further assume that $X_{J_{n}}$ has full column rank and 
\begin{equation}\label{eq:partial_id}
\spanvec(\{X_{j}: j\in J_{n}\})\cap \spanvec(\{X_{j}: j\in J_{n}^{c}\}) = \{0\}  
\end{equation}
Let $\hat{\beta}$ be any minimizer, which might not be unique, of 
\[F(\beta)\triangleq\frac{1}{n}\sum_{i=1}^{n}\rho(y_{i} - x_{i}^{T}\beta)\]
Then $\hat{\beta}_{J_{n}}$ is independent of the choice of $\hat{\beta}$.
\end{proposition}
\begin{proof}
As in the proof of Proposition \ref{prop:partial_id_resid}, we conclude that for any minimizers $\hat{\beta}^{(1)}$ and $\hat{\beta}^{(2)}$, $X\eta = 0$ where $\eta = \hat{\beta}^{(2)} - \hat{\beta}^{(1)}$. Decompose the term into two parts, we have
\[X_{J_{n}}\eta_{J_{n}} = -X_{J_{n}}^{c}\eta_{J_{n}^{c}} \in \spanvec(\{X_{j}: j \in J_{n}^{c}\}).\]
It then follows from \eqref{eq:partial_id} that $X_{J_{n}}\eta_{J_{n}} = 0$. Since $X_{J_{n}}$ has full column rank, we conclude that $\eta_{J_{n}}= 0$ and hence $\hat{\beta}^{(1)}_{J_{n}} = \hat{\beta}^{(2)}_{J_{n}}$.
\end{proof}

\begin{proof}[\textbf{Proof of Corollary \ref{cor:transform}}]
Under assumption \textbf{A}3*, $X_{J_{n}}$ must have full column rank. Otherwise there exists $\alpha\in \R^{|J_{n}|}$ such that $X_{J_{n}}\alpha$, in which case $\alpha^{T}X_{J_{n}}^{T}(I - H_{J_{n}^{c}})X_{J_{n}}\alpha = 0$. This violates the assumption that $\td{\lambda}_{-} > 0$. On the other hand, it also guarantees that 
\[\spanvec(\{X_{j}: j\in J_{n}\})\cap \spanvec(\{X_{j}: j\in J_{n}^{c}\}) = \{0\}.\]
This together with assumption \textbf{A}1 and Proposition \ref{prop:partial_id_estimator} implies that $\hat{\beta}_{J_{n}}$ is independent of the choice of $\hat{\beta}$.

~\\
\noindent Let $B_{1}\in \R^{|J_{n}^{c}|\times |J_{n}|}$, $B_{2}\in \R^{|J_{n}^{c}|\times |J_{n}^{c}|}$ and assume that $B_{2}$ is invertible. Let $\td{X}\in \R^{n\times p}$ such that 
\[\td{X}_{J_{n}} = X_{J_{n}} - X_{J_{n}^{c}}B_{1}, \quad \td{X}_{J_{n}^{c}} = X_{J_{n}^{c}}B_{2}.\]
Then $\rank(X) = \rank(\td{X})$ and model \eqref{eq:linearmodel} can be rewritten as 
\[y = \td{X}\td{\betanull} + \eps\]
where 
\[\td{\beta}_{J_{n}}^{*} = \betanull_{J_{n}}, \quad \td{\beta}_{J_{n}^{c}}^{*} = B_{2}^{-1}\betanull_{J_{n}^{c}} + B_{1}\betanull_{J_{n}}.\]
Let $\td{\hat{\beta}}$ be an M-estimator, which might not be unique, based on $\td{X}$. Then Proposition \ref{prop:partial_id_estimator} shows that $\td{\hat{\beta}}_{J_{n}}$ is independent of the choice of $\td{\hat{\beta}}$, and an invariance argument shows that 
\[\td{\hat{\beta}}_{J_{n}} = \hat{\beta}_{J_{n}}.\]
In the rest of proof, we use $\td{\cdot}$ to denote the quantity obtained based on $\td{X}$. First we show that the assumption \textbf{A}4 is not affected by this transformation. In fact, for any $j\in J_{n}$, by definition we have
\[\spanvec(\td{X}_{[j]}) = \spanvec(X_{[j]})\]
and hence the leave-$j$-th-predictor-out residuals are not changed by Proposition \ref{prop:partial_id_resid}. This implies that $\td{h_{j, 0}} = h_{j, 0}$ and $\td{Q}_{j} = Q_{j}$. Recall the definition of $h_{j, 0}$, the first-order condition of $\hat{\beta}$ entails that $X^{T}h_{j, 0} = 0$. In particular, $X_{J_{n}^{c}}^{T}h_{j, 0} = 0$ and this implies that for any $\alpha \in \R^{n}$, 
\[0 = \Cov(X_{J_{n}^{c}}^{T}h_{j, 0}, \alpha^{T}h_{j, 0}) = X_{J_{n}^{c}}Q_{j}\alpha.\]
Thus, 
\[\frac{\td{X}_{j}^{T}\td{Q}_{j}\td{X}_{j}}{\tr(\td{Q}_{j})} = \frac{(X_{j} - X_{J_{n}}^{c}(B_{1})_{j})^{T}Q_{j}(X_{j} - X_{J_{n}^{c}}(B_{1})_{j})}{\tr(Q_{j})} = \frac{X_{j}^{T}Q_{j}X_{j}}{\tr(Q_{j})}.\]
Then we prove that the assumption \textbf{A}5 is also not affected by the transformation. The above argument has shown that 
\[\frac{\td{h}_{j, 0}^{T}\td{X}_{j}}{\|\td{h}_{j, 0}\|_{2}} = \frac{h_{j, 0}^{T}X_{j}}{\|h_{j, 0}\|_{2}}.\]
On the other hand, let $B = \lb
\begin{array}{cc}
 I_{|J_{n}|} & 0 \\
 -B_{1} & B_{2}
\end{array}
\rb$, then $B$ is non-singular and $\td{X} = XB$. Let $B_{(j),[j]}$ denote the matrix $B$ after removing $j$-th row and $j$-th column. Then $B_{(j),[j]}$ is also non-singular and $\td{X}_{[j]} = X_{[j]}B_{(j), [j]}$. Recall the definition of $h_{j, 1, i}$, we have
\begin{align*}
\td{h}_{j, 1, i} &= (I - \td{D}_{[j]}\td{X}_{[j]}(\td{X}_{[j]}^{T}\td{D}_{[j]}\td{X}_{j})^{-1}\td{X}_{[j]}^{T})e_{i}\\
& = (I - D_{[j]}X_{[j]}B_{(j), [j]}(B_{(j), [j]}^{T}X_{[j]}^{T}D_{[j]}X_{j}B_{(j), [j]})^{-1}B_{(j), [j]}^{T}X_{[j]})e_{i}\\
& = (I - D_{[j]}X_{[j]}(X_{[j]}^{T}D_{[j]}X_{j})^{-1}X_{[j]})e_{i}\\
& = h_{j, 1, i}.
\end{align*}
On the other hand, by definition,
\[X_{[j]}^{T}h_{j, 1, i} = X_{[j]}^{T}(I - D_{[j]}X_{[j]}(X_{[j]}^{T}D_{[j]}X_{[j]})^{-1}X_{[j]}^{T})e_{i} = 0.\]
Thus, 
\[h_{j, 1, i}^{T}\td{X}_{j} = h_{j, 1, i}^{T}(X_{j} - X_{J_{n}}^{c}(B_{1})_{j}) = h_{j, 1, i}^{T}X_{j}.\]
In summary, for any $j\in J_{n}$ and $i\le n$,
\[\frac{\td{h}_{j, 1, i}^{T}\td{X}_{j}}{\|\td{h}_{j, 1, i}\|_{2}} = \frac{h_{j, 1, i}^{T}X_{j}}{\|h_{j, 1, i}\|_{2}}.\]
Putting the pieces together we have 
\[\td{\Delta}_{C} = \Delta_{C}.\]
By Theorem \ref{thm:main}, 
\[\max_{j \in J_n}d_{\mathrm{TV}}\lb\mathcal{L}\lb\frac{\hat{\beta}_{j} - \E \hat{\beta}_{j}}{\sqrt{\Var(\hat{\beta}_{j})}}\rb, N(0, 1)\rb = o(1).\]
provided that $\td{X}$ satisfies the assumption \textbf{A}3. 

~\\
\noindent Now let $U\Lambda V$ be the singular value decomposition of $X_{J_{n}^{c}}$, where $U\in \R^{n\times p}, \Lambda\in \R^{p\times p}, V \in \R^{p\times p}$ with $U^{T}U = V^{T}V = I_{p}$ and $\Lambda = \diag(\nu_{1}, \ldots, \nu_{p})$ being the diagonal matrix formed by singular values of $X_{J_{n}^{c}}$. First we consider the case where $X_{J_{n}^{c}}$ has full column rank, then $\nu_{j} > 0$ for all $j\le p$. Let $B_{1} = (X_{J_{n}}^{T}X_{J_{n}})^{-}X_{J_{n}}^{T}X_{J_{n}}$ and $B_{2} = \sqrt{n / |J_{n}^{c}|}V^{T}\Lambda^{-1}$. Then 
\[\frac{\td{X}^{T}\td{X}}{n} = \frac{1}{n}\lb
  \begin{array}{cc}
    X_{J_{n}}^{T}(I - X_{J_{n}^{c}}(X_{J_{n}^{c}}^{T}X_{J_{n}^{c}})^{-1}X_{J_{n}^{c}})X_{J_{n}} & 0\\
    0 & nI
  \end{array}
\rb.\] 
This implies that 
\[\lambda_{\max}\lb\frac{\td{X}^{T}\td{X}}{n}\rb = \max\left\{\td{\lambda}_{\max}, 1\right\}, \quad \lambda_{\min}\lb\frac{\td{X}^{T}\td{X}}{n}\rb = \min\left\{\td{\lambda}_{\min}, 1\right\}.\]
The assumption \textbf{A}3* implies that 
\[\lambda_{\max}\lb\frac{\td{X}^{T}\td{X}}{n}\rb = O(\polyLog), \quad \lambda_{\min}\lb\frac{\td{X}^{T}\td{X}}{n}\rb = \Omega\lb\frac{1}{\polyLog}\rb.\]
By Theorem \ref{thm:main}, we conclude that 

~\\
\noindent Next we consider the case where $X_{J_{n}}^{c}$ does not have full column rank. We first remove the redundant columns from $X_{J_{n}}^{c}$, i.e. replace $X_{J_{n}^{c}}$ by the matrix formed by its maximum linear independent subset. Denote by $\mathbf{X}$ this matrix. Then $\spanvec(X) = \spanvec(\mathbf{X})$ and $\spanvec(\{X_{j}: j\not\in J_{n}\}) = \spanvec(\{\mathbf{X}_{j}: j\not\in J_{n}\})$. As a consequence of  Proposition \ref{prop:partial_id} and \ref{prop:partial_id_estimator}, neither $\betanull_{J_{n}}$ nor $\hat{\beta}_{J_{n}}$ is affected. Thus, the same reasoning as above applies to this case. 
\end{proof}

% \begin{proof}[\textbf{Proof of Proposition \ref{prop:OLS_casn_bounded}}]
% By condition 1, conditioning on $X_{[j]}$, $X_{j}$ has i.i.d. entries with uniformly bounded $(8 + \delta)$-th moment and non-zero variance lower bounded by $\tau^{2}$. Since both bounds do not depend on $X_{[j]}$, by \eqref{eq:moment_idempotent} in the proof of Proposition \ref{prop:OLS_casn_moment}, we have
% \begin{equation}\label{eq:moment_idempotent_condition}
% \max_{\substack{j\in J_{n}\\X_{[j]}\in \R^{n\times (p - 1)}}}\sup_{\substack{A\in\mathbb{R}^{n\times n}, A^{2} = A,\\ tr(A) = n - p + 1}}\left\{P\lb\|AX_{j}\|_{\infty} > C_{1}n^{\frac{1}{4}} \bigg| X_{[j]}\rb + P\lb X_{j}^{T}AX_{j} < C_{2}n \bigg| X_{[j]}\rb\right\} = o\lb\frac{1}{n}\rb.
% \end{equation}
% Similar to \eqref{eq:condition}-\eqref{eq:prob_order}, we conclude that 
% \[\max_{j\in J_{n}}P\lb S_{j}(X)\ge \frac{C_{1}}{\sqrt{C_{2}}}\cdot n^{-\frac{1}{4}}\rb = o\lb\frac{1}{n}\rb\]
% and hence
% \[\max_{j\in J_{n}}S_{j}(X) = O_{p}\lb\frac{1}{n^{\frac{1}{4}}}\rb.\]
% \end{proof}

\subsection{Proofs of Results in Section \ref{subsec:examples}}
First we prove two lemmas regarding the behavior of $Q_{j}$. These lemmas are needed for justifying Assumption \textbf{A}4 in the examples.
\begin{lemma}\label{lem:Qj}
Under assumptions \textbf{A}1 and \textbf{A}2, 
\[\|Q_{j}\|_{\mathrm{\mathrm{op}}}\le c_{1}^{2}\frac{K_{3}^{2}K_{1}}{K_{0}}, \quad\|Q_{j}\|_{\mathrm{F}}\le \sqrt{n}c_{1}^{2}\frac{K_{3}^{2}K_{1}}{K_{0}}\]
where $Q_{j} = \Cov(h_{j, 0})$ as defined in section \ref{app:notation}.
\end{lemma}
\begin{proof}[\textbf{Proof of Lemma \ref{lem:Qj}}]
By definition, 
\[\lnorm Q_{j}\lnorm_{\mathrm{op}} = \sup_{\alpha\in \mathbb{S}^{n - 1}}\alpha^{T}Q_{j}\alpha\]
where $\mathbb{S}^{n - 1}$ is the $n$-dimensional unit sphere. For given $\alpha\in\mathbb{S}^{n - 1}$,
\[\alpha^{T}Q_{j}\alpha = \alpha^{T}\Cov(h_{j, 0})\alpha = \Var(\alpha^{T}h_{j, 0})\]
% For convenience, we remove the subscript $j$ from all notation, that is, we write $r_{i, [j]}$ as $R_{i}$ and $X_{[j]}$ as $X$. It is left to show that
% \[\sup_{\alpha\in\mathbb{S}^{n - 1}}\Var\lb\sum_{i=1}^{n}\alpha_{i}\psi(R_{i})\rb = O\lb\polyLog\rb.\]
% Recall that 
% \[\pd{\hat{\beta}}{\eps^{T}} = (\wm)^{-1}X^{T}D,\]
% where $D = \diag(\psi(R_{1}), \ldots, \psi(R_{n}))$ and 
% \begin{align*}
% \pd{}{\eps}\lb\sum_{i=1}^{n}\alpha_{i}\psi(R_{i})\rb& = \sum_{i=1}^{n}\alpha_{i}\psi'(R_{i})\cdot\pd{R_{i}}{\eps}\\
% & = \sum_{i=1}^{n}\alpha_{i}\psi'(R_{i})\cdot\pd{(\eps_{i} - x_{i}^{T}\hat{\beta})}{\eps}\\
% & = \sum_{i=1}^{n}\alpha_{i}\psi'(R_{i})\cdot \lb e_{i} - x_{i}^{T}\pd{\hat{\beta}}{\eps^{T}}\rb\\
% & = \alpha^{T}D\lb I - X\pd{\hat{\beta}}{\eps^{T}}\rb\\
% & = \alpha^{T} \lb D - DX(\wm)^{-1}X^{T}D\rb.
% \end{align*}
It has been shown in \eqref{eq:deriv_rij} in Appendix \ref{subsubsec:I1} that
\[\pd{r_{i, [j]}}{\eps_{k}} = e_{i}^{T}G_{[j]}e_{k},\]
where $G_{[j]} = I - X_{[j]}(\wmj)^{-1}X_{[j]}^{T}D_{[j]}$. This yields that 
\begin{align*}
\pd{}{\eps}\lb\sum_{i=1}^{n}\alpha_{i}\psi(r_{i, [j]})\rb& = \sum_{i=1}^{n}\alpha_{i}\psi'(r_{i, [j]})\cdot\pd{r_{i, [j]}}{\eps} = \sum_{i=1}^{n}\alpha_{i}\psi'(r_{i, [j]})\cdot e_{i}^{T}G_{[j]} = \alpha^{T}\td{D}_{[j]}G_{[j]}.
\end{align*}
By standard Poincar\'{e} inequality (see Proposition \ref{prop:poincare}), since $\eps_{i} = u_{i}(W_{i})$, 
\begin{align*}
&\Var\lb \sum_{i=1}^{n}\alpha_{i}\psi(r_{i, [j]})\rb\le \max_{k}\lnorm u_{k}'\lnorm_{\infty}^{2}\cdot \E \bigg\|\pd{}{\eps}\lb\sum_{i=1}^{n}\alpha_{i}\psi(r_{i, [j]})\rb\bigg\|^{2}\\
\le& c_{1}^{2}\cdot \E \lb\alpha^{T}\td{D}_{[j]}G_{[j]}G_{[j]}^{T}\td{D}_{[j]}\alpha\rb\le c_{1}^{2}\E \|\td{D}_{[j]}G_{[j]}G_{[j]}^{T}\td{D}_{[j]}\|_{2}^{2}\le c_{1}^{2}\E \|\td{D}_{j}\|_{\mathrm{\mathrm{op}}}^{2}\|G_{[j]}\|_{\mathrm{\mathrm{op}}}^{2}.
\end{align*}
We conclude from Lemma \ref{lem:psi} and \eqref{eq:Gj} in Appendix \ref{subsec:approx} that
\[\|\td{D}_{[j]}\|_{\mathrm{\mathrm{op}}}\le K_{3}, \quad \|G_{[j]}\|_{\mathrm{\mathrm{op}}}^{2}\le \frac{K_{1}}{K_{0}}.\]
Therefore,
\[\lnorm Q_{j}\lnorm_{\mathrm{\mathrm{op}}} = \sup_{\alpha\in\mathbb{S}^{n - 1}}\Var\lb\sum_{i=1}^{n}\alpha_{i}\psi(R_{i})\rb \le c_{1}^{2}\frac{K_{3}^{2}K_{1}}{K_{0}}\]
and hence
\[\lnorm Q_{j}\lnorm_{\mathrm{F}} \le \sqrt{n}\lnorm Q_{j}\lnorm_{\mathrm{\mathrm{op}}} \le \sqrt{n} \cdot c_{1}^{2}\frac{K_{3}^{2}K_{1}}{K_{0}}.\]

\end{proof}

\begin{lemma}\label{lem:var_rij_copy}
Under assumptions \textbf{A}1 - \textbf{A}3,
\[\tr(Q_{j}) \ge K^{*}n = \Omega(n\cdot \polyLog),\]
where $K^{*} = \frac{K_{0}^{4}}{K_{1}^{2}}\cdot \lb \frac{n - p + 1}{n}\rb^{2}\cdot \min_{i}\Var(\eps_{i})$.
\end{lemma}
\begin{proof}
This is a direct consequence of Lemma \ref{lem:var_rij} in p.\pageref{lem:var_rij}.
\end{proof}

Throughout the following proofs, we will use several results from the random matrix theory to bound the largest and smallest singular values of $Z$. The results are shown in Appendix \ref{app:mc}. Furthermore, in contrast to other sections, the notation \emph{$P(\cdot), \E (\cdot), \Var(\cdot)$ denotes the probability, the expectation and the variance with respect to both $\eps$ and $Z$ in this section.}

\begin{proof}[\textbf{Proof of Proposition \ref{prop:subgauss}}]
By Proposition \ref{prop:rmt_bai93}, 
\[\lammax = (1 + \sqrt{\kappa})^{2} + o_{p}(1) = O_{p}(1), \quad \lammin = (1 - \sqrt{\kappa})^{2} - o_{p}(1) = \Omega_{p}(1)\]
% On the other hand, by Proposition \ref{prop:rmt_rudelson09}, for any $i$, 
% \[P\lb \sqrt{\lambda_{\mathrm{\min}}\lb\frac{Z_{(i)}^{T}Z_{(i)}}{n}\rb}\le \frac{1}{2C}(1 - \sqrt{\frac{p - 1}{n}})\rb\le 2^{-(n - p + 1)} + e^{-cn}.\]
% Thus, with probability $1 - n(2^{-(n - p + 1)} + e^{-cn})$,
% \[\lammin = \min_{i}\left\{\lambda_{\mathrm{\min}}\lb\frac{Z_{(i)}^{T}Z_{(i)}}{n}\rb\right\}\ge \frac{1}{2C}(1 - \sqrt{\frac{p - 1}{n}}).\]
% This implies that
% \[\lammin = \Omega_{p}(1)\]  
and thus the assumption \textbf{A}3 holds with high probability. By Hanson-Wright inequality (\citeNP{hanson71, rudelson13}; see Proposition \ref{prop:hanson_wright}), for any given deterministic matrix $A$, 
\[P(|Z_{j}^{T}AZ_{j} - \E Z_{j}^{T}AZ_{j}|\ge t)\le 2\exp\left[-c\min\left\{\frac{t^{2}}{\sigma^{4}\|A\|_{\mathrm{F}}^{2}}, \frac{t}{\sigma^{2}\|A\|_{\mathrm{\mathrm{op}}}}\right\}\right]\]
for some universal constant $c$. Let $A = Q_{j}$ and conditioning on $Z_{[j]}$, then by Lemma \ref{lem:Qj}, we know that 
\[\|Q_{j}\|_{\mathrm{\mathrm{op}}}\le c_{1}^{2}\frac{K_{3}^{2}K_{1}}{K_{0}}, \quad \|Q_{j}\|_{\mathrm{F}}\le \sqrt{n}c_{1}^{2}\frac{K_{3}^{2}K_{1}}{K_{0}}\]
and hence
\begin{align}\label{eq:hanson_wright}
P\lb Z_{j}^{T}Q_{j}Z_{j} - \E (Z_{j}^{T}Q_{j}Z_{j}\big| Z_{[j]})\le  -t \bigg| Z_{[j]}\rb\le 2\exp\left[-c\min\left\{\frac{t^{2}}{\sigma^{4}\cdot nc_{1}^{4}K_{3}^{4}K_{1}^{2} / K_{0}^{2}}, \frac{t}{\sigma^{2}c_{1}^{2}K_{3}^{2}K_{1} / K_{0}}\right\}\right].
\end{align}
Note that 
\[\E (Z_{j}^{T}Q_{j}Z_{j}\big| Z_{[j]}) = \tr(\E [Z_{j}Z_{j}^{T} | Z_{[j]}]Q_{j}) = \E Z_{1j}^{2} \tr(Q_{j}) = \tau^{2}\tr(Q_{j}).\]
By Lemma \ref{lem:var_rij_copy}, we conclude that 
\begin{align}
&P\lb \frac{Z_{j}^{T}Q_{j}Z_{j}}{\tr(Q_{j})}\le \tau^{2} - \frac{t}{nK^{*}}\bigg| Z_{[j]}\rb\le P\lb \frac{Z_{j}^{T}Q_{j}Z_{j}}{\tr(Q_{j})}\le \tau^{2} - \frac{t}{\tr(Q_{j})}\bigg|Z_{[j]}\rb\nonumber\\
\le & 2\exp\left[-c\min\left\{\frac{t^{2}}{\sigma^{4}\cdot nc_{1}^{4}K_{3}^{4}K_{1}^{2} / K_{0}^{2}}, \frac{t}{2\sigma^{2}c_{1}^{2}K_{3}^{2}K_{1} / K_{0}}\right\}\right].\label{eq:hanson_wright_2}
\end{align}
Let $t = \frac{1}{2}\tau^{2}nK^{*}$ and take expectation of both sides over $Z_{[j]}$, we obtain that 
\begin{equation*}
P\lb \frac{Z_{j}^{T}Q_{j}Z_{j}}{\tr(Q_{j})} \le \frac{\tau^{2}}{2}\rb \le 2\exp\left[-cn\min\left\{\frac{K^{*2}\tau^{4}}{4\sigma^{4}c_{1}^{4}K_{3}^{4}K_{1}^{2} / K_{0}^{2}}, \frac{K^{*}\tau^{2}}{2\sigma^{2}c_{1}^{2}K_{3}^{2}K_{1} / K_{0}}\right\}\right]
\end{equation*}
and hence
\begin{equation}\label{eq:half_tau}
P\lb \min_{j\in J_{n}}\frac{Z_{j}^{T}Q_{j}Z_{j}}{\tr(Q_{j})} \le \frac{\tau^{2}}{2}\rb \le 2n\exp\left[-cn\min\left\{\frac{K^{*2}\tau^{4}}{4\sigma^{4}c_{1}^{4}K_{3}^{4}K_{1}^{2} / K_{0}^{2}}, \frac{K^{*}\tau^{2}}{2\sigma^{2}c_{1}^{2}K_{3}^{2}K_{1} / K_{0}}\right\}\right] = o(1).
\end{equation}
This entails that 
\[\min_{j\in J_{n}}\frac{Z_{j}^{T}Q_{j}Z_{j}}{\tr(Q_{j})} = \Omega_{p}(\polyLog).\]
Thus, assumption \textbf{A}4 is also satisfied with high probability. On the other hand, since $Z_{j}$ has i.i.d. mean-zero $\sigma^{2}$-sub-gaussian entries, for any deterministic unit vector $\alpha\in \R^{n}$, $\alpha^{T}Z_{j}$ is $\sigma^{2}$-sub-gaussian and mean-zero, and hence
\[P(|\alpha^{T}Z_{j}|\ge t)\le 2e^{-\frac{t^{2}}{2\sigma^{2}}}.\]
Let $\alpha_{j, i} = h_{j, 1, i} / \|h_{j, 1, i}\|_{2}$ and $\alpha_{j, 0} = h_{j, 0} / \|h_{j, 0}\|_{2}$. Since $h_{j, 1, i}$ and $h_{j, 0}$ are independent of $Z_{j}$, a union bound then gives
% \[P\lb \Delta_{C}\ge t\rb\le 2n(p + 1)e^{-\frac{t^{2}}{2\sigma^{2}}}\le 2n^{2}e^{-\frac{t^{2}}{2\sigma^{2}}}.\]
% Replace $t$ by $t + 2\sigma\sqrt{\log n}$, we obtain that
\[P\lb\Delta_{C}\ge t + 2\sigma\sqrt{\log n}\rb \le 2n^{2}e^{-\frac{t^{2} + 4\sigma^{2}\log n}{2\sigma^{2}}} = 2e^{-\frac{t^{2}}{2\sigma^{2}}}.\]
By Fubini's formula (\citeNP{durrett10}, Lemma 2.2.8.),
\begin{align}
\E \Delta_{C}^{8} &= \int_{0}^{\infty}8t^{7}P(\Delta_{C}\ge t)dt \le \int_{0}^{2\sigma\sqrt{\log n}}8t^{7}dt + \int_{2\sigma\sqrt{\log n}}^{\infty}8t^{7}P(\Delta_{C}\ge t)dt\nonumber\\
& = (2\sigma\sqrt{\log n})^{8} + \int_{0}^{\infty}8(t + 2\sigma\sqrt{\log n})^{7}P(\Delta_{C}\ge t + 2\sigma\sqrt{\log n})dt\nonumber\\
& \le (2\sigma\sqrt{\log n})^{8} + \int_{0}^{\infty}64(8t^{7} + 128\sigma^{7}(\log n)^{\frac{7}{2}})\cdot 2e^{-\frac{t^{2}}{2\sigma^{2}}}dt\nonumber\\
& = O(\sigma^{8}\cdot\polyLog) = O\lb\polyLog\rb.\label{eq:DeltaC_eight}
\end{align}
This, together with Markov inequality, guarantees that assumption \textbf{A}5 is also satisfied with high probability. 
\end{proof}

\begin{proof}[\textbf{Proof of Proposition \ref{prop:subgauss_noniid}}]
  It is left to prove that assumption \textbf{A}3 holds with high probability. The proof of assumption \textbf{A}4 and \textbf{A}5 is exactly the same as the proof of Proposition \ref{prop:subgauss_noniid}. By Proposition \ref{prop:rmt_latala05}, 
\[\lammax = O_{p}(1).\]
On the other hand, by Proposition \ref{prop:rmt_litvak05} \cite{litvak05}, 
\[P\lb \lambda_{\mathrm{\min}}\lb\frac{Z^{T}Z}{n}\rb < c_{1}\rb\le e^{-c_{2}n}.\]
and thus
\[\lammin = \Omega_{p}(1).\]
\end{proof}

\begin{proof}[\textbf{Proof of Proposition \ref{prop:subgauss_intercept}}]
Since $J_{n}$ excludes the intercept term, the proof of assumption \textbf{A}4 and \textbf{A}5 is still the same as Proposition \ref{prop:subgauss_noniid}. It is left to prove assumption \textbf{A}3. Let $R_{1}, \ldots, R_{n}$ be i.i.d. Rademacher random variables, i.e. $P(R_{i} = 1) = P(R_{i} = -1) = \frac{1}{2}$, and 
\[Z^{*} = \diag(B_{1}, \ldots, B_{n})Z.\]
Then $(Z^{*})^{T}Z^{*} = Z^{T}Z$. It is left to show that the assumption \textbf{A}3 holds for $Z^{*}$ with high probability. Note that 
\[(Z^{*}_{i})^{T} = (B_{i}, B_{i}\td{x}_{i}^{T}).\]
For any $r \in \{1, -1\}$ and borel sets $B_{1}, \ldots, B_{p}\subset \R$,
\begin{align*}
& \,\, P(B_{i} = r, B_{i}\td{Z}_{i1}\in B_{1}, \ldots, B_{i}\td{Z}_{i(p-1)}\in B_{p - 1}) \\
& = P(B_{i} = r, \td{Z}_{i1}\in rB_{1}, \ldots, \td{Z}_{i(p- 1)}\in rB_{p - 1})\\
& = P(B_{i} = r)P(\td{Z}_{i1}\in rB_{1})\ldots P(\td{Z}_{i(p- 1)}\in rB_{p - 1})\\
& = P(B_{i} = r)P(\td{Z}_{i1}\in B_{1})\ldots P(\td{Z}_{i(p- 1)}\in B_{p - 1})\\
& = P(B_{i} = r)P(B_{i}\td{Z}_{i1}\in B_{1})\ldots P(B_{i}\td{Z}_{i(p- 1)}\in B_{p - 1})
\end{align*}
where the last two lines uses the symmetry of $\td{Z}_{ij}$. Then we conclude that $Z^{*}_{i}$ has independent entries. Since the rows of $Z^{*}$ are independent, $Z^{*}$ has independent entries. Since $B_{i}$ are symmetric and sub-gaussian with unit variance and $B_{i}\td{Z}_{ij}\stackrel{d}{=}\td{Z}_{ij}$, which is also symmetric and sub-gaussian with variance bounded from below, $Z^{*}$ satisfies the conditions of Propsition \ref{prop:subgauss_noniid} and hence the assumption \textbf{A}3 is satisfied with high probability.
\end{proof}

\begin{proof}[\textbf{Proof of Proposition \ref{prop:depend_gauss_row} (with Proposition \ref{prop:depend_gauss} being a special case)}]
Let $Z_{*} = \Lambda^{-\frac{1}{2}}Z\Sigma^{-\frac{1}{2}}$, then $Z_{*}$ has i.i.d. standard gaussian entries. By Proposition \ref{prop:subgauss_intercept}, $Z_{*}$ satisfies assumption \textbf{A}3 with high probability. Thus, 
\[\lammax = \lambda_{\mathrm{\max}}\lb\frac{\Sigma^{\frac{1}{2}}Z_{*}^{T}\Lambda Z_{*}\Sigma^{\frac{1}{2}}}{n}\rb\le \lambda_{\mathrm{\max}}(\Sigma)\cdot \lambda_{\mathrm{\max}}(\Lambda)\cdot \lambda_{\mathrm{\max}}\lb\frac{Z_{*}^{T}Z_{*}}{n}\rb = O_{p}(\polyLog),\]
and 
\[\lammin = \lambda_{\mathrm{\min}}\lb\frac{\Sigma^{\frac{1}{2}}Z_{*}^{T}\Lambda Z_{*}\Sigma^{\frac{1}{2}}}{n}\rb\ge \lambda_{\mathrm{\min}}(\Sigma)\cdot \lambda_{\mathrm{\min}}(\Lambda)\cdot \lambda_{\mathrm{\min}}\lb\frac{Z_{*}^{T}Z_{*}}{n}\rb = \Omega_{p}\lb\frac{1}{\polyLog}\rb.\]
% \[\lammin = \min_{i}\lambda_{\mathrm{\min}}\lb\frac{\Sigma^{\frac{1}{2}}Z_{*(i)}^{T}\Lambda_{(i), (i)} Z_{*(i)}\Sigma^{\frac{1}{2}}}{n}\rb\ge \lambda_{\mathrm{\min}}(\Sigma)\cdot \lambda_{\mathrm{\min}}(\Lambda)\cdot \min_{i}\lambda_{\mathrm{\min}}\lb\frac{Z_{*(i)}^{T}Z_{*(i)}}{n}\rb = \Omega\lb\frac{1}{\polyLog}\rb,\]
% where we use the interlacing property that $\lambda_{\mathrm{\min}}(\Sigma_{(i), (i)})\ge \lambda_{\mathrm{\min}}(\Sigma)$.
As for assumption \textbf{A}4, % similar to the proof of Proposition \ref{prop:depend_gauss}, 
the first step is to calculate $\E (Z_{j}^{T}Q_{j}Z_{j} | Z_{[j]})$. Let $\td{Z} = \Lambda^{-\frac{1}{2}}Z$, then $\vect(\td{Z})\sim N(0, I\otimes \Sigma)$. As a consequence, 
% It is shown in the proof of Proposition \ref{prop:depend_gauss} that 
\[\td{Z}_{j} | \td{Z}_{[j]}\sim N(\td{\mu}_{j}, \sigma_{j}^{2}I)\]
where 
\[\td{\mu}_{j} = \td{Z}_{[j]}\Sigma_{[j], [j]}^{-1}\Sigma_{[j], j} = \Lambda^{-\frac{1}{2}}Z_{[j]}\Sigma_{[j], [j]}^{-1}\Sigma_{[j], j}.\]
Thus, 
\[Z_{j} | Z_{[j]}\sim N(\mu_{j}, \sigma_{j}^{2}\Lambda)\]
where $\mu_{j} = Z_{[j]}\Sigma_{[j], [j]}^{-1}\Sigma_{[j], j}$. It is easy to see that 
\begin{equation}\label{eq:sigmaj}
\lammin \le \min_{j}\sigma_{j}^{2}\le \max_{j}\sigma_{j}^{2}\le \lammax.
\end{equation}
It has been shown that $Q_{j}\mu_{j} = 0$ and hence 
\[Z_{j}^{T}Q_{j}Z_{j} = (Z_{j} - \mu_{j})^{T}Q_{j}(Z_{j} - \mu_{j}).\]
Let $\mathscr{Z}_{j} = \Lambda^{-\frac{1}{2}}(Z_{j} - \mu_{j})$ and $\td{Q}_{j} = \Lambda^{\frac{1}{2}}Q_{j}\Lambda^{\frac{1}{2}}$, then $\mathscr{Z}_{j}\sim N(0, \sigma_{j}^{2}I)$ and 
\[Z_{j}^{T}Q_{j}Z_{j} = \mathscr{Z}_{j}^{T}\td{Q}_{j}\mathscr{Z}_{j}.\]
By Lemma \ref{lem:Qj}, 
\[\|\td{Q}_{j}\|_{\mathrm{\mathrm{op}}}\le \|\Lambda\|_{\mathrm{\mathrm{op}}}\cdot \|Q_{j}\|_{\mathrm{\mathrm{op}}} \le \lambda_{\mathrm{\max}}(\Lambda)\cdot c_{1}^{2}\frac{K_{3}^{2}K_{1}}{K_{0}},\]
and hence
\[\|\td{Q}_{j}\|_{\mathrm{F}}\le \sqrt{n}\lambda_{\mathrm{\max}}(\Lambda)\cdot c_{1}^{2}\frac{K_{3}^{2}K_{1}}{K_{0}}.\]
By Hanson-Wright inequality (\citeNP{hanson71, rudelson13}; see Proposition \ref{prop:hanson_wright}), we obtain a similar inequality to \eqref{eq:hanson_wright} as follows: 
\begin{align*}
&P\lb|Z_{j}^{T}Q_{j}Z_{j} - \E (Z_{j}^{T}Q_{j}Z_{j}\big| Z_{[j]})|\ge t \bigg| Z_{[j]}\rb\\
\le & 2\exp\left[-c\min\left\{\frac{t^{2}}{\sigma_{j}^{4}\cdot n\lambda_{\mathrm{\max}}(\Lambda)^{2}c_{1}^{4}K_{3}^{4}K_{1}^{2} / K_{0}^{2}}, \frac{t}{\sigma_{j}^{2}\lambda_{\mathrm{\max}}(\Lambda)c_{1}^{2}K_{3}^{2}K_{1} / K_{0}}\right\}\right].
\end{align*}
On the other hand,
\[\E (Z_{j}^{T}Q_{j}Z_{j} | Z_{[j]}) = \E (\mathscr{Z}_{j}^{T}\td{Q}_{j}\mathscr{Z}_{j} | Z_{[j]}) = \sigma_{j}^{2}\tr(\td{Q}_{j}).\]
By definition,
\[\tr(\td{Q}_{j}) = \tr(\Lambda^{\frac{1}{2}}Q_{j}\Lambda^{\frac{1}{2}}) = \tr(\Sigma Q_{j}) = \tr(Q_{j}^{\frac{1}{2}}\Lambda Q_{j}^{\frac{1}{2}})\ge \lambda_{\mathrm{\min}}(\Lambda)\tr(Q_{j}).\]
By Lemma \ref{lem:var_rij_copy},
\[\tr(\td{Q}_{j})\ge \lambda_{\mathrm{\min}}(\Lambda)\cdot nK^{*}.\]
Similar to \eqref{eq:hanson_wright_2}, we obtain that 
\begin{align*}
&P\lb \frac{Z_{j}^{T}Q_{j}Z_{j}}{\tr(Q_{j})}\ge \sigma_{j}^{2} - \frac{t}{nK^{*}}\bigg| Z_{[j]}\rb \\
\le & 2\exp\left[-c\min\left\{\frac{t^{2}}{\sigma_{j}^{4}\cdot n\lambda_{\mathrm{\max}}(\Lambda)^{2}c_{1}^{4}K_{3}^{4}K_{1}^{2} / K_{0}^{2}}, \frac{t}{\sigma_{j}^{2}\lambda_{\mathrm{\max}}(\Lambda)c_{1}^{2}K_{3}^{2}K_{1} / K_{0}}\right\}\right].
\end{align*}
Let $t = \frac{1}{2}\sigma_{j}^{2}nK^{*}$, we have
\[P\lb \frac{Z_{j}^{T}Q_{j}Z_{j}}{\tr(Q_{j})}\ge \frac{\sigma_{j}^{2}}{2}\rb\le 2\exp\left[-cn\min\left\{\frac{K^{*2}}{4\lambda_{\mathrm{\max}}(\Lambda)^{2}c_{1}^{4}K_{3}^{4}K_{1}^{2} / K_{0}^{2}}, \frac{K^{*}}{2\lambda_{\mathrm{\max}}(\Lambda)c_{1}^{2}K_{3}^{2}K_{1} / K_{0}}\right\}\right] = o\lb\frac{1}{n}\rb\]
and a union bound together with \eqref{eq:sigmaj} yields that 
\[\min_{j\in J_{n}}\frac{Z_{j}^{T}Q_{j}Z_{j}}{\tr(Q_{j})} = \Omega_{p}\lb\min_{j}\sigma_{j}^{2}\cdot \frac{1}{\polyLog}\rb = \Omega_{p}\lb\frac{1}{\polyLog}\rb.\]
As for assumption \textbf{A}5, let
\[\alpha_{j, 0} = \frac{\Lambda^{\frac{1}{2}}h_{j, 0}}{\|h_{j, 0}\|_{2}}, \quad \alpha_{j, i} = \frac{\Lambda^{\frac{1}{2}}h_{j, 1, i}}{\|h_{j, 1, i}\|_{2}}\]
then for $i = 0, 1, \ldots, p$, 
\[\|\alpha_{j, i}\|_{2}\le \sqrt{\lambda_{\mathrm{\max}}(\Lambda)}.\]
Note that 
\[\frac{h_{j, 0}^{T}Z_{j}}{\|h_{j, 0}\|_{2}} = \alpha_{j, 0}^{T}Z_{j}, \quad \frac{h_{j, 1, i}^{T}Z_{j}}{\|h_{j, 1, i}\|_{2}} = \alpha_{j, i}^{T}Z_{j}\]
using the same argument as in \eqref{eq:DeltaC_eight}, we obtain that 
\[\E \Delta_{C}^{8} = O\lb \lambda_{\mathrm{\max}}(\Lambda)^{4}\cdot \max_{j}\sigma_{j}^{8}\cdot \polyLog\rb = O\lb\polyLog\rb,\]
and by Markov inequality and \eqref{eq:sigmaj},
\[\E (\Delta_{C}^{8} | Z) = O_{p}\lb \E \Delta_{C}^{8} \rb = O_{p}(\polyLog).\]
\end{proof}

\begin{proof}[\textbf{Proof of Proposition \ref{prop:depend_gauss_intercept}}]
The proof that assumptions \textbf{A}4 and \textbf{A}5 hold with high probability is exactly the same as the proof of Proposition \ref{prop:depend_gauss_row}. It is left to prove assumption \textbf{A}3*; see Corollary \ref{cor:transform}. Let $c = (\min_{i}|(\Lambda^{-\frac{1}{2}}\textbf{1})_{i}|)^{-1}$ and $\mathbf{Z} = (c\textbf{1}\,\, \td{Z})$. Recall the the definition of $\td{\lambda}_{+}$ and $\td{\lambda}_{-}$, we have
\[\td{\lambda}_{+} = \lambda_{\max}(\Sigma_{\{1\}}), \quad \td{\lambda}_{-} = \lambda_{\min}(\Sigma_{\{1\}}),\]
where
\[\Sigma_{\{1\}} = \frac{1}{n}\td{Z}^{T}\lb I - \frac{\textbf{1}\textbf{1}^{T}}{n}\rb \td{Z}.\]
Rewrite $\Sigma_{\{1\}}$ as
\[\Sigma_{\{1\}} = \frac{1}{n}\lb \lb I - \frac{\textbf{1}\textbf{1}^{T}}{n}\rb\td{Z}\rb^{T}\lb \lb I - \frac{\textbf{1}\textbf{1}^{T}}{n}\rb\td{Z}\rb.\]
It is obvious that 
\[\spanvec\lb \lb I - \frac{\textbf{1}\textbf{1}^{T}}{n}\rb\td{Z}\rb \subset \spanvec(\mathbf{Z}).\]
As a consequence
\[\td{\lambda}_{+}\le \lambda_{\max}\lb\frac{\mathbf{Z}^{T}\mathbf{Z}}{n}\rb, \quad \td{\lambda}_{-}\ge \lambda_{\min}\lb\frac{\mathbf{Z}^{T}\mathbf{Z}}{n}\rb.\]
It remains to prove that 
\[\lambda_{\max}\lb\frac{\mathbf{Z}^{T}\mathbf{Z}}{n}\rb = O_{p}\lb\polyLog\rb, \quad \lambda_{\min}\lb\frac{\mathbf{Z}^{T}\mathbf{Z}}{n}\rb = \Omega_{p}\lb\frac{1}{\polyLog}\rb.\]
To prove this, we let
\[Z_{*} = \Lambda^{-\frac{1}{2}}\mathbf{Z}\lb\begin{array}{cc}1 & 0 \\ 0 & \Sigma^{-\frac{1}{2}}\end{array}\rb\triangleq (\nu\,\, \td{Z}_{*}),\]
where $\nu = c\Lambda^{-\frac{1}{2}}\textbf{1}$ and $\td{Z}_{*} = \Lambda^{-\frac{1}{2}}\td{Z}\Sigma^{-\frac{1}{2}}$. Then 
\[\lambda_{\max}\lb\frac{\mathbf{Z}^{T}\mathbf{Z}}{n}\rb = \lambda_{\mathrm{\max}}\lb\frac{\Sigma^{\frac{1}{2}}Z_{*}^{T}\Lambda Z_{*}\Sigma^{\frac{1}{2}}}{n}\rb\le \lambda_{\mathrm{\max}}(\Sigma)\cdot \lambda_{\mathrm{\max}}(\Lambda)\cdot \lambda_{\mathrm{\max}}\lb\frac{Z_{*}^{T}Z_{*}}{n}\rb,\]
and 
\[\lambda_{\min}\lb\frac{\mathbf{Z}^{T}\mathbf{Z}}{n}\rb = \lambda_{\mathrm{\min}}\lb\frac{\Sigma^{\frac{1}{2}}Z_{*}^{T}\Lambda Z_{*}\Sigma^{\frac{1}{2}}}{n}\rb\ge \lambda_{\mathrm{\min}}(\Sigma)\cdot \lambda_{\mathrm{\min}}(\Lambda)\cdot \lambda_{\mathrm{\min}}\lb\frac{Z_{*}^{T}Z_{*}}{n}\rb.\]
It is left to show that 
\[\lambda_{\mathrm{\max}}\lb\frac{Z_{*}^{T}Z_{*}}{n}\rb = O_{p}(\polyLog), \quad \lambda_{\mathrm{\min}}\lb\frac{Z_{*}^{T}Z_{*}}{n}\rb = \Omega_{p}\lb\frac{1}{\polyLog}\rb.\]
By definition, $\min_{i}|\nu_{i}| = 1$ and $\max_{i}|\nu_{i}| = O\lb\polyLog\rb$, then 
\[\lambda_{\mathrm{\max}}\lb\frac{Z_{*}^{T}Z_{*}}{n}\rb = \lambda_{\mathrm{\max}}\lb\frac{\td{Z}_{*}^{T}\td{Z}_{*}}{n} + \frac{\nu \nu^{T}}{n}\rb\le \lambda_{\mathrm{\max}}\lb\frac{\td{Z}_{*}^{T}\td{Z}_{*}}{n}\rb + \frac{\|\nu\|_{2}^{2}}{n}.\]
Since $\td{Z}_{*}$ has i.i.d. standard gaussian entries, by Proposition \ref{prop:rmt_bai93}, 
\[\lambda_{\mathrm{\max}}\lb\frac{\td{Z}_{*}^{T}\td{Z}_{*}}{n}\rb = O_{p}(1).\]
Moreover, $\|\nu\|_{2}^{2} \le n \max_{i}|\nu_{i}|^{2} = O(n\cdot \polyLog)$ and thus, 
\[\lambda_{\mathrm{\max}}\lb\frac{Z_{*}^{T}Z_{*}}{n}\rb = O_{p}(\polyLog).\]
On the other hand, similar to Proposition \ref{prop:subgauss_intercept}, 
\[\mathbf{Z}_{*} = \diag(B_{1}, \ldots, B_{n})Z_{*}\]
where $B_{1}, \ldots, B_{n}$ are i.i.d. Rademacher random variables. The same argument in the proof of Proposition \ref{prop:subgauss_intercept} implies that $\mathbf{Z}_{*}$ has independent entries with sub-gaussian norm bounded by $\|\nu\|_{\infty}^{2}\vee 1$ and variance lower bounded by $1$. By Proposition \ref{prop:rmt_litvak05}, $Z_{*}$ satisfies assumption \textbf{A}3 with high probability. Therefore, \textbf{A}3* holds with high probability.
\end{proof}

\begin{proof}[\textbf{Proof of Proposition \ref{prop:ellip}}]
Let $\Lambda = (\lambda_{1}, \ldots, \lambda_{n})$ and $\mathcal{Z}$ be the matrix with entries $\mathcal{Z}_{ij}$, then by Proposition \ref{prop:subgauss} or Proposition \ref{prop:subgauss_noniid}, $\mathcal{Z}_{ij}$ satisfies assumption \textbf{A}3 with high probability. Notice that
\[\lammax = \lambda_{\mathrm{\max}}\lb\frac{\mathcal{Z}^{T}\Lambda^{2} \mathcal{Z}}{n}\rb\le \lambda_{\mathrm{\max}}(\Lambda)^{2}\cdot \lambda_{\mathrm{\max}}\lb\frac{\mathcal{Z}^{T}\mathcal{Z}}{n}\rb = O_{p}(\polyLog),\]
and 
\[\lammin = \lambda_{\mathrm{\min}}\lb\frac{\mathcal{Z}^{T}\Lambda^{2} \mathcal{Z}}{n}\rb\ge \lambda_{\mathrm{\min}}(\Lambda)^{2}\cdot \lambda_{\mathrm{\min}}\lb\frac{\mathcal{Z}^{T}\mathcal{Z}}{n}\rb = \Omega_{p}\lb\frac{1}{\polyLog}\rb.\]
Thus $Z$ satisfies assumption \textbf{A}3 with high probability. 

~\\
\noindent Conditioning on any realization of $\Lambda$, the law of $\mathcal{Z}_{ij}$ does not change due to the independence between $\Lambda$ and $\mathcal{Z}$. Repeating the arguments in the proof of Proposition \ref{prop:subgauss} and Proposition \ref{prop:subgauss_noniid}, we can show that 
\begin{equation}
  \label{eq:A4A5}
  \frac{\mathcal{Z}_{j}^{T}\td{Q}_{j}\mathcal{Z}_{j}}{\tr(\td{Q}_{j})} = \Omega_{p}\lb\frac{1}{\polyLog}\rb, \quad\mbox{and}\quad  \E \max_{i = 0, \ldots, n; j = 1, \ldots, p}|\td{\alpha}_{j, i}^{T}\mathcal{Z}_{j}|^{8} = O_{p}(\polyLog),
\end{equation}
where 
\[\td{Q}_{j} = \Lambda Q_{j}\Lambda, \quad \td{\alpha}_{j, 0} = \frac{\Lambda h_{j, 0}}{\|\Lambda h_{j, 0}\|_{2}}, \quad \td{\alpha}_{j, 1, i} = \frac{\Lambda h_{j, 1, i}}{\|\Lambda h_{j, 1, i}\|_{2}}.\]
Then 
\begin{equation}\label{eq:ellip_A4}
\frac{Z_{j}^{T}Q_{j}Z_{j}}{\tr(Q_{j})} = \frac{\mathcal{Z}_{j}^{T}\td{Q}_{j}\mathcal{Z}_{j}}{\tr(\td{Q}_{j})}\cdot \frac{\tr(\Lambda Q_{j}\Lambda)}{\tr(Q_{j})} \ge a^{2}\cdot\frac{\mathcal{Z}_{j}^{T}\td{Q}_{j}\mathcal{Z}_{j}}{\tr(\td{Q}_{j})} = \Omega_{p}\lb\frac{1}{\polyLog}\rb,
\end{equation}
and 
\begin{align}
\E \Delta_{C}^{8} &= \E \left[\max_{i = 0, \ldots, n; j = 1, \ldots, p}|\td{\alpha}_{j, i}^{T}\mathcal{Z}_{j}|^{8}\cdot \max\left\{\max_{j}\frac{\|\Lambda h_{j, 0}\|_{2}}{\|h_{j, 0}\|_{2}}, \max_{i, j}\frac{\|\Lambda h_{j, 1, i}\|_{2}}{\|h_{j, 1, i}\|_{2}}\right\}^{8}\right]\label{eq:DeltaCLambda}\\
& \le b^{8}\E \left[\max_{i = 0, \ldots, n; j = 1, \ldots, p}|\td{\alpha}_{j, i}^{T}\mathcal{Z}_{j}|^{8}\right]\nonumber\\
& = O_{p}(\polyLog).\nonumber
\end{align}
By Markov inequality, the assumption \textbf{A}5 is satisfied with high probability.
\end{proof}

\begin{proof}[\textbf{Proof of Proposition \ref{prop:ellip_relax}}]
The concentration inequality of $\zeta_{i}$ plus a union bound imply that 
\[P\lb \max_{i}\zeta_{i} > (\log n)^{\frac{2}{\alpha}}\rb\le nc_{1}e^{-c_{2}(\log n)^{2}} = o(1).\]
Thus, with high probability, 
\[\lambda_{\mathrm{\max}} = \lambda_{\mathrm{\max}}\lb \frac{\mathcal{Z}^{T}\Lambda^{2}\mathcal{Z}}{n}\rb\le (\log n)^{\frac{4}{\alpha}}\cdot \lambda_{\mathrm{\max}}\lb\frac{\mathcal{Z}^{T}\mathcal{Z}}{n}\rb = O_{p}(\polyLog).\]
Let $n' = \lfloor(1 - \delta)n\rfloor$ for some $\delta \in (0, 1 - \kappa)$. Then for any subset $I$ of $\{1, \ldots, n\}$ with size $n'$, by Proposition \ref{prop:rmt_rudelson09} (Proposition \ref{prop:rmt_litvak05}), under the conditions of Proposition \ref{prop:subgauss} (Proposition \ref{prop:subgauss_noniid}), there exists constants $c_{3}$ and $c_{4}$, which only depend on $\kappa$, such that 
\[P\lb \lambda_{\mathrm{\min}}\lb\frac{\mathcal{Z}_{I}^{T}\mathcal{Z}_{I}}{n}\rb < c_{3}\rb\le e^{-c_{4}n}\]
where $\mathcal{Z}_{I}$ represents the sub-matrix of $\mathcal{Z}$ formed by $\{\mathcal{Z}_{i}: i \in I\}$, where $\mathcal{Z}_{i}$ is the $i$-th row of $\mathcal{Z}$. Then by a union bound, 
\[P\lb \min_{|I| = n'}\lambda_{\mathrm{\min}}\lb\frac{\mathcal{Z}_{I}^{T}\mathcal{Z}_{I}}{n}\rb < c_{3}\rb\le \com{n}{n'}e^{-c_{4}n}.\]
By Stirling's formula, there exists a constant $c_{5} > 0$ such that 
\[\com{n}{n'} = \frac{n!}{n'!(n - n')!}\le c_{5}\exp\left\{(-\td{\delta}\log \td{\delta} - (1 - \td{\delta})\log (1 - \td{\delta}))n\right\}\]
where $\td{\delta} = n' / n$. For sufficiently small $\delta$ and sufficiently large $n$, 
\[-\td{\delta}\log \td{\delta} - (1 - \td{\delta})\log (1 - \td{\delta}) < c_{4}\]
and hence
\begin{equation}\label{eq:mineig}
P\lb \min_{|I| = n'}\lambda_{\mathrm{\min}}\lb\frac{\mathcal{Z}_{I}^{T}\mathcal{Z}_{I}}{n}\rb < c_{3}\rb < c_{5}e^{-c_{6}n}
\end{equation}
for some $c_{6} > 0$. By Borel-Cantelli Lemma, 
\[\liminf_{n\rightarrow\infty}\min_{|I| = \lfloor (1 - \delta)n \rfloor}\lambda_{\mathrm{\min}}\lb\frac{\mathcal{Z}_{I}^{T}\mathcal{Z}_{I}}{n}\rb \ge  c_{3}\quad a.s..\]
On the other hand, since $F^{-1}$ is continuous at $\delta$, then 
\[\zeta_{(\lfloor (1 - \delta)n\rfloor)}\stackrel{a.s.}{\rightarrow} F^{-1}(\delta) > 0.\]
where $\zeta_{(k)}$ is the $k$-th largest of $\{\zeta_{i}: i = 1, \ldots, n\}$. Let $I^{*}$ be the set of indices corresponding to the largest $\lfloor (1 - \delta) n\rfloor$ $\zeta_{i}'$s. Then with probability $1$, 
\begin{align*}
\liminf_{n\rightarrow \infty}\lambda_{\mathrm{\min}}\lb\frac{Z^{T}Z}{n}\rb &= \liminf_{n\rightarrow \infty}\lambda_{\mathrm{\min}}\lb\frac{\mathcal{Z}^{T}\Lambda^{2} \mathcal{Z}}{n}\rb \ge \liminf_{n\rightarrow \infty}\zeta_{(\lfloor (1 - \delta)n\rfloor)}\cdot \liminf_{n\rightarrow \infty}\lambda_{\mathrm{\min}}\lb\frac{\mathcal{Z}_{I^{*}}^{T}\Lambda_{I^{*}}^{2}\mathcal{Z}_{I^{*}}}{n}\rb\\
&\ge \liminf_{n\rightarrow \infty}\zeta_{(\lfloor (1 - \delta)n\rfloor)}\cdot \liminf_{n\rightarrow \infty}\min_{|I| = \lfloor (1 - \delta)n \rfloor}\lambda_{\mathrm{\min}}\lb\frac{\mathcal{Z}_{I}^{T}\mathcal{Z}_{I}}{n}\rb\\
& \ge c_{3}F^{-1}(\delta)^{2} > 0.
\end{align*}
To prove assumption \textbf{A}4, similar to \eqref{eq:ellip_A4} in the proof of Proposition \ref{prop:ellip}, it is left to show that 
\[\min_{j}\frac{\tr(\Lambda Q_{j}\Lambda)}{\tr(Q_{j})} = \Omega_{p}\lb\frac{1}{\polyLog}\rb.\]
Furthermore, by Lemma \ref{lem:var_rij_copy}, it remains to prove that
\[\min_{j}\tr(\Lambda Q_{j}\Lambda) = \Omega_{p}\lb\frac{n}{\polyLog}\rb.\]
Recalling the equation \eqref{eq:Qjii} in the proof of Lemma \ref{lem:var_rij}, we have 
\begin{equation}\label{eq:Qjii2}
e_{i}^{T}Q_{j}e_{i} \ge \frac{K_{0}}{K_{1}}\cdot \frac{1}{1 + e_{i}^{T}Z_{[j]}^{T}(Z_{(i), [j]}^{T}Z_{(i), [j]})^{-1}Z_{[j]}e_{i}}.
\end{equation}
By Proposition \ref{prop:rmt_rudelson10}, 
\[P\lb \sqrt{\lambda_{\mathrm{\max}}\lb\frac{\mathcal{Z}_{j}^{T}\mathcal{Z}_{j}}{n}\rb} > 3C_{1}\rb\le 2e^{-C_{2}n}.\]
On the other hand, apply \eqref{eq:mineig} to $\mathcal{Z}_{(i), [j]}$, we have
\[P\lb \min_{|I| = \lfloor (1 - \delta) n\rfloor}\lambda_{\mathrm{\min}}\lb\frac{(\mathcal{Z}_{(i), [j]})_{I}^{T}(\mathcal{Z}_{(i), [j]})_{I}}{n}\rb < c_{3}\rb < c_{5}e^{-c_{6}n}.\]
A union bound indicates that with probability $(c_{5}np + 2p)e^{-\min\{C_{2}, c_{6}\}n} = o(1)$, 
\[\max_{j}\lambda_{\mathrm{\max}}\lb\frac{\mathcal{Z}_{[j]}^{T}\mathcal{Z}_{[j]}}{n}\rb\le 9C_{1}^{2},\quad \min_{i, j}\min_{|I| = \lfloor (1 - \delta) n\rfloor}\lambda_{\mathrm{\min}}\lb\frac{(\mathcal{Z}_{(i), [j]})_{I}^{T}(\mathcal{Z}_{(i), [j]})_{I}}{n}\rb \ge c_{3}.\]
This implies that for any $j$,
\[\lambda_{\mathrm{\max}}\lb\frac{Z_{[j]}^{T}Z_{[j]}}{n}\rb = \lambda_{\mathrm{\max}}\lb\frac{\mathcal{Z}_{[j]}^{T}\Lambda^{2} \mathcal{Z}_{[j]}}{n}\rb \le  \zeta_{(1)}^{2}\cdot 9C_{1}^{2}\]
and for any $i$ and $j$,
\begin{align*}
& \lambda_{\mathrm{\min}}\lb\frac{Z_{(i), [j]}^{T}Z_{(i), [j]}}{n}\rb = \lambda_{\mathrm{\min}}\lb\frac{\mathcal{Z}_{(i), [j]}^{T}\zeta_{(i)}^{2} \mathcal{Z}_{(i), [j]}}{n}\rb\\ 
  \ge & \min\{\zeta_{(\lfloor (1 - \delta)n\rfloor)}, \zeta_{(\lfloor (1 - \delta)n\rfloor)} + 1\}^{2}\cdot \min_{|I| = \lfloor (1 - \delta)n \rfloor}\lambda_{\mathrm{\min}}\lb\frac{(\mathcal{Z}_{(i), [j]})_{I}^{T}\zeta_{(i)}^{2}(\mathcal{Z}_{(i), [j]})_{I}}{n}\rb\\
\ge & c_{3}\min\{\zeta_{(\lfloor (1 - \delta)n\rfloor)}, \zeta_{(\lfloor (1 - \delta)n\rfloor)} + 1\}^{2} > 0.
\end{align*}
Moreover, as discussed above, 
\[\zeta_{(1)} \le (\log n)^{\frac{2}{\alpha}}, \min\{\zeta_{(\lfloor (1 - \delta)n\rfloor)}, \zeta_{(\lfloor (1 - \delta)n\rfloor)} + 1\}\rightarrow F^{-1}(\delta)\]
almost surely. Thus, it follows from \eqref{eq:Qjii2} that with high probability, 
\begin{align*}
e_{i}^{T}Q_{j}e_{i} &\ge \frac{K_{0}}{K_{1}}\cdot \frac{1}{1 + e_{i}^{T}Z_{[j]}^{T}(Z_{(i), [j]}^{T}Z_{(i), [j]})^{-1}Z_{[j]}e_{i}}\\
& \ge \frac{K_{0}}{K_{1}}\cdot \frac{1}{1 + e_{i}^{T}\frac{Z_{[j]}^{T}Z_{[j]}}{n}e_{i}\cdot c_{3}(F^{-1}(\delta))^{2}}\\
& \ge \frac{K_{0}}{K_{1}}\cdot \frac{1}{1 + (\log n)^{\frac{4}{\alpha}}\cdot 9C_{1}^{2}\cdot c_{3}(F^{-1}(\delta))^{2}}.
\end{align*}
The above bound holds for all diagonal elements of $Q_{j}$ uniformly with high probability. Therefore,
\[\tr(\Lambda Q_{j}\Lambda)\ge \zeta_{(\lfloor (1 - \delta)n\rfloor)}^{2} \cdot \lfloor (1 - \delta)n\rfloor\cdot \frac{K_{0}}{K_{1}}\cdot \frac{1}{1 + (\log n)^{\frac{4}{\alpha}}\cdot 9C_{1}^{2}\cdot c_{3}(F^{-1}(\delta))^{2}} = \Omega_{p}\lb\frac{n}{\polyLog}\rb.\]
As a result, the assumption \textbf{A}4 is satisfied with high probability. Finally, by \eqref{eq:DeltaCLambda}, we obtain that 
\[\E \Delta_{C}^{8}\le \E \left[\max_{i = 0, \ldots, n; j = 1, \ldots, p}|\td{\alpha}_{j, i}^{T}\mathcal{Z}_{j}|^{8}\cdot \|\Lambda\|_{\mathrm{\mathrm{op}}}^{8}\right].\]
By Cauchy's inequality, 
\[\E \Delta_{C}^{8}\le \sqrt{\E \max_{i = 0, \ldots, n; j = 1, \ldots, p}|\td{\alpha}_{j, i}^{T}\mathcal{Z}_{j}|^{16}}\cdot \sqrt{\E \max_{i}\zeta_{i}^{16}}.\]
Similar to \eqref{eq:DeltaC_eight}, we conclude that 
\[\E \Delta_{C}^{8} = O\lb\polyLog\rb\]
and by Markov inequality, the assumption \textbf{A}5 is satisfied with high probability.
\end{proof}

\subsection{More Results of Least-Squares (Section \ref{sec:lse})}
\subsubsection{The Relation Between $S_{j}(X)$ and $\Delta_{C}$}\label{app:relationSjDelta}
In Section \ref{sec:lse}, we give a sufficient and almost necessary condition for the coordinate-wise asymptotic normality of the least-square estimator $\lse$; see Theorem \ref{thm:OLS_casn_kd}. In this subsubsection, we show that $\Delta_{C}$ is a generalization of $\max_{j\in J_{n}}S_{j}(X)$ for general M-estimators.

Consider the matrix $(X^{T}DX)^{-1}X^{T}$, where $D$ is obtain by using general loss functions, then by block matrix inversion formula (see Proposition \ref{prop:block_inv}), 
\begin{align*}
e_{1}^{T}(X^{T}DX)^{-1}X^{T} &= e_{1}^{T}\lb\begin{array}{cc}X_{1}^{T}DX_{1} & X_{1}^{T}DX_{[1]}\\ X_{[1]}^{T}DX_{1} & X_{[1]}^{T}DX_{[1]}\end{array}\rb^{-1}\com{X_{1}^{T}}{X_{[1]}^{T}}\\
& = \frac{X_{1}^{T}(I - DX_{[1]}(X_{[1]}^{T}DX_{[1]})^{-1}X_{[1]}^{T})}{X_{1}^{T}(D - DX_{[1]}(X_{[1]}^{T}DX_{[1]})^{-1}X_{[1]}^{T}D)X_{1}}\\
& \approx \frac{X_{1}^{T}(I - D_{[1]}X_{[1]}(X_{[1]}^{T}D_{[1]}X_{[1]})^{-1}X_{[1]}^{T})}{X_{1}^{T}(D - DX_{[1]}(X_{[1]}^{T}DX_{[1]})^{-1}X_{[1]}^{T}D)X_{1}}
\end{align*}
where we use the approximation $D \approx D_{[1]}$. The same result holds for all $j\in J_{n}$, then 
\[\frac{\|e_{j}^{T}(X^{T}DX)^{-1}X^{T}\|_{\infty}}{\|e_{j}^{T}(X^{T}DX)^{-1}X^{T}\|_{2}}\approx \frac{\|X_{1}^{T}(I - D_{[1]}X_{[1]}(X_{[1]}^{T}D_{[1]}X_{[1]})^{-1}X_{[1]}^{T})\|_{\infty}}{\|X_{1}^{T}(I - D_{[1]}X_{[1]}(X_{[1]}^{T}D_{[1]}X_{[1]})^{-1}X_{[1]}^{T})\|_{2}}.\]
Recall that $h_{j, 1, i}^{T}$ is $i$-th row of $I - D_{[1]}X_{[1]}(X_{[1]}^{T}D_{[1]}X_{[1]})^{-1}X_{[1]}^{T}$, we have
\[\max_{i}\frac{|h_{j, 1, i}^{T}X_{1}|}{\|h_{j, 1, i}\|_{2}}\approx \frac{\|e_{j}^{T}(X^{T}DX)^{-1}X^{T}\|_{\infty}}{\|e_{j}^{T}(X^{T}DX)^{-1}X^{T}\|_{2}}.\]
The right-handed side equals to $S_{j}(X)$ in the least-square case. Therefore, although of complicated form, assumption \textbf{A}5 is not an artifact of the proof but is essential for the asymptotic normality. 

\subsubsection{Additional Examples}\label{app:lseExamples}
Benefit from the analytical form of the least-square estimator, we can depart from sub-gaussinity of the entries. The following proposition shows that a random design matrix $Z$ with i.i.d. entries under appropriate moment conditions satisfies $\max_{j\in J_{n}}S_{j}(Z) = o(1)$ with high probability. This implies that, when $X$ is one realization of $Z$, the conditions Theorem \ref{thm:OLS_casn_kd} are satisfied for $X$ with high probability over $Z$.

\begin{proposition}\label{prop:OLS_casn_moment}
If $\{Z_{ij}: i\le n, j\in J_{n}\}$ are independent random variables with
\begin{enumerate}
  \item $\max_{i\le n, j\in J_{n}}(\E |Z_{ij}|^{8 + \delta})^{\frac{1}{8 + \delta}}\le M$ for some $\delta, M > 0$;
  \item $\min_{i\le n, j\in J_{n}}\Var(Z_{ij}) > \tau^{2}$ for some $\tau > 0$
  \item $P(Z\mbox{ has full column rank}) = 1 - o(1)$;
  \item $\E Z_{j} \in \spanvec\{Z_{j}: j\in J_{n}^{c}\}$ almost surely for all $j\in J_{n}$;
\end{enumerate}
where $Z_{j}$ is the $j$-th column of $Z$. Then 
\[\max_{j\in J_{n}}S_{j}(Z) = O_{p}\lb\frac{1}{n^{\frac{1}{4}}}\rb = o_{p}(1).\]
\end{proposition}
A typical practically interesting example is that $Z$ contains an intercept term, which is not in $J_{n}$, and $Z_{j}$ has i.i.d. entries for $j\in J_{n}$ with continuous distribution and sufficiently many moments, in which case the first three conditions are easily checked and $\E Z_{j}$ is a multiple of $(1, \ldots, 1)$, which belongs to $\spanvec\{Z_{j}: j\in J_{n}^{c}\}$. 

In fact, the condition 4 allows Proposition \ref{prop:OLS_casn_moment} to cover more general cases than the above one. For example, in a census study, a state-specific fix effect might be added into the model, i.e. 
\[y_{i} = \alpha_{s_{i}} + z_{i}^{T}\betanull + \eps_{i}\]
where $s_{i}$ represents the state of subject $i$. In this case, $Z$ contains a sub-block formed by $z_{i}$ and a sub-block with ANOVA forms as mentioned in Example 1. The latter is usually incorporated only for adjusting group bias and not the target of inference. Then condition 4 is satisfied if only $Z_{ij}$ has same mean in each group for each $j$, i.e. $\E Z_{ij} = \mu_{s_{i}, j}$. 

\begin{proof}[\textbf{Proof of Proposition \ref{prop:OLS_casn_moment}}]
By Sherman-Morison-Woodbury formula,
\[e_{j}^{T}(Z^{T}Z)^{-1}Z^{T} = \frac{Z_{j}^{T}(I - H_{j})}{Z_{j}^{T}(I - H_{j})Z_{j}}\]
where $H_{j} = Z_{[j]}(Z_{[j]}^{T}Z_{[j]})^{-1}Z_{[j]}^{T}$ is the projection matrix generated by $Z_{[j]}$. Then
\begin{equation}
  \label{eq:Hxj}
  S_{j}(Z) = \frac{\|e_{j}^{T}(Z^{T}Z)^{-1}Z^{T}\|_{\infty}}{\|e_{j}^{T}(Z^{T}Z)^{-1}Z^{T}\|_{2}} = \frac{\|Z_{j}^{T}(I - H_{j})\|_{\infty}}{\sqrt{Z_{j}^{T}(I - H_{j})Z_{j}}}.
\end{equation}
Similar to the proofs of other examples, the strategy is to show that the numerator, as a linear contrast of $Z_{j}$, and the denominator, as a quadratic form of $Z_{j}$, are both concentrated around their means. Specifically, we will show that there exists some constants $C_{1}$ and $C_{2}$ such that 
\begin{equation}\label{eq:moment_idempotent}
\max_{j\in J_{n}}\sup_{\substack{A\in\mathbb{R}^{n\times n}, A^{2} = A,\\ \tr(A) = n - p + 1}}\left\{P\lb\|AZ_{j}\|_{\infty} > C_{1}n^{\frac{1}{4}}\rb + P\lb Z_{j}^{T}AZ_{j} < C_{2}n\rb\right\} = o\lb\frac{1}{n}\rb.
\end{equation}
If \eqref{eq:moment_idempotent} holds, since $H_{j}$ is independent of $Z_{j}$ by assumptions, we have
\begin{align}
  &P\lb S_{j}(Z) \ge \frac{C_{1}}{\sqrt{C_{2}}}\cdot n^{-\frac{1}{4}}\rb = P\lb\frac{\|Z_{j}^{T}(I - H_{j})\|_{\infty}}{\sqrt{Z_{j}^{T}(I - H_{j})Z_{j}}} \ge \frac{C_{1}}{\sqrt{C_{2}}}\cdot n^{-\frac{1}{4}}\rb\nonumber\\
\le &P\lb \|(I - H_{j})Z_{j}\|_{\infty} > C_{1}n^{\frac{1}{4}}\rb + P\lb Z_{j}^{T}(I - H_{j})Z_{j} < C_{2}n\rb\nonumber\\
= & \E \left[P\lb \|(I - H_{j})Z_{j}\|_{\infty} > C_{1}n^{\frac{1}{4}}\rb \bigg| Z_{[j]}\right] + \E \left[P\lb Z_{j}^{T}(I - H_{j})Z_{j} < C_{2}n\rb\bigg| Z_{[j]}\right]\label{eq:condition}\\
\le & \sup_{A\in\mathbb{R}^{n\times n}, A^{2} = A, \tr(A) = n - p + 1}P\lb\|AZ_{j}\|_{\infty} > C_{1}n^{\frac{1}{4}}\rb + P\lb Z_{j}^{T}AZ_{j} < C_{2}n\rb\nonumber\\
\le & \max_{j\in J_{n}}\left\{\sup_{A\in\mathbb{R}^{n\times n}, A^{2} = A, \tr(A) = n - p + 1}P\lb\|AZ_{j}\|_{\infty} > C_{1}n^{\frac{1}{4}}\rb + P\lb Z_{j}^{T}AZ_{j} < C_{2}n\rb\right\} = o\lb\frac{1}{n}\rb.\label{eq:prob_order}
\end{align}
Thus with probability $1 - o(|J_{n}| / n) = 1 - o(1)$, 
\[\max_{j\in J_{n}}S_{j}(Z)\le \frac{C_{1}}{\sqrt{C_{2}}}\cdot n^{-\frac{1}{4}}\]
and hence
\[\max_{j\in J_{n}}S_{j}(Z) = O_{p}\lb \frac{1}{n^{\frac{1}{4}}}\rb.\]

~\\

\noindent Now we prove \eqref{eq:moment_idempotent}. The proof, although looks messy, is essentially the same as the proof for other examples. Instead of relying on the exponential concentration given by the sub-gaussianity, we show the concentration in terms of higher-order moments. 

~\\
\noindent In fact, for any idempotent $A$, the sum square of each row is bounded by 1 since
\[\sum_{i}A_{ij}^{2} = (A^{2})_{j, j} \le \lambda_{\mathrm{\max}}(A^{2}) =  1.\] 
By Jensen's inequality, 
\[\E Z_{ij}^{2}\le (\E |Z_{ij}|^{8 + \delta})^{\frac{2}{8 + \delta}}.\]
For any $j$, by Rosenthal's inequality \cite{rosenthal70}, there exists some universal constant $C$ such that 
\begin{align*}
\E \left|\sum_{i=1}^{n}A_{ij}Z_{ij}\right|^{8 + \delta}&\le C\left\{\sum_{i=1}^{n}|A_{ij}|^{8 + \delta}\E |Z_{ij}|^{8 + \delta} + \lb \sum_{i=1}^{n}A_{ij}^{2}\E Z_{ij}^{2}\rb^{4 + \delta / 2}\right\}\\
&\le C\left\{\sum_{i=1}^{n}|A_{ij}|^{2}\E |Z_{ij}|^{8 + \delta} + \lb \sum_{i=1}^{n}A_{ij}^{2}\E Z_{ij}^{2}\rb^{4 + \delta / 2}\right\}\\
& \le CM^{8 + \delta}\left\{\sum_{i=1}^{n}A_{ij}^{2} + \lb \sum_{i=1}^{n}A_{ij}^{2}\rb^{4 + \delta / 2}\right\}\le 2CM^{8 + \delta}.
\end{align*}
Let $C_{1}= (2CM^{8 + \delta})^{\frac{1}{8 + \delta}}$, then for given $i$, by Markov inequality, 
\[P\lb \bigg|\sum_{i=1}^{n}A_{ij}Z_{ij}\bigg| > C_{1}n^{\frac{1}{4}}\rb\le \frac{1}{n^{2 + \delta / 4}}\]
and a union bound implies that
\begin{equation}\label{eq:moment_numerator}
P\lb \|AZ_{j}\|_{\infty} > C_{1}n^{\frac{1}{4}}\rb\le \frac{1}{n^{1 + \delta / 4}} = o\lb\frac{1}{n}\rb.
\end{equation}
Now we derive a bound for $Z_{j}^{T}AZ_{j}$. Since $p/n \rightarrow \kappa\in (0, 1)$, there exists $\td{\kappa}\in (0, 1 - \kappa)$ such that  $n - p > \td{\kappa} n$. Then 
\begin{equation}\label{eq:expect_quad}
\E Z_{j}^{T}AZ_{j} = \sum_{i=1}^{n}A_{ii}\E Z_{ij}^{2} > \tau^{2}\tr(A) = \tau^{2}(n - p + 1) > \td{\kappa} \tau^{2} n.
\end{equation}
To bound the tail probability, we need the following result:
\begin{lemma}[\citeA{bai10}, Lemma 6.2]\label{lem:quad_markov}
  Let $B$ be an $n\times n$ nonrandom matrix and $W = (W_{1}, \ldots, W_{n})^{T}$ be a random vector of independent entries. Assume that $\E W_{i} = 0$, $\E W_{i}^{2} = 1$ and $\E |W_{i}|^{k}\le \nu_{k}$. Then, for any $q\ge 1$, 
\[E|W^{T}BW - \tr(B)|^{q} \le  C_{q}\lb(\nu_{4}\tr(BB^{T}))^{\frac{q}{2}} + \nu_{2q}\tr(BB^{T})^{\frac{q}{2}}\rb,\]
where $C_{q}$ is a constant depending on $q$ only.
\end{lemma}
It is easy to extend Lemma \ref{lem:quad_markov} to non-isotropic case by rescaling. In fact, denote $\sigma_{i}^{2}$ by the variance of $W_{i}$, and let $\Sigma = \diag(\sigma_{1}, \ldots, \sigma_{n})$, $Y = (W_{1} / \sigma_{1}, \ldots, W_{n} / \sigma_{n})$. Then 
\[W^{T}BW = Y^{T}\Sigma^{\frac{1}{2}}B\Sigma^{\frac{1}{2}}Y,\]
with $\Cov(Y) = I$. Let $\td{B} = \Sigma^{\frac{1}{2}}B\Sigma^{\frac{1}{2}}$, then
\[\td{B}\td{B}^{T} =  \Sigma^{\frac{1}{2}}B\Sigma B^{T}\Sigma^{\frac{1}{2}}\preceq \nu_{2}\Sigma^{\frac{1}{2}}BB^{T}\Sigma^{\frac{1}{2}}.\]
This entails that 
\[\tr(\td{B}\td{B}^{T})\le nu_{2}\tr(\Sigma^{\frac{1}{2}}BB^{T}\Sigma^{\frac{1}{2}}) = \nu_{2}\tr(\Sigma BB^{T})\le \nu_{2}^{2}\tr(BB^{T}).\]
On the other hand, 
\[\tr(\td{B}\td{B}^{T})^{\frac{q}{2}}\le n\lambda_{\mathrm{\max}}(\td{B}\td{B}^{T})^{\frac{q}{2}} = n\nu_{2}^{\frac{q}{2}}\lambda_{\mathrm{\max}}\lb\Sigma^{\frac{1}{2}}BB^{T}\Sigma^{\frac{1}{2}}\rb^{\frac{q}{2}}\le n\nu_{2}^{q}\lambda_{\mathrm{\max}}(BB^{T})^{\frac{q}{2}}.\]
Thus we obtain the following result
\begin{lemma}\label{lem:quad_markov_noniso}
Let $B$ be an $n\times n$ nonrandom matrix and $W = (W_{1}, \ldots, W_{n})^{T}$ be a random vector of independent mean-zero entries. Suppose $\E |W_{i}|^{k}\le \nu_{k}$, then for any $q\ge 1$, 
\[E|W^{T}BW -\E W^{T}BW|^{q} \le C_{q}\nu_{2}^{q}\lb(\nu_{4}\tr(BB^{T}))^{\frac{q}{2}} + \nu_{2q}\tr(BB^{T})^{\frac{q}{2}}\rb,\]
where $C_{q}$ is a constant depending on $q$ only.
\end{lemma}
Apply Lemma \ref{lem:quad_markov_noniso} with $W = Z_{j}$, $B = A$ and $q = 4 + \delta / 2$, we obtain that 
\[E|Z_{j}^{T}AZ_{j} -\E Z_{j}^{T}AZ_{j}|^{4 + \delta / 2}\le CM^{16 + 2\delta}\lb (\tr(AA^{T}))^{2 + \delta/4} + \tr(AA^{T})^{2 + \delta / 4}\rb\]
for some constant $C$. Since $A$ is idempotent, all eigenvalues of $A$ is either $1$ or $0$ and thus $AA^{T}\preceq I$. This implies that 
\[\tr(AA^{T})\le n, \quad \tr(AA^{T})^{2 + \delta / 4}\le n\]
and hence
\[E|Z_{j}^{T}AZ_{j} -\E Z_{j}^{T}AZ_{j}|^{4 + \delta / 2} \le 2CM^{16 + 2\delta}n^{2 + \delta / 4}\]
for some constant $C_{1}$, which only depends on $M$. By Markov inequality, 
\[P\lb |Z_{j}^{T}AZ_{j}  - \E Z_{j}^{T}AZ_{j}| \ge \frac{\td{\kappa}\tau^{2}n}{2} \rb\le 2CM^{16 + 2\delta}\lb\frac{2}{\td{\kappa}\tau^{2}}\rb^{4 + \delta / 2}\cdot \frac{1}{n^{2 + \delta / 4}}.\]
Combining with \eqref{eq:expect_quad}, we conclude that 
\begin{equation}\label{eq:moment_denominator}
P\lb Z_{j}^{T}AZ_{j} < C_{2}n\rb = O\lb\frac{1}{n^{2 + \delta / 4}}\rb = o\lb\frac{1}{n}\rb
\end{equation}
where $C_{2} = \frac{\td{\kappa}\tau^{2}}{2}$. Notice that both \eqref{eq:moment_numerator} and \eqref{eq:moment_denominator} do not depend on $j$ and $A$.  Therefore, (\ref{eq:moment_idempotent}) is proved and hence the Proposition.
\end{proof}

\section{Additional Numerical Experiments}\label{app:numerical}
In this section, we repeat the experiments in section \ref{sec:numerical} by using $L_{1}$ loss, i.e. $\rho(x) = |x|$. $L_{1}$-loss is not smooth and does not satisfy our technical conditions. The results are displayed below. It is seen that the performance is quite similar to that with the huber loss.

%\subsection{Asymptotic Normality of Single Coordinate}

\begin{figure}[H]
  \centering
  \includegraphics[width = 0.47\textwidth]{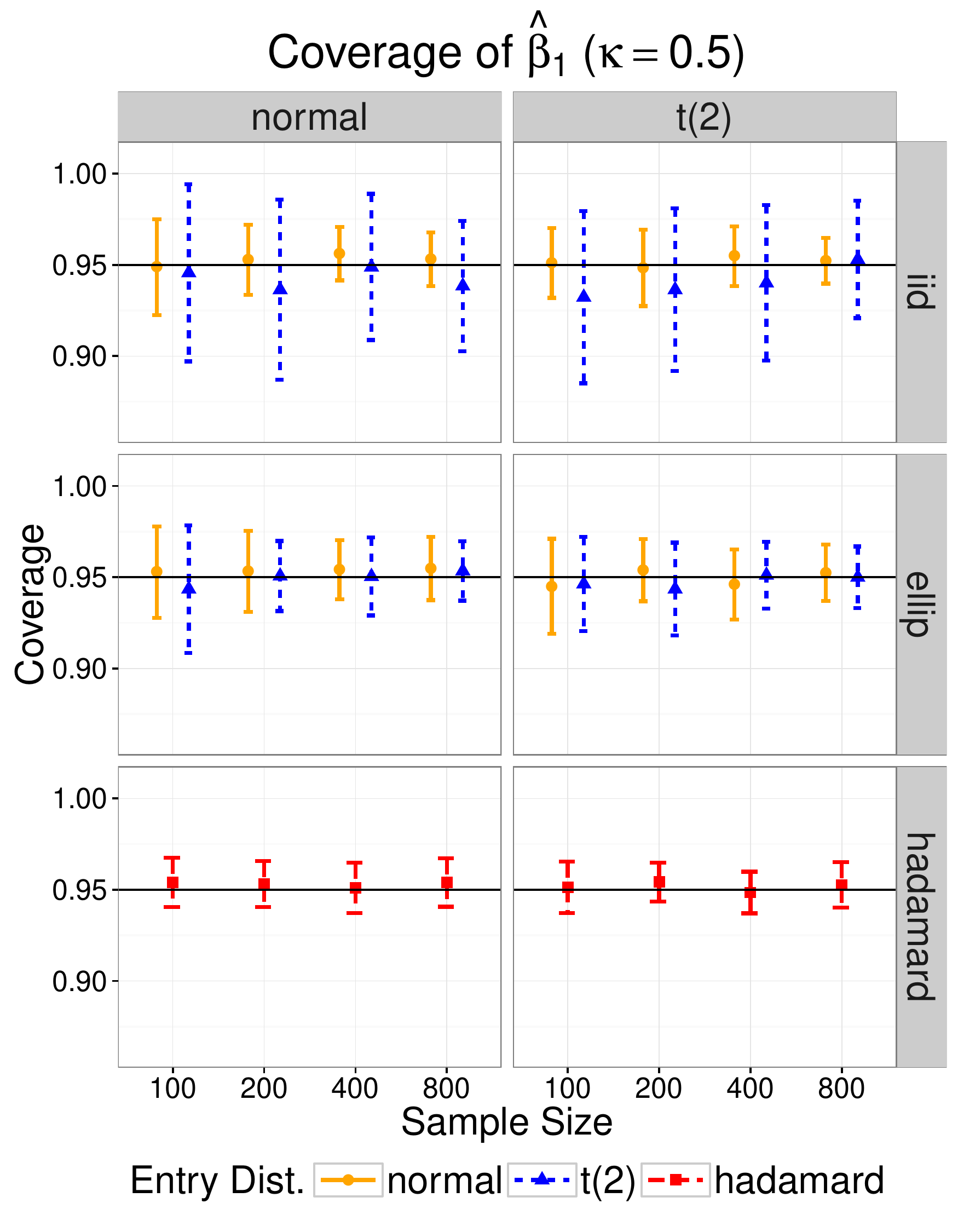}
  \includegraphics[width = 0.47\textwidth]{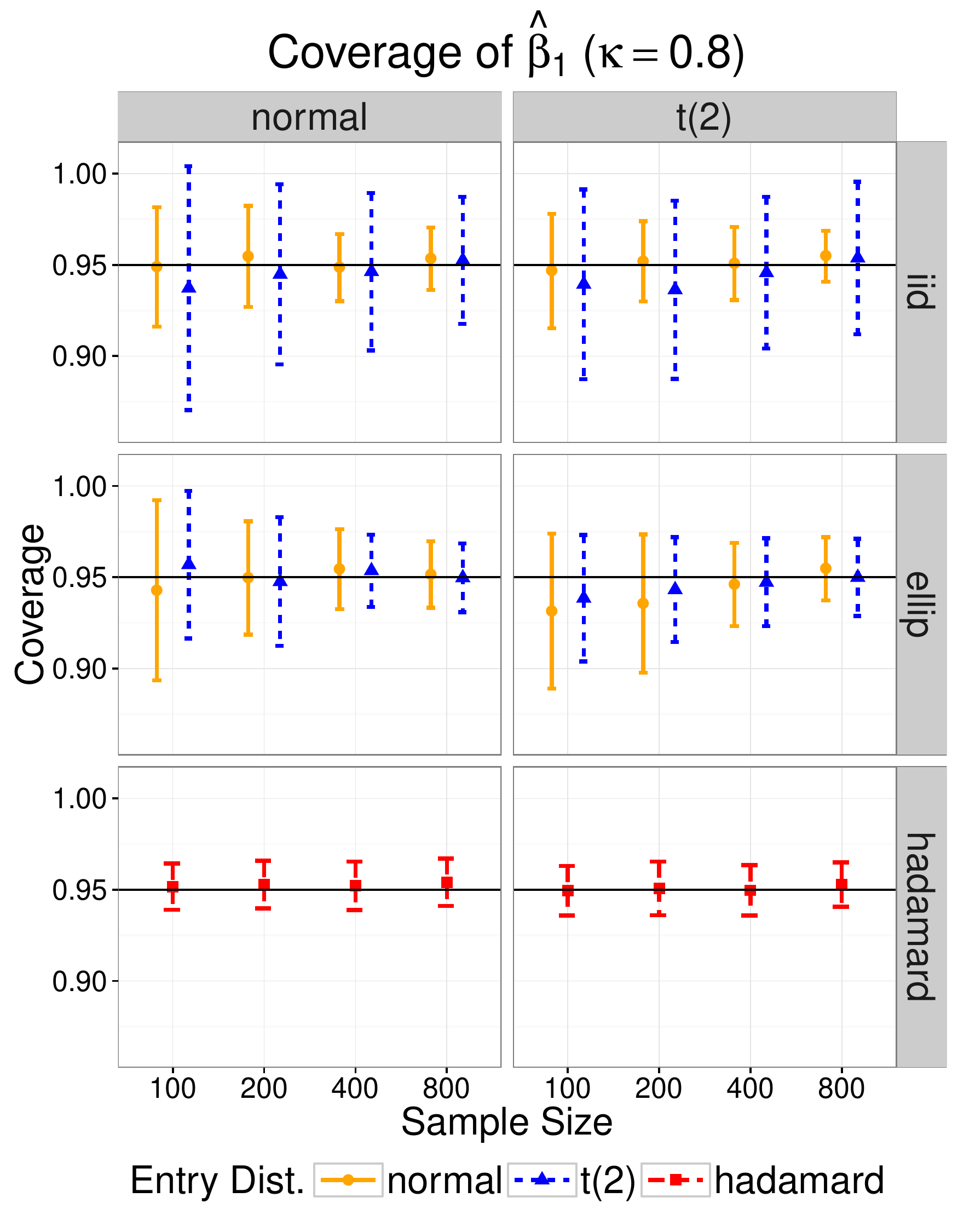}
  \caption{Empirical 95\% coverage of $\hat{\beta}_{1}$ with $\kappa = 0.5$ (left) and $\kappa = 0.8$ (right) using $L_1$ loss. The x-axis corresponds to the sample size, ranging from $100$ to $800$; the y-axis corresponds to the empirical 95\% coverage. Each column represents an error distribution and each row represents a type of design. The orange solid bar corresponds to the case $F = \text{Normal}$; the blue dotted bar corresponds to the case $F = \tdist_2$; the red dashed bar represents the Hadamard design. 
}\label{fig:singleCoord_MAD}

\end{figure}

% \begin{figure}[H]
%   \centering
%   \includegraphics[width = 0.47\textwidth]{MAD_med_qq.pdf}
%   \includegraphics[width = 0.47\textwidth]{MAD_high_qq.pdf}
%   \caption{QQ-plot of $\hat{\beta}_{1}$ with $\kappa = 0.5$ (left) and $\kappa = 0.9$ (right) by using $L_{1}$ loss. The sample size is 100 and entry distribution $F$ is $\tdist(2)$ in i.i.d. design and elliptical design. Each column represents an error distribution and each row represents a type of design. The x-axis corresponds to the quantile of standard normal distribution and the y-axis corresponds to the quantile of data. The red solid line is the $45^{\circ}$ line and the black dots represent $\hat{\beta}_{1}$ after normalizing and sorting.}\label{fig:qqplot}
% \end{figure}

%\subsection{Asymptotic Normality for Multiple Marginals}

% \begin{figure}[H]
%   \centering
%   \includegraphics[width = 0.47\textwidth]{MAD_med_size.pdf}
%   \includegraphics[width = 0.47\textwidth]{MAD_high_size.pdf}
%   \caption{Average proportion of p-values, produced by Shapiro-Wilks test, below 0.1 with $\kappa = 0.5$ (left) and $\kappa = 0.9$ (right) by using $L_{1}$ loss. The x-axis corresponds to the sample size, ranging from 100 to 1000 and the y-axis corresponds to the proportion. Each column represents an error distribution and each row represents a type of design. }\label{fig:size}
% \end{figure}

\begin{figure}[H]
  \centering
  \includegraphics[width = 0.47\textwidth]{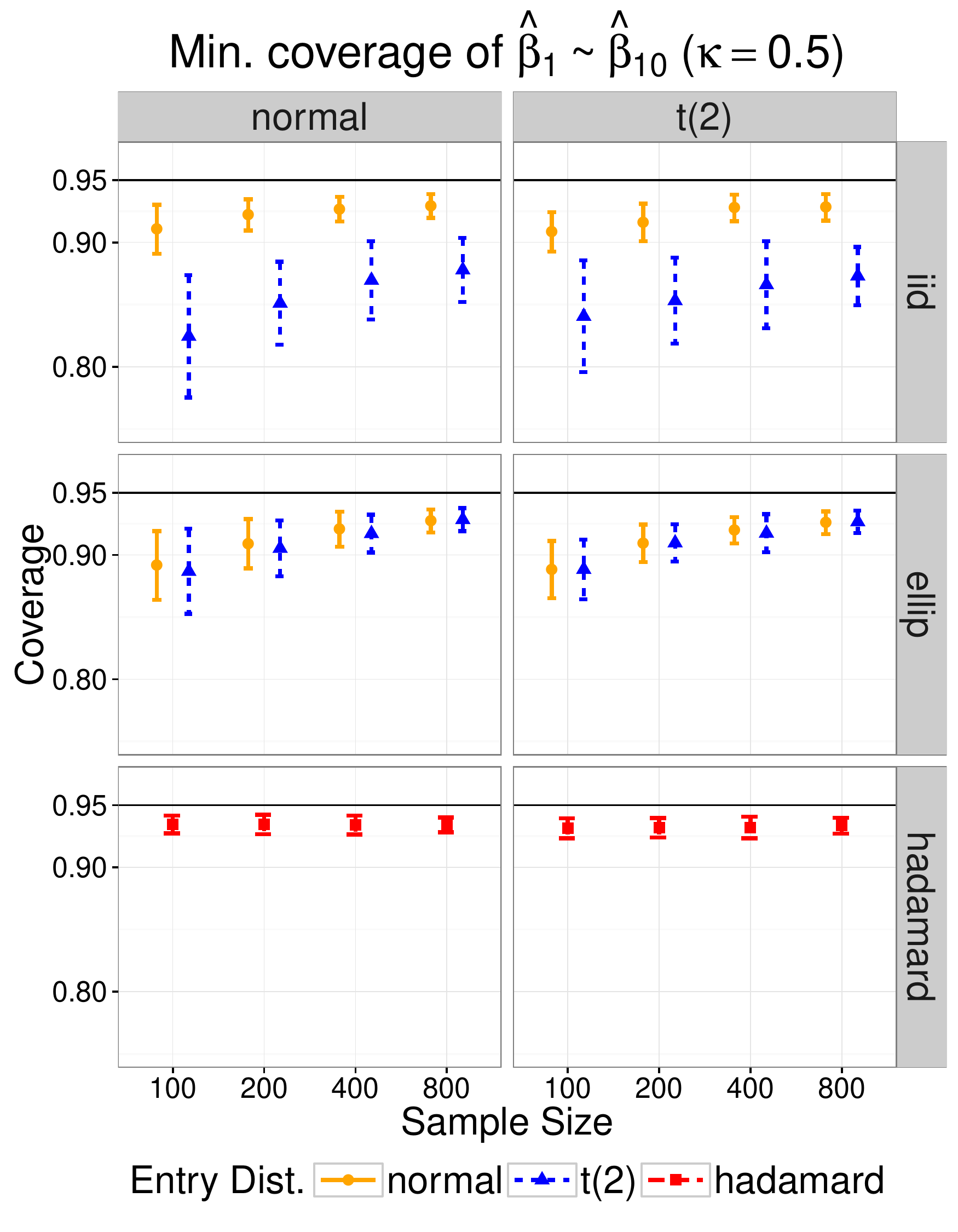}
  \includegraphics[width = 0.47\textwidth]{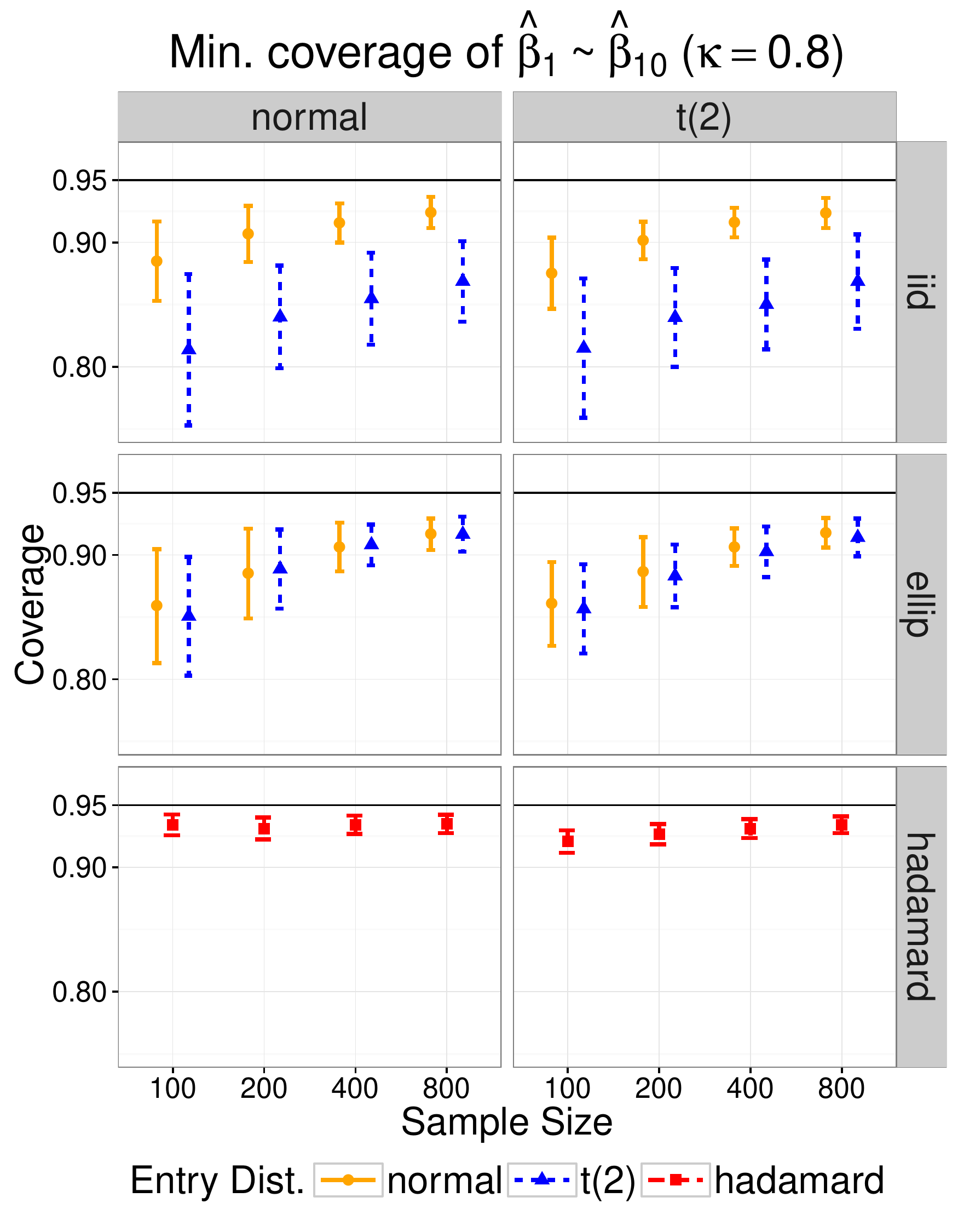}
  \caption{Mininum empirical 95\% coverage of $\hat{\beta}_{1}\sim\hat{\beta}_{10}$ with $\kappa = 0.5$ (left) and $\kappa = 0.8$ (right) using $L_1$ loss. The x-axis corresponds to the sample size, ranging from $100$ to $800$; the y-axis corresponds to the minimum empirical 95\% coverage. Each column represents an error distribution and each row represents a type of design. The orange solid bar corresponds to the case $F = \text{Normal}$; the blue dotted bar corresponds to the case $F = \tdist_2$; the red dashed bar represents the Hadamard design. 
}\label{fig:multipleCoord_MAD}
\end{figure}

\begin{figure}[H]
  \centering
  \includegraphics[width = 0.47\textwidth]{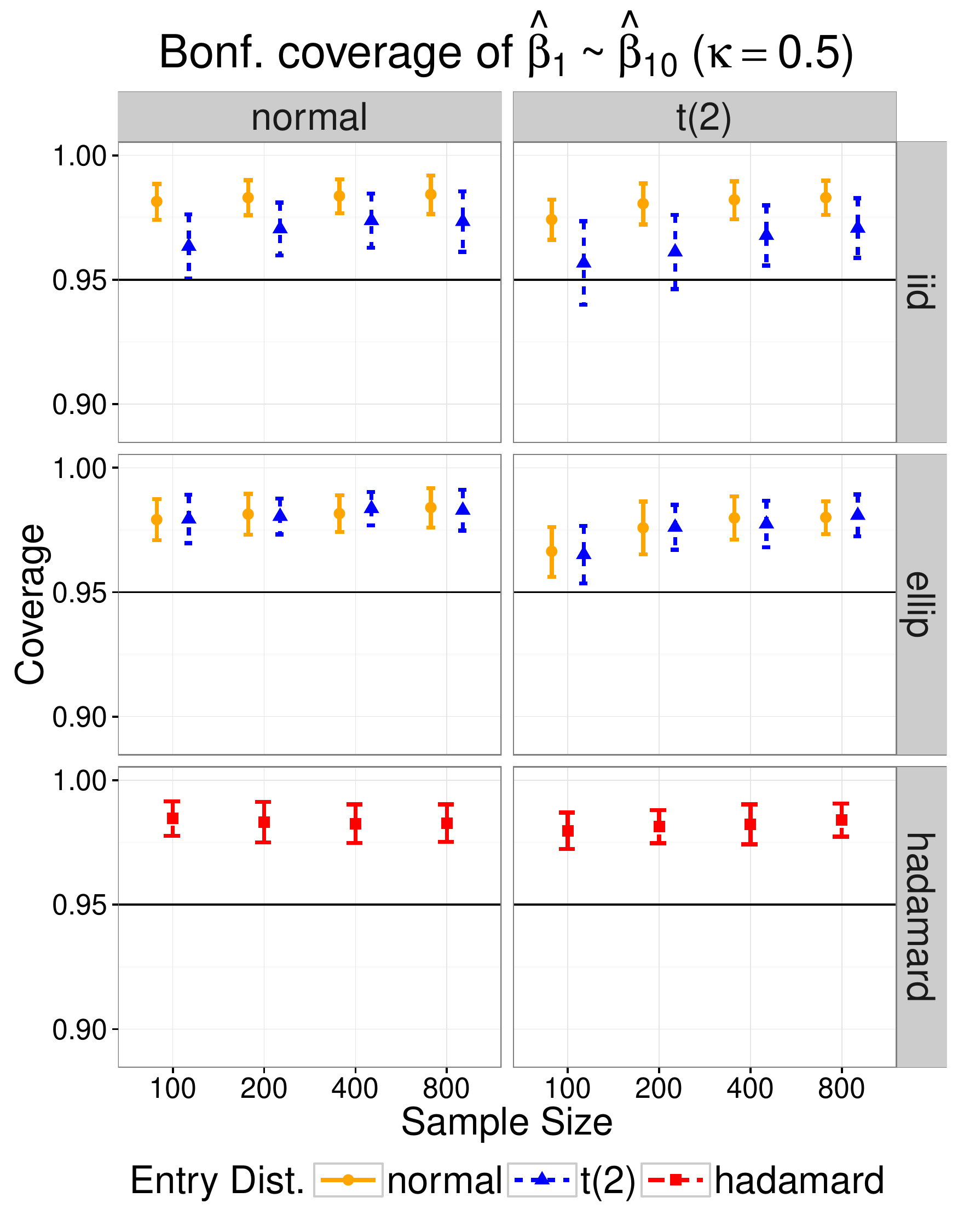}
  \includegraphics[width = 0.47\textwidth]{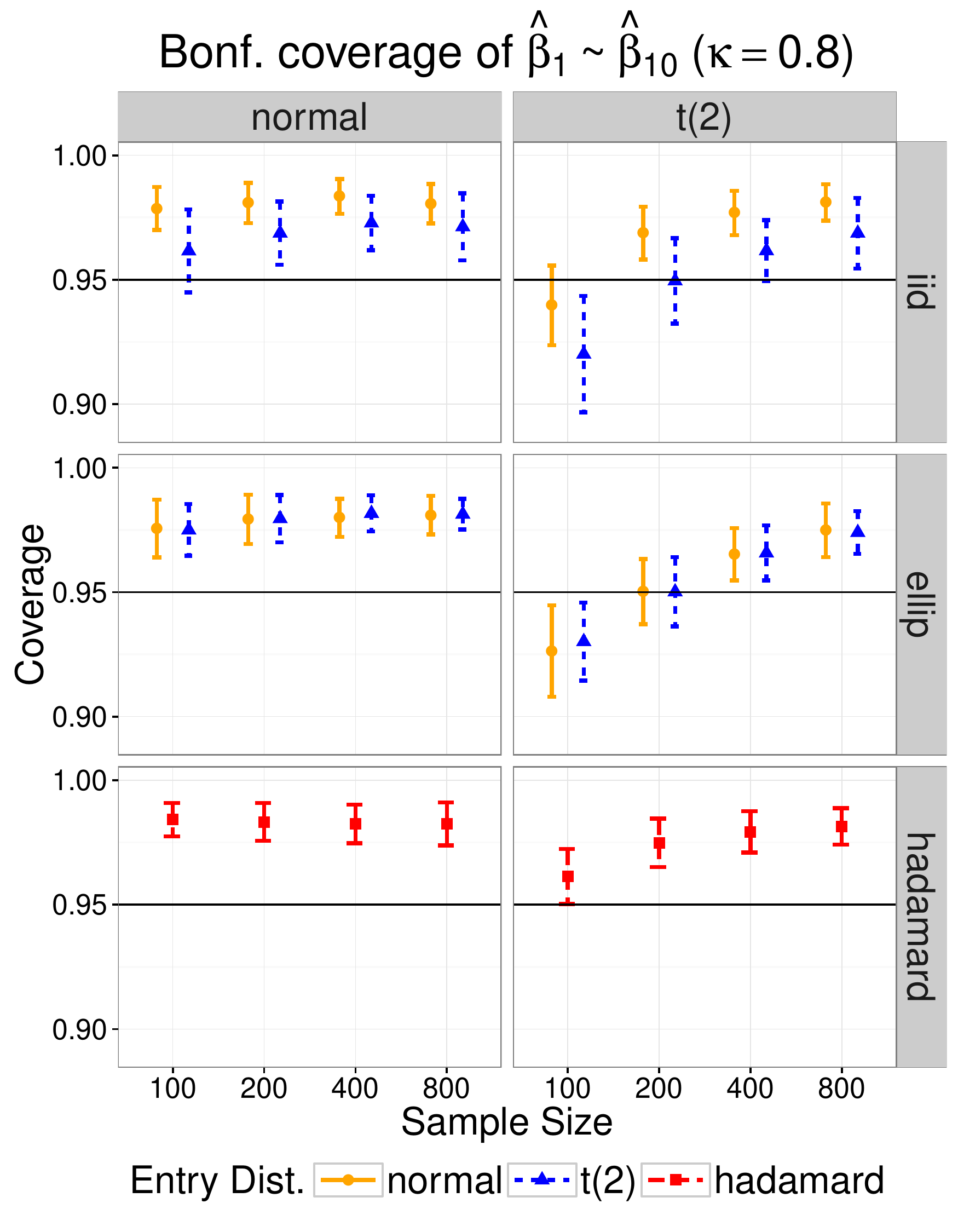}
  \caption{Empirical 95\% coverage of $\hat{\beta}_{1}\sim\hat{\beta}_{10}$ after Bonferroni correction with $\kappa = 0.5$ (left) and $\kappa = 0.8$ (right) using $L_1$ loss. The x-axis corresponds to the sample size, ranging from $100$ to $800$; the y-axis corresponds to the empirical uniform 95\% coverage after Bonferroni correction. Each column represents an error distribution and each row represents a type of design. The orange solid bar corresponds to the case $F = \text{Normal}$; the blue dotted bar corresponds to the case $F = \tdist_2$; the red dashed bar represents the Hadamard design. 
}\label{fig:multipleCoordBon_MAD}
\end{figure}

\newpage
\section{Miscellaneous}\label{app:mc}
In this appendix we state several technical results for the sake of completeness. 

\begin{proposition}[\citeNP{horn12}, formula (0.8.5.6)]\label{prop:block_inv}
Let $A\in \R^{p\times p}$ be an invertible matrix and write $A$ as a block matrix 
\[A = \lb
  \begin{array}{cc}
    A_{11} & A_{12}\\
    A_{21} & A_{22}
  \end{array}
\rb\]
with $A_{11}\in \R^{p_{1}\times p_{1}}, A_{22}\in \R^{(p - p_{1})\times (p -p_{1})}$ being invertible matrices. Then 
\[A^{-1} = \lb
  \begin{array}{cc}
    A_{11} + A_{11}^{-1}A_{12}S^{-1}A_{21}A_{11}^{-1} & -A_{11}^{-1}A_{12}S^{-1}\\
    -S^{-1}A_{21}A_{11}^{-1} & S^{-1}
  \end{array}
\rb\]
where $S = A_{22} - A_{21}A_{11}^{-1}A_{12}$ is the Schur's complement.
\end{proposition}

\begin{proposition}[\citeNP{rudelson13}; improved version of the original form by \citeNP{hanson71}]\label{prop:hanson_wright}
  Let $X = (X_{1}, \ldots, X_{n})\in \R^{n}$ be a random vector with independent mean-zero $\sigma^{2}$-sub-gaussian components $X_{i}$. Then, for every $t$,
  \[P\lb|X^{T}AX - \E X^{T}AX| > t\rb\le 2\exps{-c\min \lb \frac{t^{2}}{\sigma^{4}\|A\|_{F}^{2}}, \frac{t}{\sigma^{2}\|A\|_{\mathrm{op}}}\rb}\]
\end{proposition}

\begin{proposition}[\citeNP{bai93}]\label{prop:rmt_bai93}
  If $\{Z_{ij}: i = 1, \ldots, n, j = 1, \ldots, p\}$ are i.i.d. random variables with zero mean, unit variance and finite fourth moment and $p / n\rightarrow \kappa$, then 
\[\lambda_{\mathrm{\max}}\lb\frac{Z^{T}Z}{n}\rb\stackrel{a.s.}{\rightarrow} (1 + \sqrt{\kappa})^{2}, \quad \lambda_{\mathrm{\min}}\lb\frac{Z^{T}Z}{n}\rb\stackrel{a.s.}{\rightarrow} (1 - \sqrt{\kappa})^{2}.\]
\end{proposition}

\begin{proposition}[\citeNP{latala05}]\label{prop:rmt_latala05}
  Suppose $\{Z_{ij}: i = 1, \ldots, n, j = 1, \ldots, p\}$ are independent mean-zero random variables with finite fourth moment, then 
\[\E \sqrt{\lambda_{\mathrm{\max}}\lb Z^{T}Z\rb}\le C\lb\max_{i}\sqrt{\sum_{j}\E Z_{ij}^{2}} + \max_{j}\sqrt{\sum_{i}\E Z_{ij}^{2}} + \sqrt[4]{\sum_{i, j}\E Z_{ij}^{4}}\rb\]
for some universal constant $C$. In particular, if $\E Z_{ij}^{4}$ are uniformly bounded, then 
\[\lambda_{\mathrm{\max}}\lb\frac{Z^{T}Z}{n}\rb = O_{p}\lb 1 + \sqrt{\frac{p}{n}}\rb.\]
\end{proposition}

\begin{proposition}[\citeNP{rudelson10}]\label{prop:rmt_rudelson10}
Suppose $\{Z_{ij}: i = 1, \ldots, n, j = 1, \ldots, p\}$ are independent mean-zero $\sigma^{2}$-sub-gaussian random variables. Then there exists a universal constant $C_{1}, C_{2} > 0$ such that 
\[P\lb \sqrt{\lambda_{\mathrm{\max}}\lb\frac{Z^{T}Z}{n}\rb} > C\sigma\lb 1 + \sqrt{\frac{p}{n}} + t\rb\rb\le 2e^{-C_{2}nt^{2}}.\]
\end{proposition}

\begin{proposition}[\citeNP{rudelson09}]\label{prop:rmt_rudelson09}
  Suppose $\{Z_{ij}: i = 1, \ldots, n, j = 1, \ldots, p\}$ are i.i.d. $\sigma^{2}$-sub-gaussian random variables with zero mean and unit variance, then for $\eps \ge 0$
\[P\lb \sqrt{\lambda_{\mathrm{\min}}\lb\frac{Z^{T}Z}{n}\rb}\le \eps(1 - \sqrt{\frac{p - 1}{n}})\rb\le (C\eps)^{n - p + 1} + e^{-cn}\]
for some universal constants $C$ and $c$.
\end{proposition}

\begin{proposition}[\citeNP{litvak05}]\label{prop:rmt_litvak05}
  Suppose $\{Z_{ij}: i = 1, \ldots, n, j = 1, \ldots, p\}$ are independent $\sigma^{2}$-sub-gaussian random variables such that 
\[Z_{ij}\stackrel{d}{=}-Z_{ij}, \quad \Var(Z_{ij}) > \tau^{2}\]
for some $\sigma, \tau > 0$, and $p / n\rightarrow \kappa \in (0, 1)$, then there exists constants $c_{1}, c_{2} > 0$, which only depends on $\sigma$ and $\tau$,  such that 
\[P\lb \lambda_{\mathrm{\min}}\lb\frac{Z^{T}Z}{n}\rb < c_{1}\rb\le e^{-c_{2}n}.\]
\end{proposition}

%%% Local Variables:
%%% mode: latex
%%% TeX-master: t
%%% End:

\end{document}